\documentclass[mathpazo]{cicp}

\usepackage{lineno,hyperref}
\usepackage{fullpage}
\usepackage{amsmath,amsfonts,amsmath,amssymb,amsthm,cases}
\usepackage{amsfonts}
\usepackage{graphics}
\usepackage{caption2}
\usepackage{amssymb}
\usepackage{psfrag}
\usepackage{amsmath}
\usepackage{graphicx}
\usepackage{multirow}
\usepackage{subfigure}
\usepackage{wasysym}
\usepackage{color}
\usepackage[dvipsnames]{xcolor}

\usepackage{algpseudocode}
\usepackage{algorithmicx}
\usepackage[section]{algorithm}
\usepackage{amsthm}
\usepackage{ulem}

\makeatletter
\newcommand{\algmargin}{\the\ALG@thistlm}
\makeatother
\newlength{\whilewidth}
\algdef{SE}[parWHILE]{parWhile}{EndparWhile}[1]
{\parbox[t]{\dimexpr\linewidth-\algmargin}{%
		\hangindent\whilewidth\strut\algorithmicwhile\ #1\ \algorithmicdo\strut}}{\algorithmicend\ \algorithmicwhile}%
\algnewcommand{\parState}[1]{\State%
	\parbox[t]{\dimexpr\linewidth-\algmargin}{\strut #1\strut}}

\newtheorem{thm}{Theorem}[section]
\newtheorem{lemmaa}[thm]{\textbf{Lemma}}
\newtheorem{rulee}[thm]{\textbf{Rule}}

\newtheorem{rem}{Remark}

\newtheorem{algorithm_}{Algorithm}[section]

\newcommand{\mC}[1]{\multicolumn{1}{|c|}{#1}}

\newcommand{\mR}[1]{\multirow{2}{*}{#1}}

\def\p{\partial}
\def\DD{\displaystyle}
\def\R{\mathbb{R}}
\def\ii{\mathbf{i}}

\def\l{\,l}
\def\bx{{\bf x}}

\def\f{\mbox{field}}
\def\a{\alpha}

\def\B{B}
\def\d{d}                       
\def\P{\mathcal{P}}

\def\L{{\mathcal{L}}}

\def\Nbx{N_x}
\def\Nby{N_y}
\def\Nbz{N_z}
\def\MaxN2{\max\{\Nbx, \Nby\}}
\def\MaxN3{\max\{\Nbx, \Nby, \Nbz\}}

\def\f32{\frac{3}{2}}

\def\s1{{s-1}}
\def\s2{{s-2}}
\def\d{\boldsymbol{d}}

\def\dirThreeD{\Box, \vartriangle, \ocircle}

\def\rangeTwoDRow{i = 1,\ldots,\Nbx \\ j = 1,\ldots,\Nby}

\def\rangeThreeDRow{i = 1,\ldots,\Nbx \\ j = 1,\ldots,\Nby \\ k = 1,\ldots,\Nbz}
\def\rangeThreeDCol{i = 1,\ldots,\Nbx,\; j = 1,\ldots,\Nby,\; k = 1,\ldots,\Nbz}

\DeclareMathOperator\supp{supp}

\def\Tave{{T}_{\text{ave}}}

\begin{document}


\title{A  Diagonal Sweeping Domain Decomposition Method with Source Transfer for the Helmholtz Equation}

\author[Wei Leng et.~al.]{Wei Leng \affil{1},
	Lili Ju\affil{2}\comma\corrauth}
\address{\affilnum{1}\ State Key
	Laboratory of Scientific and Engineering Computing, Chinese Academy
	of Sciences, Beijing 100190, China.\\
	\affilnum{2}\ Department of Mathematics, University of South Carolina, Columbia, SC 29208, USA.
}
\emails{{\tt wleng@lsec.cc.ac.cn} (W.~Leng), 
	{\tt ju@math.sc.edu} (L.~Ju),
}

\begin{abstract}
  In this paper, we propose and test a novel diagonal sweeping domain decomposition method (DDM) with source transfer for solving the high-frequency Helmholtz equation in $\R^n$.  In the method the computational domain is partitioned into overlapping checkerboard subdomains for source transfer with the perfectly matched layer (PML) technique, then a set of diagonal sweeps over the subdomains are specially designed to solve the system efficiently.
The method improves  the additive overlapping DDM \cite{Leng2019} and the L-sweeps method \cite{Zepeda2019} by
  employing a more efficient subdomain solving order.
  We show that the method achieves the exact solution of the global PML problem with $2^n$ sweeps in the constant medium case.  Although the sweeping usually implies sequential subdomain solves, the number of sequential steps required for each sweep in the method is only proportional to the $n$-th root of the number of subdomains when the domain decomposition is quasi-uniform with respect to all directions, thus  it is very suitable for parallel computing of
 the Helmholtz problem with multiple right-hand sides through the pipeline processing.  Extensive numerical experiments in two and three dimensions are presented to demonstrate the effectiveness and efficiency of the proposed method.

\end{abstract}



\maketitle

\newif\ifdraftFig
\newif\ifSecTwo
\newif\ifSecTwoTwoD
\newif\ifSecTwoThreeD
\newif\ifSecThree
\newif\ifSecFour

\draftFigtrue

\SecTwotrue
\SecTwoTwoDtrue
\SecTwoThreeDtrue
\SecThreetrue
\SecFourtrue

\section{Introduction}

In this paper,  we consider  the well-known Helmholtz equation defined in $\R^n$ ($n=2,3$) as follows:
\begin{align} \label{eq:helm}
  \Delta u + \kappa^2 u &= f ,\qquad  \mbox{in} \;\;\; \R^n,
\end{align}
imposed with the Sommerfeld radiation condition
\begin{align} \label{cond_s}
  r^{\frac{n-1}{2}}\Big(\frac{\p u}{\p r} - \mathbf{i} \kappa u\Big) &\rightarrow 0, \qquad \mbox{as} \;\;\; r = |\bx| \rightarrow \infty, 
\end{align}
where $u(\bx)$ is the unknown function, $f(\bx)$ is the source and
$\kappa(\bx) := {\omega}/{c(\bx)} $ denotes the wave number with $\omega$ being the angular frequency and $c(\bx)$ the wave speed.
Solving the Helmholtz equation \eqref{eq:helm} with large wave number accurately and efficiently is crucial to many physics and engineering problems. For example, in exploration seismology, the Helmholtz equation  with  pre-given wave speed  needs to be solved for hundreds of different sources  in reverse time migration, and even more in full wave inversion.  However, since the discrete Helmholtz system with large wave number is highly indefinite, constructing efficient solvers is quite important and challenging \cite{Gander2012}, and for this purpose many methods have been proposed and studied, including the direct method \cite{Frontal1}, the multigrid method \cite{Erlangga2006} and the domain decomposition method \cite{Despres1990}.


The direct method, such as the multifrontal method \cite{Frontal1} with nested dissection \cite{George1973}, was designed to solve linear systems arising from discretization of general PDE problems, and has  been employed to solve the discrete Helmholtz problem.
The multifrontal method was  further coupled with the hierarchically semi-separable matrices (HSS) in \cite{HSS}, and the low rank proprieties were exploited to  reduce the computational complexity for many problems including Helmholtz equation in \cite{Xia2010,Wang2016}.  However, the low-rank representation for the Helmholtz kernel in high frequency is missing \cite{Engquist2016}, which causes the HSS and multifrontal coupled method to be less effective for high frequency problems.
On the other hand, some variants of the multifrontal method were  also introduced in \cite{Gillman2015, Leng2015} for the Helmholtz problem. Those methods mostly focus on constructing the Dirichlet to Neumann (DtN) map for the subdomains in the nested dissection, which is more intuitive than manipulating the algebraic matrices in the multifrontal method, while the order of computational complexity remains the same.

The multigrid method with the shifted Laplace was first introduced in \cite{Erlangga2006}, and then further developed in \cite{Erlangga2005, Erlangga2006b,Erlangga2008,Umetani2009,Sheikh2013,Aruliah2002}.  A complex shift is added to the Helmholtz operator, resulting in an easier problem that could be solved with multigrid solver, which  then can be used as an effective preconditioner for the original Helmholtz problem. The shifted Laplace method has been shown to be very effective, and followed by many researches in literature, to name a few, \cite{Airaksinen2007,Calandra2013,Cools2014,Cools2015,Tsuji2015,Hu2016,Stolk2016,Reps2017,Gander2015}.  The amount of the shift is a compromise, a larger shift leads to an easier problem to solve in preconditioning but more iteration steps in the Krylov subspace solve, while a smaller shift results in harder preconditioning but fewer iteration steps.  For the high frequency problem, if the shifted problem in preconditioning is required to be solved efficiently, then the number of iterations in the Krylov subspace solve grows as fast as the square of the frequency\cite{Gander2017}, thus the high frequency problem is still a big challenge for the shifted Laplace method.

The domain decomposition method (DDM) for solving the Helmholtz problem was first studied  in \cite{Despres1990}.  A good approximation of the Dirichlet to Neumann (DtN) map is the key to maintain the effectiveness of DDM for the Helmholtz equation, and later various transmission conditions are proposed to approximate the DtN map, leading to different DDMs as in \cite{Collino2000, SchwarzBdry1, SchwarzBdry2, SchwarzBdry3, SchwarzBdry4, SchwarzBdry5, SchwarzBdry6, Douglas1998,Boubendir2012, Schadle2007, Toselli1999}.  However, the additive nature of these DDMs cause the number of iterations grows as fast as the $n$-th root of the number of subdomains in the checkerboard partition case.

The first sweeping type DDM for the Helmholtz problem was introduced by Engquist and Ying in \cite{Engquist2011a, Engquist2011b}, and followed by many variants, such as the single layer potential \cite{Stolk2013}, the double sweeping preconditioner \cite{Vion2014}, the polarized trace method \cite{Zepeda2014}, and the source transfer DDM (STDDM) \cite{Chen2013a,Chen2013b}. These methods employ  the perfectly matched layer (PML)  boundary condition on each subdomain and mainly differ at the transmission conditions between subdomains, and they all can be uniformly formulated in the context of optimized Schwarz method \cite{GanderReview}.  These DDMs usually decomposes the domain into layers, and sweep forwards and backwards in the layers to obtain good approximations of the solution. 
The sweeping type DDMs could be interpreted as  $LU$ or $LDL^T$ factorizations  and  forward/backward substitutions, and they generally have two phases, the factorization phase and the sweeping phase.
In the factorization phase, the local discrete systems of subdomains are factorized, which could be done in parallel. 
In the sweeping phase, the local 
solves of subdomain problems are applied one by one to form the global solution, which is a sequential process.    
The factorization phase is the bottleneck for the sweeping type DDMs for the Helmholtz problem in $\R^3$, since the factorization of each 2D layered subdomain  requires a scalable and efficient direct solver, which is often hard to accomplish as mentioned previously. On the other hand, although the sweeping phase is sequential, it could be arranged in a pipeline for parallel processing in the case of multiple right-hand sides (RHSs), which is quite common in many practical applications such as seismic imaging and electromagnetic scattering.    

Some recursive sweeping DDMs were proposed and studied in \cite{Liu2015b} and \cite{Wu2015}, which are based on the sweeping preconditioner and the source transfer DDM, respectively. In these methods, each of the layered subdomains is further decomposed into smaller layers in the perpendicular direction, and again solved with the sweeping DDM.  In such a way, the bottleneck caused by the factorization of subdomains no longer exists. However, the difficulty is  then shifted to the sweeping phase.  The number of steps used for each of the sequential subdomain sweeps is now proportional to the number of subdomains, that causes these methods not suitable for parallel computing in practice, for example,  when solving the multiple RHSs problem with the pipeline processing, the construction of an efficient pipeline will require a large number of RHSs in order to achieve good parallel efficiency.

The success of using source or trace transfer in sweeping DDMs with layered partitions inspires the development of the additive overlapping DDM for the Helmholtz equation in \cite{Leng2019}, which is  based on structured subdomains along all spatial directions  (i.e., checkerboard domain decomposition) in the context of the source transfer method. It is proved that this method could produce the exact solution in finite steps for the constant medium problem. 
The corner  transfer is considered for the first time in this method.  It is observed that for the case that the source lies only in one subdomain, the exact global solution can be constructed  with the subdomain solution marching along four diagonal directions in $\R^2$.
{Recently,  a sweeping-type DDM method called ``L-sweeps" was proposed in \cite{Zepeda2019}, which is also based on the corner  transfer. The  L-sweeps method wisely utilizes the property of diagonal subdomain solution marching,
  and employs a novel subdomain solving order of sweeps of all directions, which is the main difference between the additive overlapping DDM \cite{Leng2019} and the ``L-sweeps" method.}
The L-sweeps method produces an outstanding algorithm with $O(N \log N)$ complexity where $N$ denotes the number of unknowns of the discrete system.
Furthermore, the number of steps required by each sequential subdomain sweep in the L-sweeps method is only proportional to the $n$-th root of the number of subdomains, 
thus  this method is much more suitable for parallel computing compared to the recursive sweeping methods. When solving the multiple RHSs problem   using the L-sweeps method with pipeline, the requirement on the number of RHSs to achieve good parallel efficiency is feasible and could be easily satisfied in practical applications.

In this paper, we propose a novel diagonal sweeping DDM  for solving the Helmholtz equation \eqref{eq:helm} based on checkerboard domain decomposition. {Our method adopts a new subdomain solving order, which partly originates from the L-sweeps method \cite{Zepeda2019} but is more efficient.} Compared to the L-sweeps method, the proposed method has two major advantages in terms of efficiency and effectiveness:
\begin{itemize}
  \item The needed sweeps in each preconditioning solve are reduced from  L-sweeps of  $3^n-1$ directions (8 in $\R^2$ and 26 in $\R^3$ respectively) to diagonal sweeps of $2^n$ directions (4 in $\R^2$ and 8 in $\R^3$ respectively).
  \item The reflections are treated more appropriately for the layered media problems, increasing from one reflection to averagely two reflections per preconditioning solve.
\end{itemize}
The rest of the paper is organized as follows.  We first review the PML method associated with the Helmholtz equation and the corresponding additive overlapping DDM  \cite{Leng2019} with source transfer  in $\R^2$ and $\R^3$ in Section 2.  By wisely re-arranging the solving order of the additive DDM,   the diagonal sweeping DDM with source transfer in $\R^2$  is proposed and analyzed in Section3 and its  extension to $\R^3$ in Section 4. In addition,  we show that  the  DDM solutions are the exact solutions of the corresponding PML problems in $\R^2$ and $\R^3$ in the constant medium case. In Section 5, various numerical experiments in two and three dimensions are performed to verify convergence of the  diagonal sweeping DDM  for constant medium problems, and to test efficiency and effectiveness of the method as the preconditioner for layered media and even more complicated problems.
Some concluding remarks are finally drawn in Section 6.

\ifSecTwo
\section{Perfectly matched layer and  additive overlapping DDM with source transfer}

In this section, we  first recall the perfectly matched layer  method and the source transfer technique, and then review the additive overlapping DDM with source transfer proposed in \cite{Leng2019}, which is the basis of the diagonal sweeping DDM proposed in this paper. 

\subsection{Perfectly matched layer  and source transfer}

The Helmholtz equation \eqref{eq:helm} defined in the whole space  with the Sommerfeld radiation condition  \eqref{cond_s}  can be solved in a bounded domain such as a  rectangular box using the so-called uniaxial
PML method \cite{Berenger, Chew1994, Kim2010, Bramble2013, Chen2013a}, provided that the source lies inside the box. 
Suppose that a  rectangular box in $\R^2$ is defined as
$\B = \{(x_1,x_2)\;|\;a_j \le x_j \le b_j, j = 1,2\}$, 
with the center of the box  denoted by $(c_1, c_2)$ where $c_j = \frac{a_j+b_j}{2}$, for $j = 1,2$.
Let $\alpha_1(x_1) = 1 + \ii \sigma_1(x_1) $  and $\alpha_2(x_2) = 1 + \ii \sigma_2(x_2) $,
with $\{\sigma_j\}_{j=1}^2$ being  piecewise smooth functions such that
\begin{equation} \label{eq:sigma}
\sigma_j(\bx) = \left\{  
\begin{array}{ll}
\widehat{\sigma}(x_j - b_j), & \text{if} \,\,\, b_j \le x_j,\\
0,                         & \text{if} \,\,\, a_j < x_j < b_j,\\
\widehat{\sigma}(a_j - x_j), & \text{if} \,\,\, x_j \le a_j,\\
\end{array}
\right. 
\end{equation}
where $\widehat{\sigma}(t)$ is certain smooth medium profile function, then
 the complex coordinate stretching $\tilde{\bx}(\bx)$ for $\bx = (x_1, x_2)$ is defined as
\begin{equation}
\tilde{x}_j(x_j) = c_j+\int_{c_j}^{x_j} \alpha_j(t)\, dt = x_j
+ \ii \int_{c_j}^{x_j} \sigma_j(t) \,dt,\qquad j = 1,2.
\end{equation}

The PML equation is then defined under the  complex coordinate stretching as follows:
\begin{equation} \label{eq:PML}
J_{\B}^{-1} \nabla \cdot (A_{\B} \nabla\tilde{u}) + \kappa^2 \tilde{u} = f, 
\end{equation}
where 
$$\DD A_{\B}(\bx) = \mbox{diag}\left(\frac{\a_2(x_2)}{\a_1(x_1)}, \frac{\a_1(x_1)}{\a_2(x_2)}\right),\quad  J_{\B}(\bx) = \a_1(x_1) \a_2(x_2),$$
and $\tilde{u}$ is called the PML solution.   
The well-posedness of the weak problem associated with equation \eqref{eq:PML} has been established  in  \cite[Lemma 3.3]{Chen2013a}, and the PML solution $\tilde{u}$ equals $u$ within the box  and decays exponentially outside of the box.
For convenience, we denote by $\P_{\B}$ the PML problem \eqref{eq:PML} associated with the rectangular box $\B$, and denote by $\L_{\B}:=J_{\B}^{-1} \nabla \cdot (A_{\B} \nabla \, \boldsymbol{\cdot}) + \kappa^2 $ the linear operator associated with $\P_B$.

Similarly in $\R^3$, the PML equation for the cuboidal box 
$\B = \{(x_1,x_2,x_3)\;|\; a_j \le x_j \le b_j, j = 1,2,3\}$ could be defined as \eqref{eq:PML}
for $\bx=(x_1,x_2,x_3)$ with 
$$\DD A_{\B}(\bx) = \mbox{diag}\left(\frac{\a_2(x_2)\a_3(x_3)}{\a_1(x_1)}\right. , \frac{\a_1(x_1)\a_3(x_3)}{\a_2(x_2)}, \left. \frac{\a_1(x_1)\a_2(x_2)}{\a_3(x_3)}\right),\quad J_{\B}(\bx) = \a_1(x_1) \a_2(x_2)\a_3(x_3),$$ where $\alpha_3(x_3) = 1 + \ii \sigma_3(x_3) $ 
and $\sigma_3(x_3)$ is defined  in the same way as \eqref{eq:sigma}.


From now on, the constant medium (i.e., the constant wave number $\kappa(\bx)\equiv \kappa$) is assumed for development and analysis of the DDM methods.
The source transfer technique is presented in the following. The case of $\R^2$  is used for illustration and the results can be similarly extended to
the case of $\R^3$.  Suppose that a piecewise smooth curve $\gamma$ divides $\R^2$ into two parts $\Omega_1$ and $\Omega_2$, and at the meantime, the curve also divides the rectangular box $\B$ into two parts.
Let $\widetilde{\Omega}_1$ be the extended domain of $\Omega_1$ by a distance of $d$, for instance, $\widetilde{\Omega}_1 = \{\bx : \rho(\bx, \Omega_1) \le d\}$ where $d>0$ is a positive constant and denote $\widetilde{\gamma} = \p \widetilde{\Omega}_1$, as shown in Figure \ref{fig:twopart}-(a). There always exists a smooth cutoff function $\beta \in C^2(\R^n)$ with $0 \leq \beta \leq 1$ such that
$$\beta|_{\Omega_1} \equiv 1, \quad \beta|_{\R^n \setminus \widetilde{\Omega}_1 } \equiv 0, $$ 
and 
$$ |\nabla \beta(\bx) | < C, \qquad \forall\,\bx\in \widetilde{\Omega}_1 \setminus\Omega_1,$$
where $C$ is a generic positive constant. Then we have the following result on source transfer  \cite{Leng2019}:

\begin{lemmaa} \label{lemma:src_trans}
	Suppose that the support of $f$ is in $\Omega_1 \cap \B $.
	Let $u$ be the solution to the PML problem $\P_B$ with  the source $f$ (i.e, $\L_{\B} u = f$ in $\R^2$).
	Given $u_1$ as the restriction of $u$ on $\widetilde{\Omega}_1$, such as $u_1 := u {\chi}_{\widetilde{\Omega}_1} $, 
	and let $u_2$ be solution to 
	the PML problem $\P_B$ with  the source $-\L_{\B} (u_1 \beta)\chi_{\Omega_2}$  (i.e.,
	$	\L_{\B} u_2 = -\L_{\B} (u_1 \beta)\chi_{\Omega_2}$ in $\R^2$). Then it holds  that $u_1\beta + u_2 = u$ in $\R^2$ and  $u_2 = 0$ in $\Omega_1$.
\end{lemmaa}

The above Lemma  is straightforward based on the fact $u_1 \beta$ is the partial modification of $u$ and $u_2$ is the correction to the modification according to the residual, $ f - \L_{\B} (u_1\beta) = -\L_{\B} (u_1\beta)\chi_{\Omega_2}$. Lemma \ref{lemma:src_trans} is  applied in the additive DDM  \cite{Leng2019} for two types of boundaries, the straight line and the fold line, as  shown in Figure \ref{fig:twopart}-(b) and (c), which correspond to the horizontal/vertical transfer and the corner transfer, respectively.
The overlapping region $\widetilde{\Omega}_1\setminus\Omega_1 $ in Lemma \ref{lemma:src_trans} is   handled with 
 a  shifted PML media profile function $\widehat{\widehat{\sigma}}(t)$   defined by
\begin{equation} \label{eq:PMLprofile}
\widehat{\widehat{\sigma}}(t) = \left\{  
\begin{array}{ll}
0,                          & \text{if} \,\,\, t \leq d,\\
\widehat{\sigma}(t-d), & \text{if} \,\,\, t > d.\\
\end{array}
\right. 
\end{equation}
For simplicity, the above shifted medium profile is denoted by ${\widehat{\sigma}}(t)$  in the rest of the paper,  and  when we refer to the PML problem $\P_{B}$, an extended region of width $d$ is always attached to the rectangular box $B$, for possible overlapping with  its neighbor regions.

\begin{figure}[!ht]	
	\centering
	\begin{minipage}[t]{0.38\linewidth}
		\vspace{0pt}
		\centering
		\includegraphics[width=0.9\textwidth]{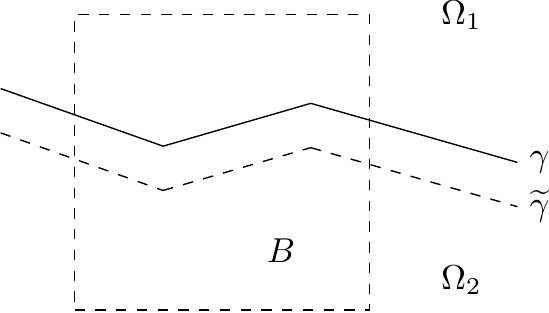}\\
	\end{minipage}
	\hspace{0.2cm}
	\begin{minipage}[t]{0.27\linewidth}
		\vspace{0pt}
		\centering
		\includegraphics[width=0.9\textwidth]{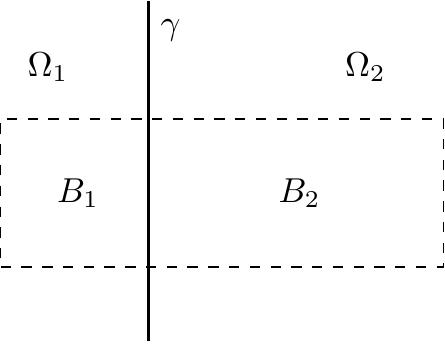}
	\end{minipage}
	\hspace{0.4cm}
	\begin{minipage}[t]{0.27\linewidth}
		\vspace{0pt}
		\centering
		\includegraphics[width=0.85\textwidth]{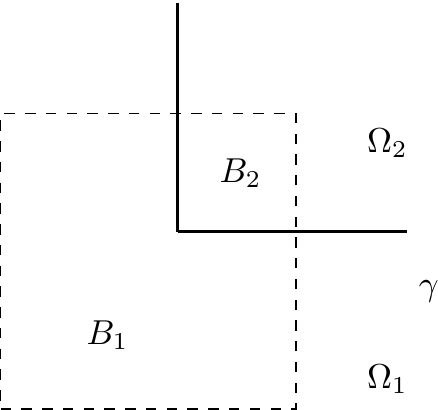}
	\end{minipage}\\
	\vspace{0.2cm}
	\begin{minipage}[t]{0.38\linewidth}
		\centering
		(a)
	\end{minipage}
	\hspace{0.2cm}
	\begin{minipage}[t]{0.27\linewidth}
		\centering
		(b)
	\end{minipage}
	\hspace{0.4cm}
	\begin{minipage}[t]{0.27\linewidth}
		\centering
		(c)
	\end{minipage}	
	\caption{Divide $\R^2$ and the box $B$ with  a piecewise smooth curve $\gamma$.\label{fig:twopart} }
\end{figure}


The domain decomposition that we use  is stated below.
The rectangular domain $\Omega = [-l_1,l_1]\times[-l_2,l_2]$ in $\R^2$ is uniformly partitioned into $\Nbx \times \Nby$ nonoverlapping rectangular subdomains.  Let $\Delta\xi = 2 l_1 / \Nbx$, $\xi_i = -l_1 + (i-1) \Delta \xi$ for $i = 1, 2, \ldots, \Nbx+1$, and $\Delta\eta = 2 \l_2 / \Nby$, $\eta_j = -l_2 + (j-1) \Delta\eta$ for $j = 1, 2, \ldots, \Nby+1$.  Then we have $\Nbx\times \Nby$ nonoverlapping rectangular subdomains as
$$\Omega_{i,j}: = [\xi_{i}, \xi_{i+1}] \times [\eta_{j}, \eta_{j+1}],\qquad i = 1, 2,\ldots, \Nbx,\; j = 1, 2,\ldots, \Nby.$$
For convenience, we also define the box $\Omega_{i0,i1;j0,j1}$ ($1\leq i0\leq i1\leq \Nbx+1$, $1\leq j0\leq j1\leq \Nby+1$), which consists of a set of  rectangular subdomains:
$$\Omega_{i0,i1;j0,j1} := \bigcup\limits_{i0 \leq i \leq i1 \atop j0 \leq j \leq j1} \Omega_{i, j}.$$ It is clear that the PML equation associated with each rectangular subdomain $\Omega_{i,j}$ needs to be solved in the DDM method.  The source $f$, which is assumed to be compactly supported in $\Omega$, is decomposed to
$$f_{i,j} = f \cdot \chi_{\Omega_{i,j}} , \qquad i = 1, 2\ldots, \Nbx, \;j = 1, 2, \ldots, \Nby.$$
Notice that the PML profile as \eqref{eq:PMLprofile} makes each subdomain has an overlapping region with its neighbor subdomains, thus we next define an overlapping domain decomposition of the two-dimensional space $\R^2$ as 
$$\widetilde{\Omega}_{i,j}
: = (\tilde{\xi}_{i}-d, \tilde{\xi}_{i+1}+d) 
\times (\tilde{\eta}_{j} - d, \tilde{\eta}_{j+1} + d),
\qquad i = 1, 2, \ldots, \Nbx, j = 1, 2, \ldots, \Nby,$$
where 
\begin{align*}
\tilde{\xi}_{i} &= \left\{ 
\begin{array}{ll}
-\infty,  &\,\,\, i = 1,\\
\xi_i,  & \,\,\, i = 2,\ldots,\Nbx,\\
+\infty, & \,\,\, i = \Nbx+1,
\end{array}
\right. &  
\tilde{\eta}_{j} &= \left\{
\begin{array}{ll}
-\infty,  &\,\,\, j = 1,\\
\eta_j,  & \,\,\, j = 2,\ldots,\Nby,\\
+\infty, & \,\,\, j = \Nby+1.
\end{array}
\right. 
\end{align*}

Similarly, for the cuboidal domain $\Omega = [-l_1,l_1]\times[-l_2,l_2]\times[-l_3,l_3]$ in $\R^3$, the partition in $z$-direction is done with $\Delta\zeta = 2 \l_3 / \Nbz$, $\zeta_k = -l_3 + (k-1) \Delta\zeta$ for $k = 1, \ldots, \Nbz+1 $, then we have $\Nbx \times \Nby \times \Nbz$ nonoverlapping subdomains $\Omega_{i,j,k}$, overlapping subdomains $\widetilde{\Omega}_{i,j,k}$, and decomposed sources $f_{i,j,k} = f\cdot \chi_{\Omega_{i,j,k}}$.

\subsection{The additive overlapping DDM with source transfer}

The  additive  DDM proposed in \cite{Leng2019} is based on checkerboard domain decomposition and source transfer between overlapping subdomains. Let us first illustrate it with the $2\times 2$ domain partition in $\R^2$.
A few notations are first introduced below.
Two truncation functions are defined as
$$\chi_{\rightarrow} = \chi_{(\xi_2, +\infty) \times (-\infty, +\infty)}, \quad
\chi_{\uparrow} =\chi_{(-\infty, +\infty) \times (\eta_2, +\infty)},$$
and four one-dimensional cutoff functions are defined as
$$\beta_{\rightarrow} = \widehat{\beta}\Big(\frac{x_1 - \xi_2}{d}\Big),\quad
\beta_{\leftarrow} = \widehat{\beta}\Big(\frac{\xi_2 - x_1}{d}\Big),\quad
\beta_{\downarrow} = \widehat{\beta}\Big(\frac{\eta_2 - x_2}{d}\Big),\quad
\beta_{\uparrow} = \widehat{\beta}(\frac{x_2 - \eta_2}{d}\Big),$$
 where $\widehat{\beta}(t)$ is a monotone cutoff function in $C^2(\R)$ such that $\widehat{\beta}(t) = 1$ for
$t \leq 0$, $\widehat{\beta}(t) = 0$ for $t \geq 1$, and
$|\widehat{\beta}^{\prime}(t)| < C$ for $ 0 < t < 1.$ 
With the above one-dimensional  cutoff functions, the corresponding two-dimensional cutoff functions 
associated with  the subdomains $\Omega_{i,j}$ ($i,j=1,2$) are defined as  
$$\beta_{\Omega_{1,1}} = \beta_{\rightarrow} \beta_{\uparrow},
\quad\beta_{\Omega_{2,1}} = \beta_{\leftarrow} \beta_{\uparrow},
\quad\beta_{\Omega_{1,2}} = \beta_{\rightarrow} \beta_{\downarrow},
\quad\beta_{\Omega_{2,2}} = \beta_{\leftarrow} \beta_{\downarrow}. $$
Denote by $\L_{i,j}$  the linear operator associated with the PML problem $\P_{\Omega_{i,j}}$.

Let us first consider the simple  case that the source lies inside $\Omega_{1,1}$. At step 1, the subdomain PML problem $\P_{\Omega_{1,1}}$ is solved with the source $f_{1,1}$ and the solution is denoted by $u_{0}$, as  is shown in Figure \ref{fig:ddm2d}-(a), 
the horizontal and vertical transferred sources are computed on each subdomain. 
At step 2, 
the local problem $\P_{\Omega_{2,1}}$ is solved with the right transferred source
$\L_{{1,1}}(u_0\beta_{\rightarrow}) \chi_{\rightarrow}$  (see Figure \ref{fig:ddm2d}-(b)) as the local source,
and the solution is denoted by $u_{\rightarrow}$ (see Figure \ref{fig:ddm2d} (c)).  By using 
Lemma \ref{lemma:src_trans} for $\Omega_{1,2;1,1}$, the rightward source transfer is applied and we have
\begin{equation} \label{eq:u2x2lower}
  u_0\beta_{\rightarrow} + u_{\rightarrow} = u, \qquad  \text{in} 
\quad (-\infty,+\infty)\times(-\infty,\eta_2+d),
\end{equation}
as is shown in
Figure \ref{fig:ddm2d}-(d).  Similarly, the local solution
$u_{\uparrow}$ of the subdomain $\Omega_{1,2}$ are obtained and we
have
\begin{equation} \label{eq:u2x2left}
  u_0\beta_{\uparrow} + u_{\uparrow} = u, \qquad  \text{in} 
\quad (-\infty, \xi_2+d)\times(-\infty,+\infty),
\end{equation}
as  shown in
Figure \ref{fig:ddm2d}-(f) and (g).

Note that the additive DDM  has an important property that the subdomain solving is not direction related, in the sense that, on each subdomain, the transferred sources coming from different directions are summed into one local source and then solved. 
  At step 3, the subdomain solution has already been constructed in $\Omega_{1,1}$, $\Omega_{1,2}$ and $\Omega_{2,1}$ in previous steps, thus only the solution for $\Omega_{2,2}$ needs to be constructed.  However, in order to derive an algorithm that is not direction related as mentioned above, instead of directly using a corner  source transfer, the horizontal and vertical source transfers are applied again on each subdomain, though the nonzero ones are only the upper transfer of $u_\rightarrow$ from $\Omega_{2,1}$ and the right transfer of $u_\uparrow$ from $\Omega_{1,2}$. In addition to the horizontal and vertical sources $-\L_{{1,2}}(u_{\uparrow}\beta_{\rightarrow}) \chi_{\rightarrow}$ and $-\L_{{2,1}}(u_{\rightarrow}\beta_{\uparrow}) \chi_{\uparrow}$ transferred to $\Omega_{2,2}$,  the corner direction transferred source $-\L_{{2,2}}(u_{0}\beta_{\rightarrow}\beta_{\uparrow}) \chi_{\rightarrow}\chi_{\uparrow}$ is also passed to $\Omega_{2,2}$. Using \eqref{eq:u2x2lower} and \eqref{eq:u2x2left}, we have that the summation of transferred sources for $\Omega_{2,2}$ is in fact $\L_{\Omega}(u \overline{\beta}_{\nearrow})$, where $\overline{\beta}_{\nearrow} = 1 - (1-\beta_{\rightarrow})(1-\beta_{\uparrow})$ is the cutoff function for the L-shaped domain $\Omega_{1,1} \cup \Omega_{2,1} \cup \Omega_{1,2} $, then by using Lemma \ref{lemma:src_trans}, the corner  source transfer is applied, and we know that $u \overline{\beta}_{\nearrow} + u_{\nearrow} = u$. 
Now that the local solutions in all subdomains are obtained, a formula of the global solution expressed as the combination of local solutions  is then to be derived.
From  \eqref{eq:u2x2lower} and \eqref{eq:u2x2left}, it holds that
\begin{align}
u_0\beta_{\uparrow}\beta_{\rightarrow} + u_{\rightarrow} \beta_{\uparrow} &= u \beta_{\uparrow},  \;\qquad \text{in} \,\,\, \R^2,\\
u_0\beta_{\uparrow}\beta_{\rightarrow} + u_{\uparrow} \beta_{\rightarrow} &= u \beta_{\rightarrow}, \qquad  \text{in} \,\,\, \R^2,
\end{align}
thus $u \overline{\beta}_{\nearrow}$ can be expressed as 
$u \overline{\beta}_{\nearrow} = u_0\beta_{\rightarrow}\beta_{\uparrow} +
u_{\rightarrow}\beta_{\uparrow} + u_{\uparrow} \beta_{\rightarrow}$,
and we have 
$$u_0\beta_{\rightarrow}\beta_{\uparrow} +
u_{\rightarrow}\beta_{\uparrow} + u_{\uparrow} \beta_{\rightarrow} +
u_{\nearrow} = u,$$
or in a more symmetric form,
$$u_0\beta_{\Omega_{1,1}} + u_{\rightarrow} \beta_{\Omega_{2,1}} + u_{\uparrow} \beta_{\Omega_{1,2}} + u_{\nearrow}\beta_{\Omega_{2,2}} = u.$$

Now we are ready to state the additive DDM for $2\times 2$ domain partition in the case of general source $f$.  At step 1, solve the local problems $\P_{\Omega_{i,j}}$  with local sources $f_{i,j}$ and denote the solutions  as $u_{i,j}^1$.  Then at step 2, solve the local problems $\P_{\Omega_{i,j}}$  with horizontal and vertical transferred sources 
calculated by using the solutions of step 1, and denote the solutions as $u_{i,j}^2$. Finally at step 3, solve the local problems $\P_{\Omega_{i,j}}$    with horizontal, vertical and corner transferred sources calculated by using the solutions of step 1 and 2, and denote the solution as $u_{i,j}^2$. Then the DDM solution is constructed to be
$$u_{\text{DDM}} = \sum\limits_{s=1,2,3} \sum\limits_{i = 1,2 \atop j = 1,2} u_{i,j}^{s}\beta_{\Omega_{i,j}}, $$ 
which is indeed the solution to $\P_{\Omega}$ with source $f$ in the constant medium case. 

\def\wdff{0.24}
\begin{figure*}[!ht]	
	\centering
	\begin{minipage}[t]{\wdff\linewidth}
		\centering
		\includegraphics[width=1\textwidth]{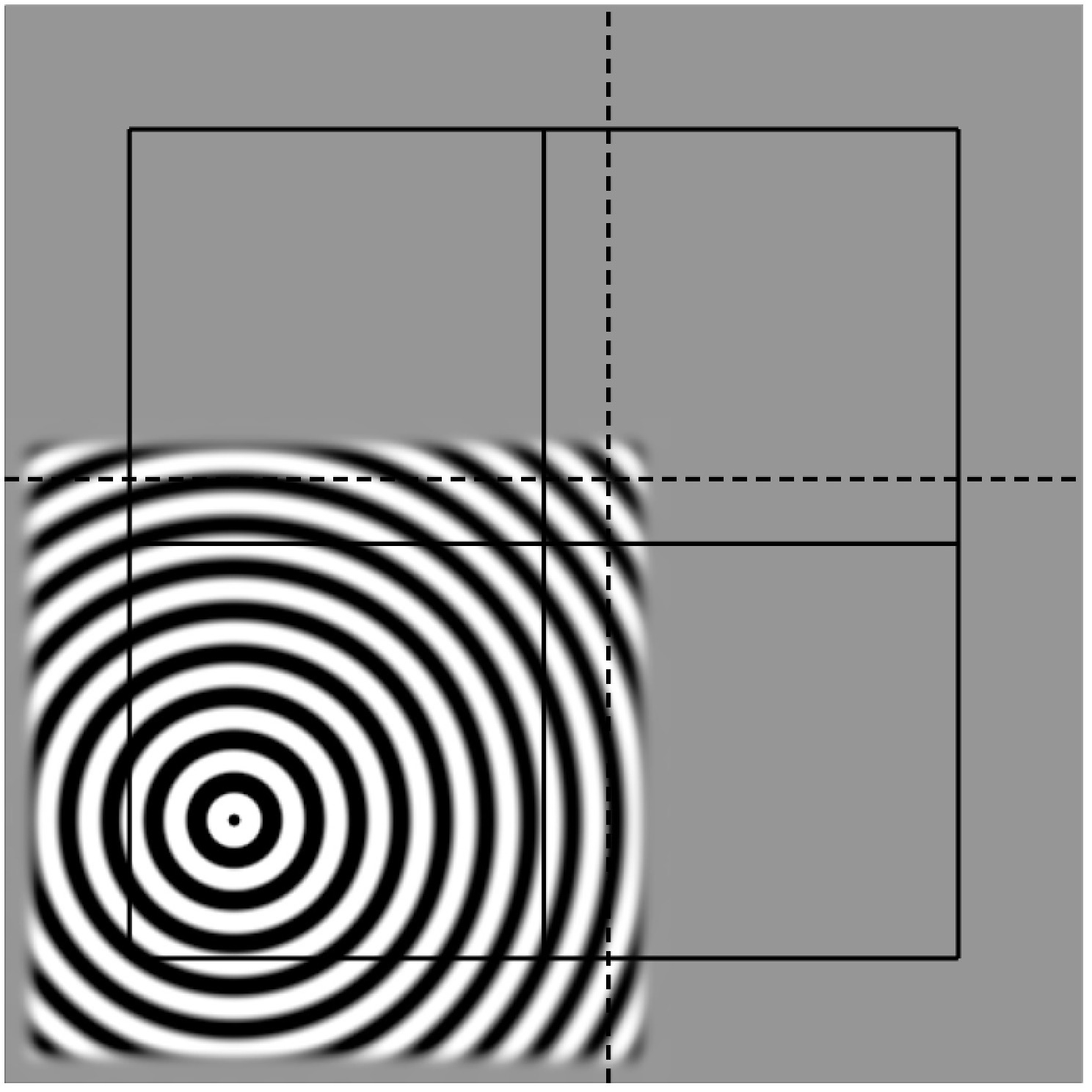}\\
		(a) $u_0$
	\end{minipage}
	\begin{minipage}[t]{\wdff\linewidth}
		\centering
		\includegraphics[width=1\textwidth]{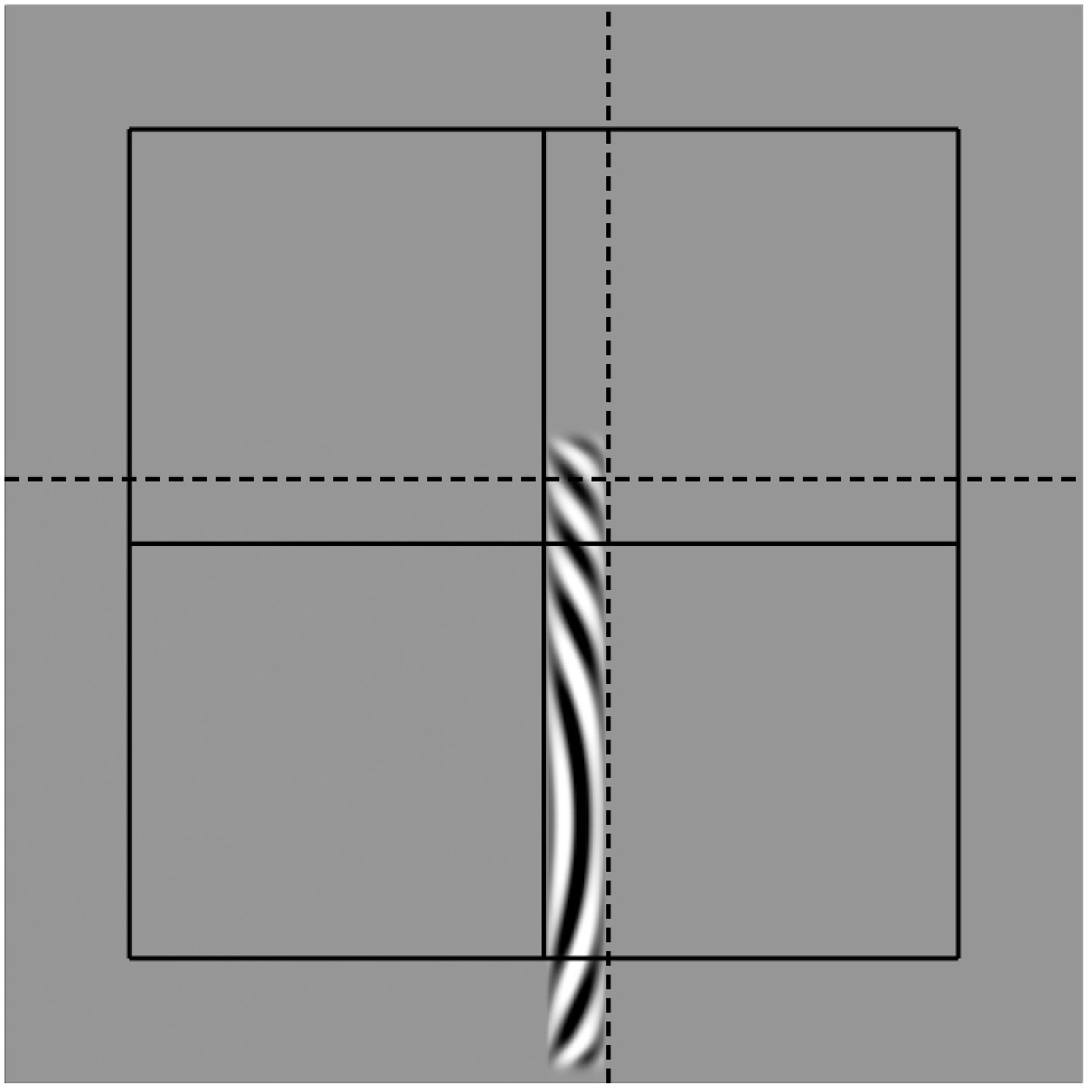}\\
		(b) $\L_{{1,1}}u_0(\beta_{\rightarrow} ) \chi_{\rightarrow}$
	\end{minipage}
	\begin{minipage}[t]{\wdff\linewidth}
		\centering
		\includegraphics[width=1\textwidth]{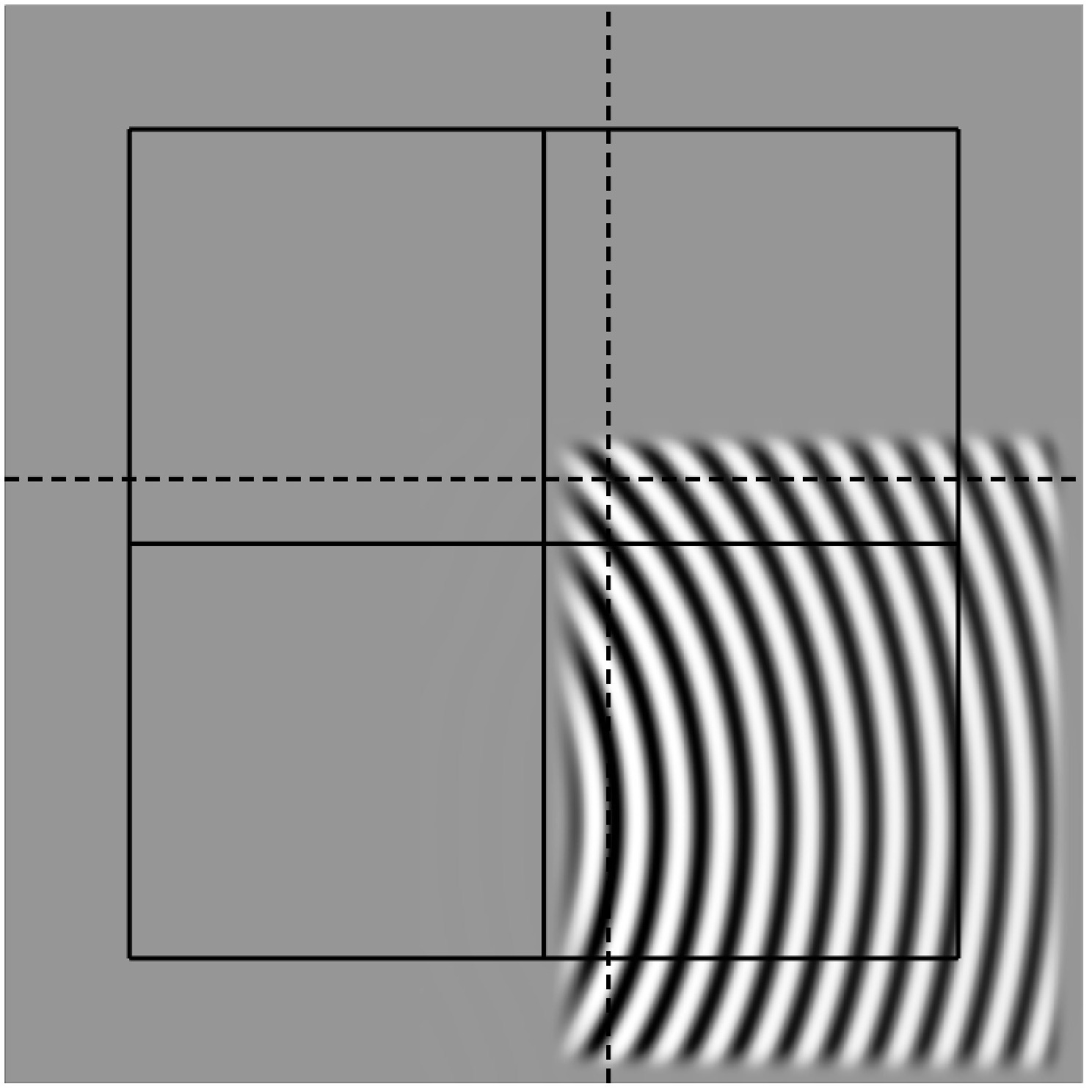}\\
		(c) $u_{\rightarrow}$
	\end{minipage}
	\begin{minipage}[t]{\wdff\linewidth}	
		\centering
		\includegraphics[width=1\textwidth]{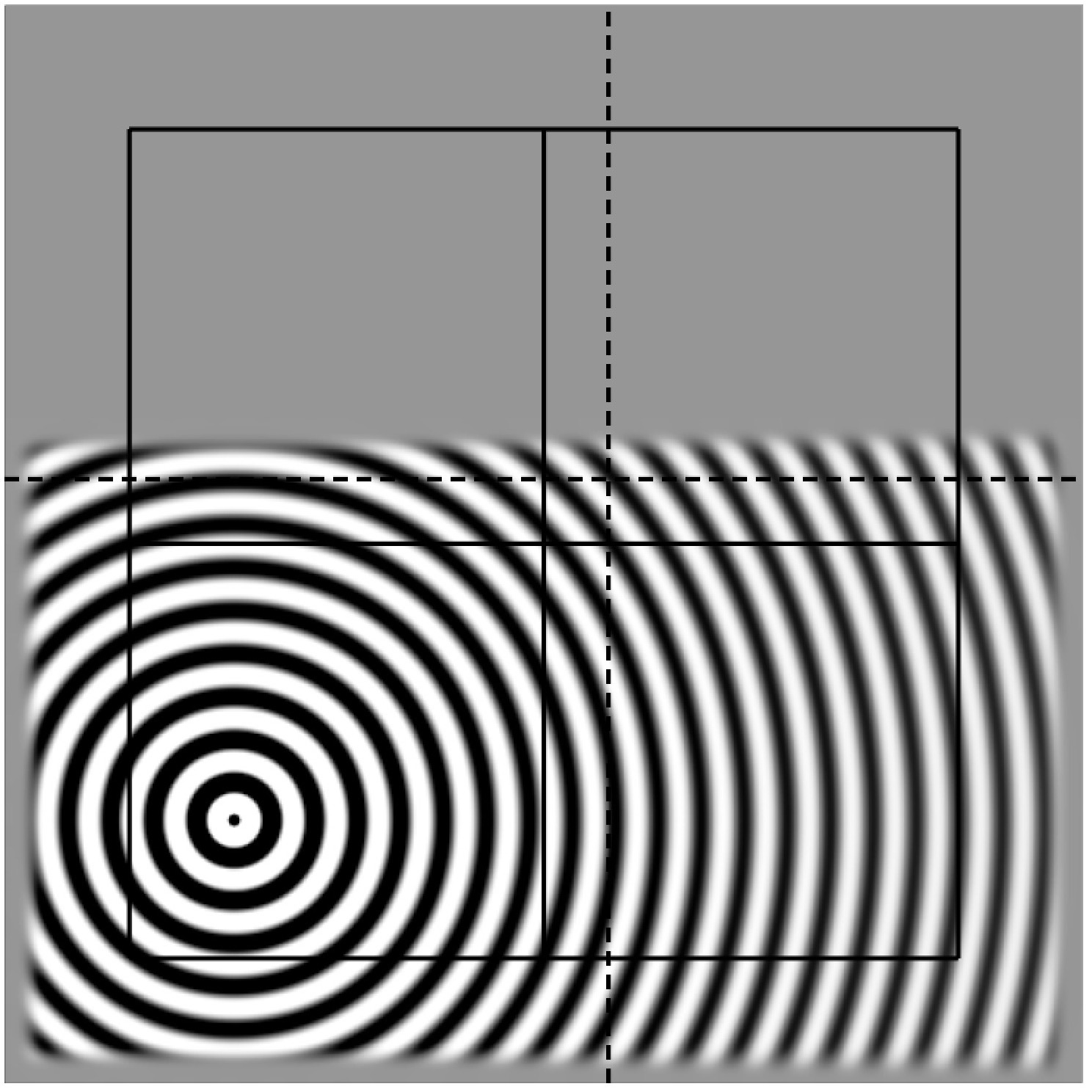}\\
		(d) $u_0 \beta_{\rightarrow} + u_{\rightarrow} $
	\end{minipage}	\\	
	\vspace{0.3cm}
	
	\begin{minipage}[t]{\wdff\linewidth}
		\centering
		\includegraphics[width=1\textwidth]{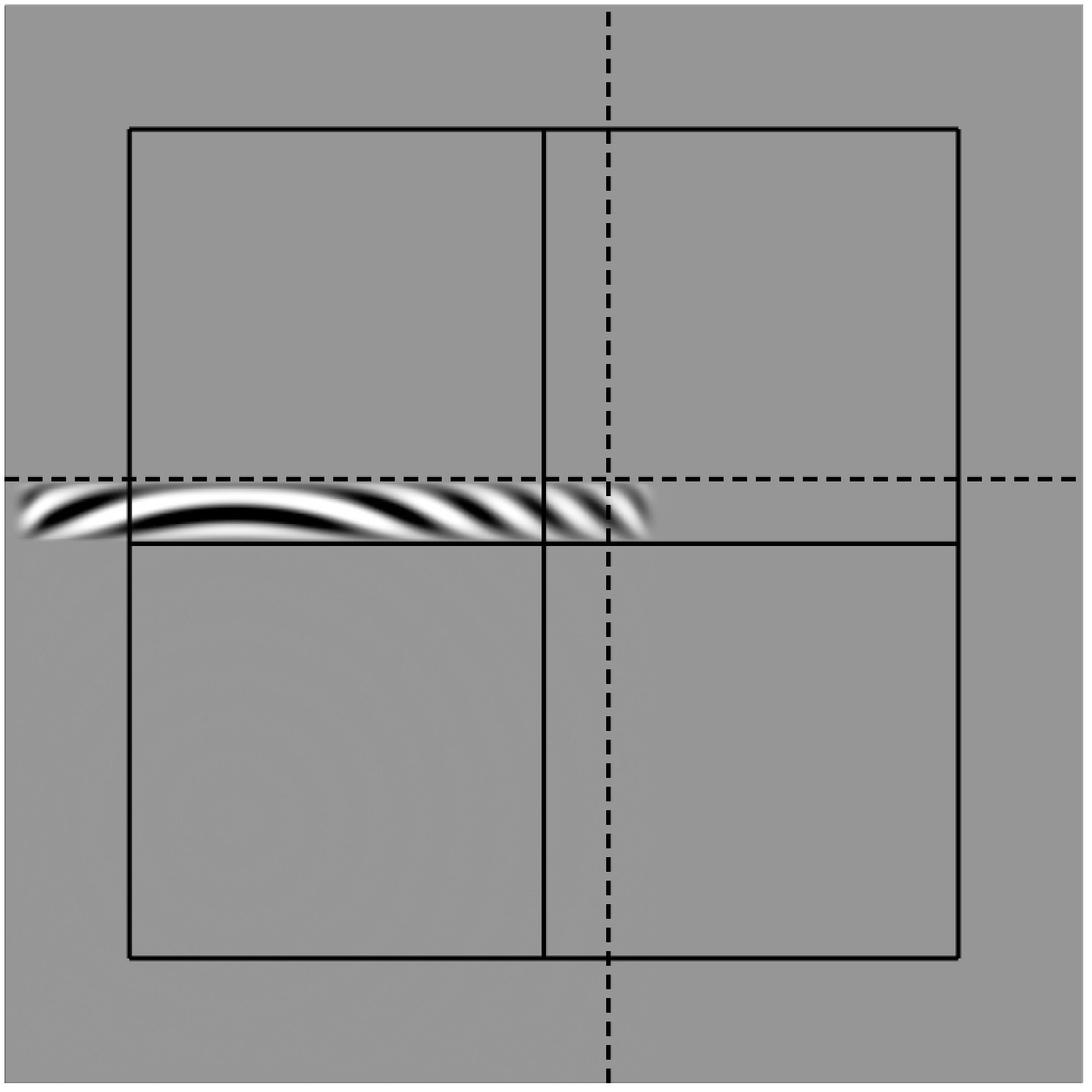}\\
		(e) $\L_{{1,1}}(u_0\beta_{\uparrow}) \chi_{\uparrow}$
	\end{minipage}
	\begin{minipage}[t]{\wdff\linewidth}
		\centering
		\includegraphics[width=1\textwidth]{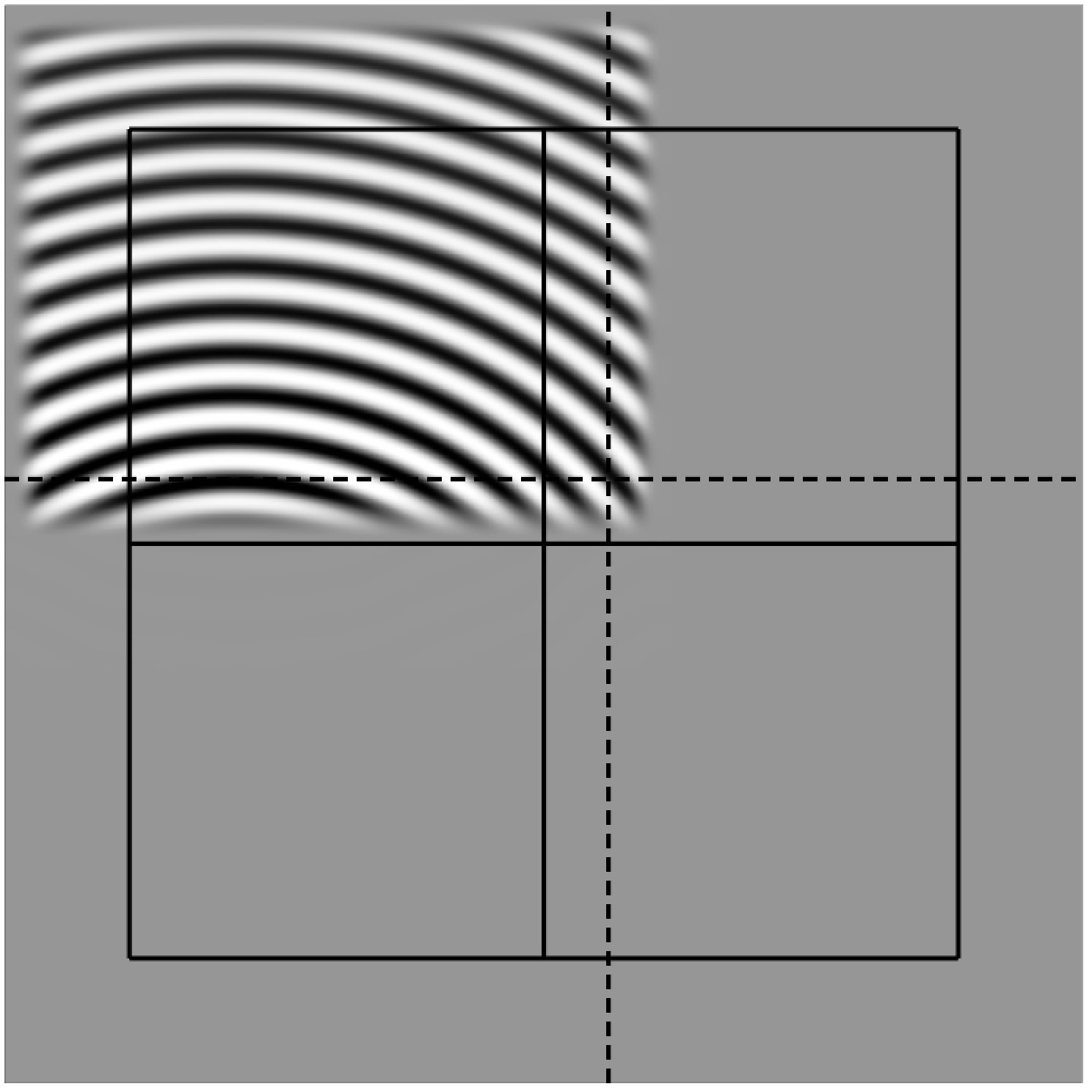}\\
		(f) $u_{\uparrow}$
	\end{minipage}
	\begin{minipage}[t]{\wdff\linewidth}	
		\centering
		\includegraphics[width=1\textwidth]{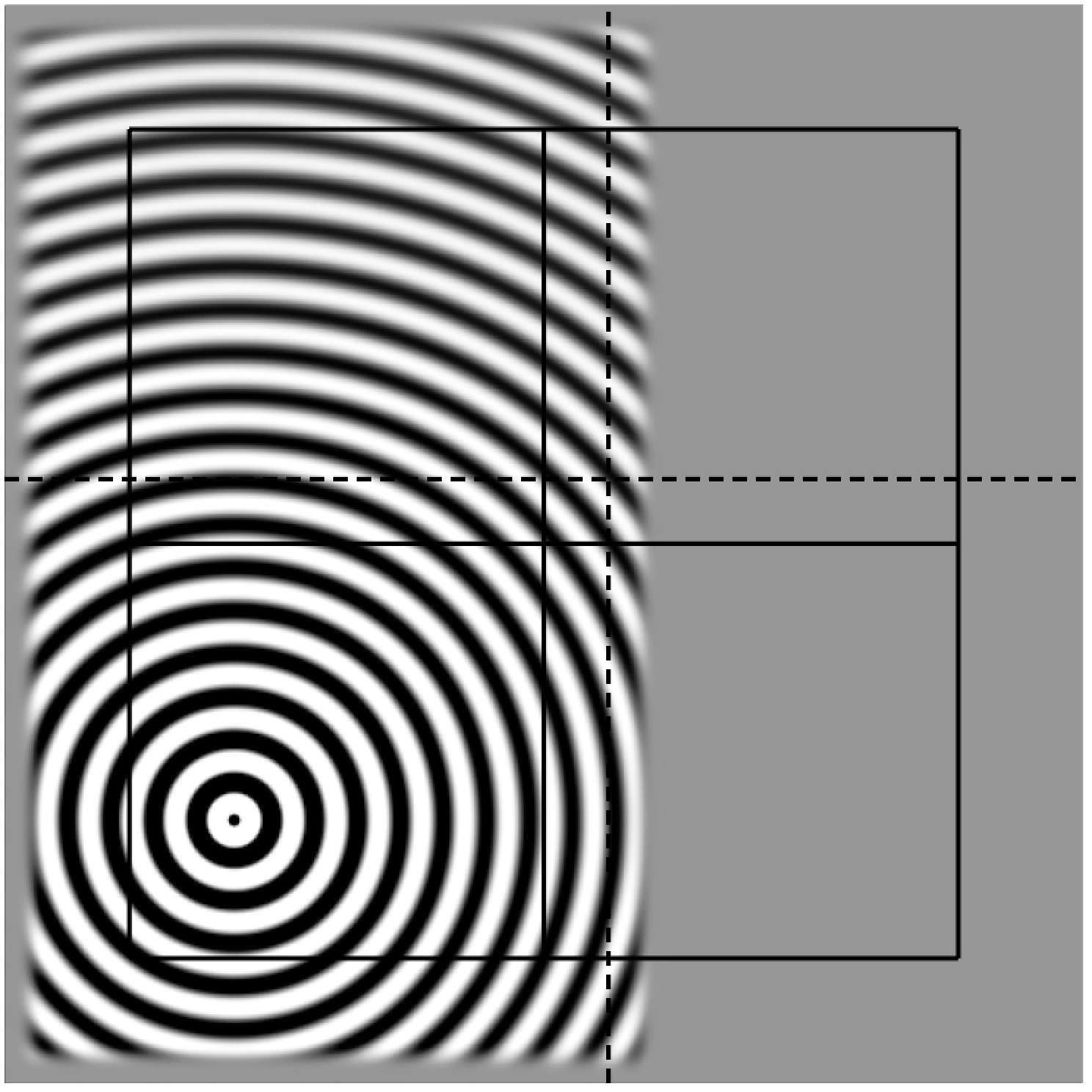}\\
		(g) $u_0 \beta_{\uparrow} + u_{\uparrow} $
	\end{minipage}
	\begin{minipage}[t]{\wdff\linewidth}
		\centering
		\includegraphics[width=1\textwidth]{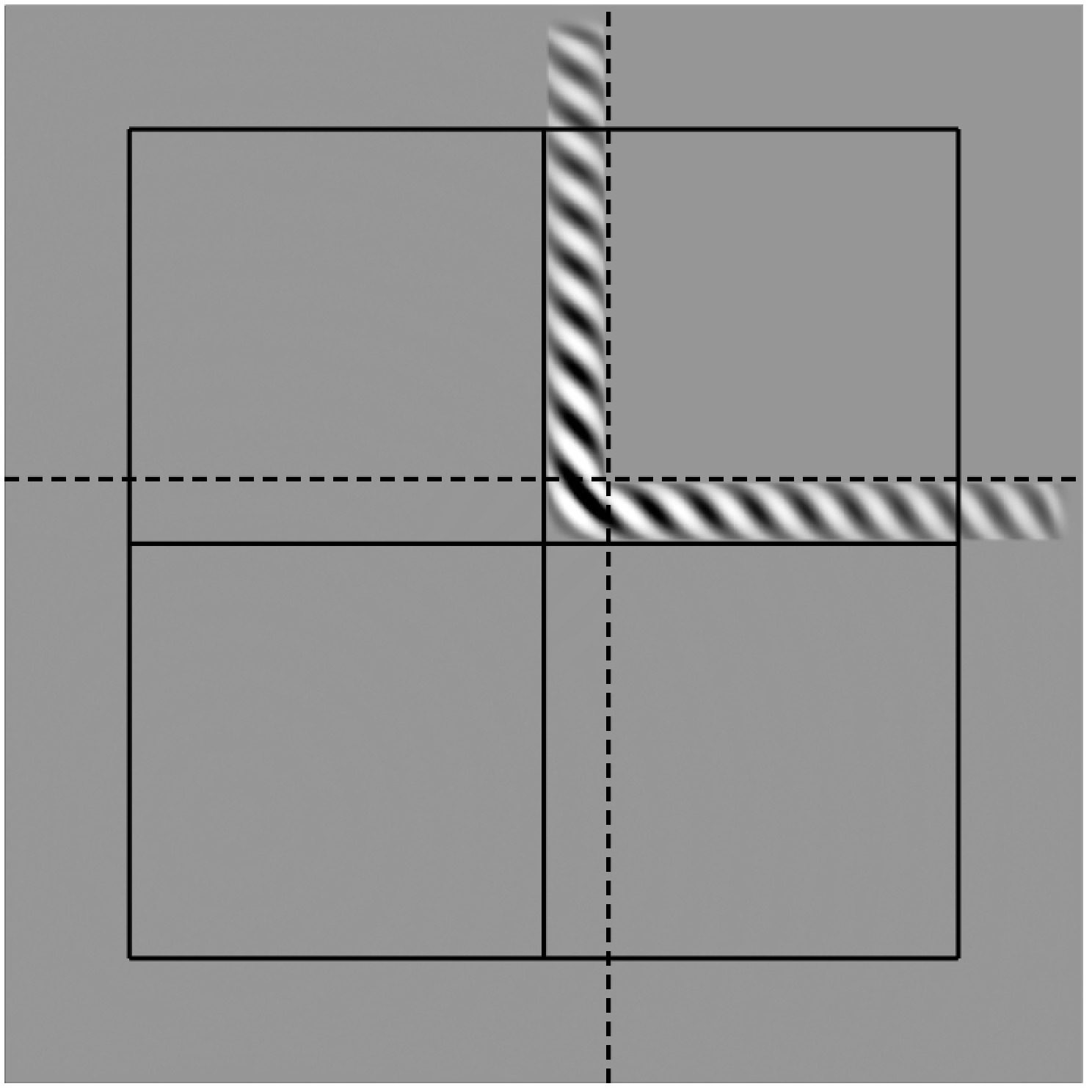}\\
		(h) 
		$\L_{{1,2}}(u_{\uparrow}\beta_{\rightarrow})
		\chi_{\rightarrow}$\\
		$+\L_{{2,1}}(u_{\rightarrow}\beta_{\uparrow}) \chi_{\uparrow}$\\
		$+\L_{{2,2}}(u_{0}\beta_{\rightarrow}\beta_{\uparrow})
		\chi_{\rightarrow}\chi_{\uparrow}$ 
	\end{minipage}		\\	
	\vspace{0.3cm}
	
	\begin{minipage}[t]{\wdff\linewidth}
		\centering
		\includegraphics[width=1\textwidth]{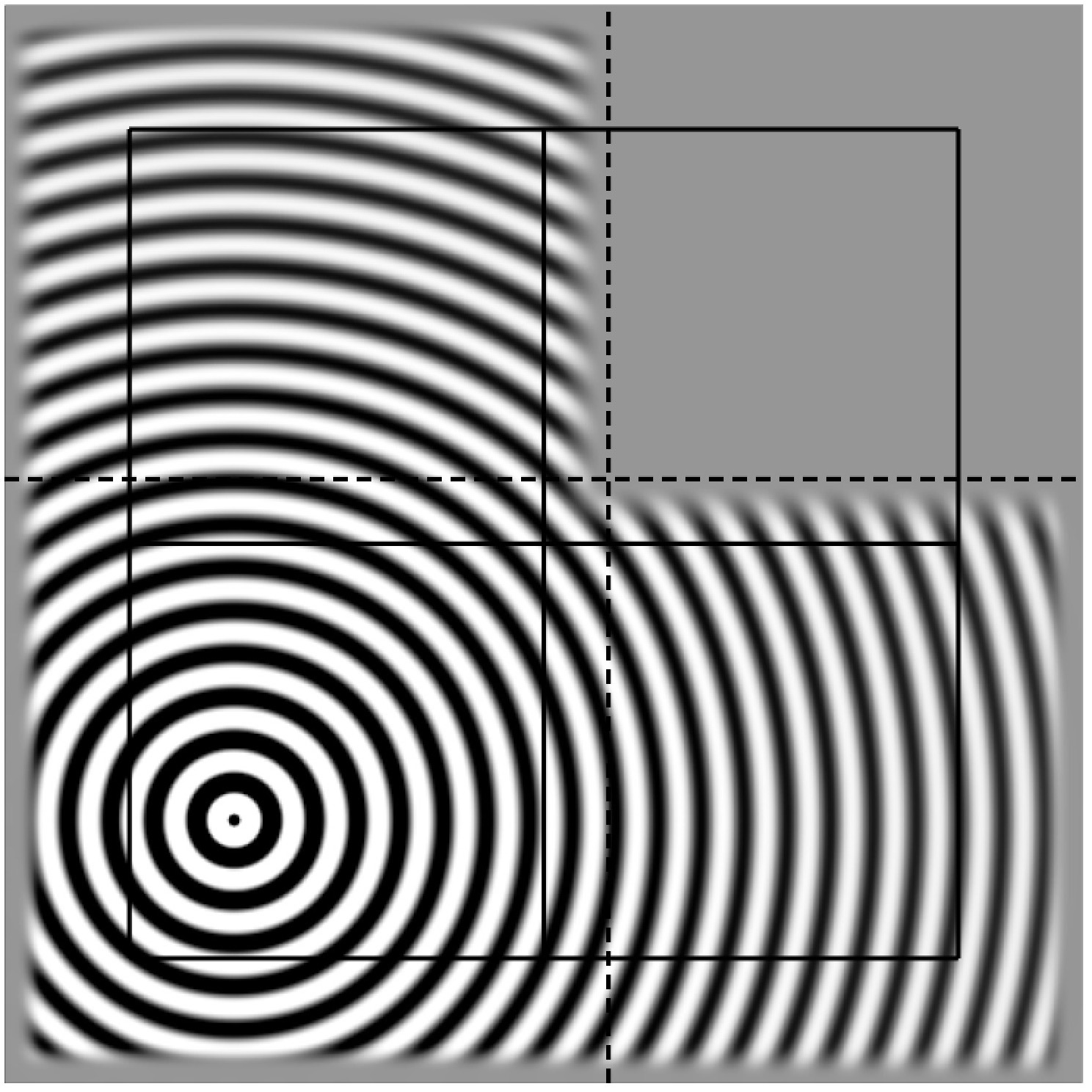}\\
		(j) $u \overline{\beta}_{\nearrow}$
	\end{minipage}
	\begin{minipage}[t]{\wdff\linewidth}	
		\centering
		\includegraphics[width=1\textwidth]{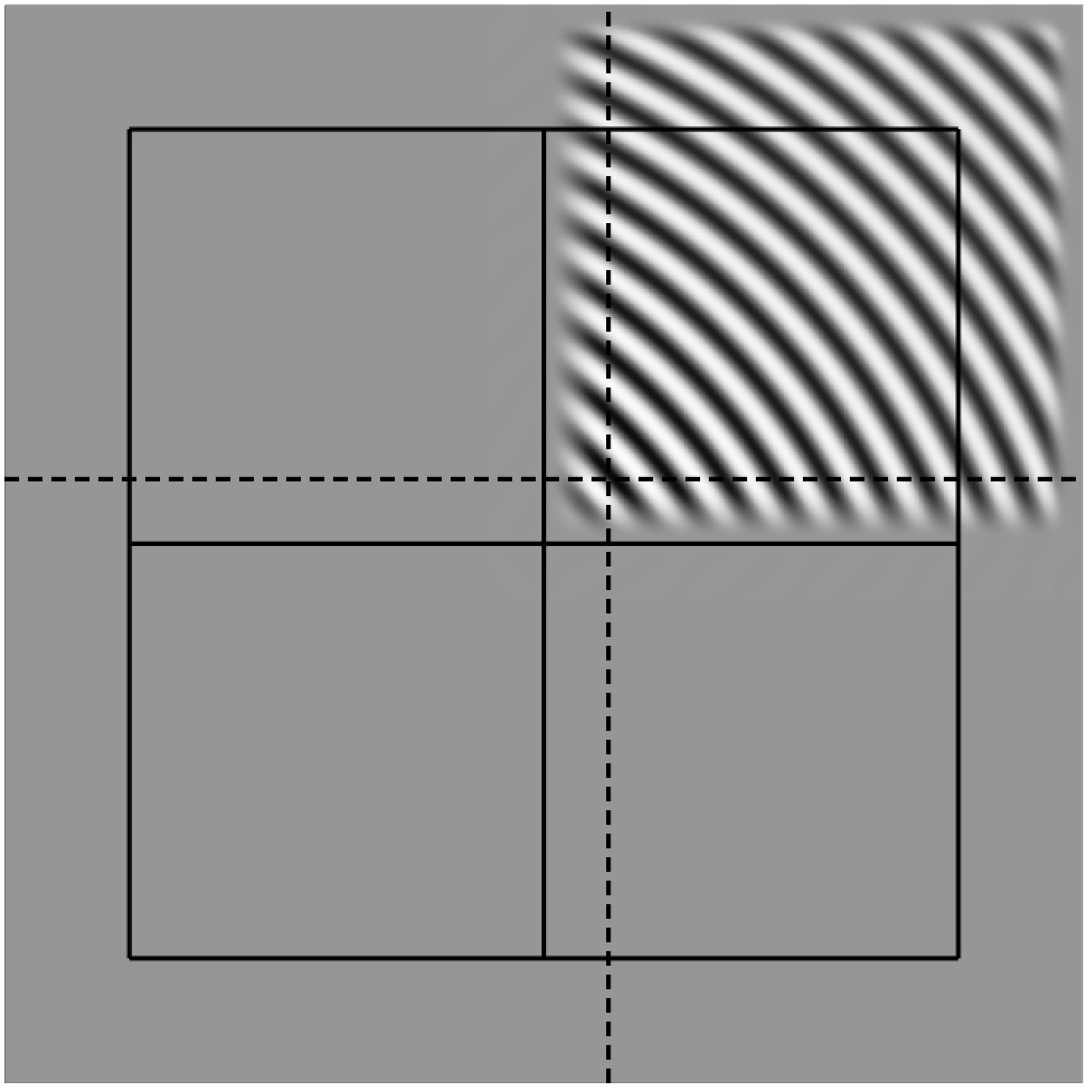}\\
		(k) $u_{\nearrow}$
	\end{minipage}	
	\begin{minipage}[t]{\wdff\linewidth}
		\centering
		\includegraphics[width=1\textwidth]{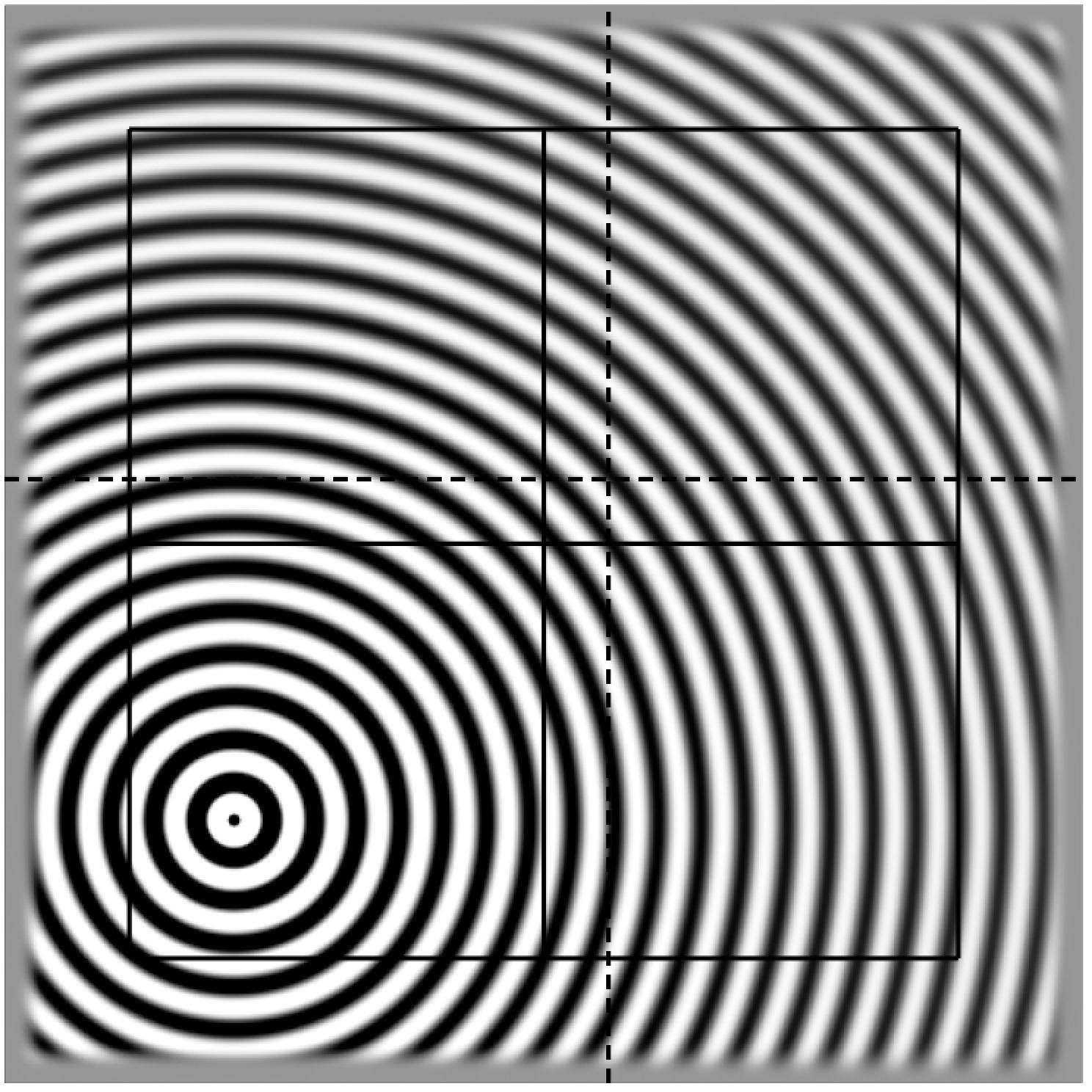}\\
		(l) $u$
	\end{minipage}
	\caption{{Illustration of the source transfer DDM solving process with the $2\times 2$ domain decomposition, where the source lies in $\Omega_{1,1}$ only.  The borders of subdomains  are shown with solid lines, and the upper and  right borders of overlapping regions are shown with dotted lines, which are  $x_1=\xi_{2}+d$	and $x_2=\eta_{2}+d$, respectively.} \label{fig:ddm2d}}
\end{figure*}

To  illustrate the extension of the additive DDM from $2\times 2$ to $ \Nbx  \times \Nby$ domain partitions,  let us define the following one-dimensional cutoff functions, 
\begin{align*}
\beta^{(1)}_{\Box, i}(x_1) &= \left\{ 
\begin{array}{ll}
\widehat{\beta}(\frac{ \xi_i - x_1}{d}),  & \, \Box = -1, \,\text{and}\, i \neq 1, \\
\widehat{\beta}(\frac{x_1 - \xi_{i+1} }{d}),  & \, \Box = 1, \,\text{and}\, i \neq \Nbx, \\
1,  & \, \, \text{otherwize},
\end{array}
\right. &  \beta^{(2)}_{\vartriangle, j}(x_2) &= \left\{
\begin{array}{ll}
\widehat{\beta}(\frac{\eta_j - x_2}{d}),  & \, \vartriangle = -1, \,\text{and}\, j \neq 1, \\
\widehat{\beta}(\frac{x_2 - \eta_{j+1}}{d}),  & \, \vartriangle = 1, \,\text{and}\, j \neq \Nby, \\
1,  & \, \text{otherwize},
\end{array}
\right. 
\end{align*}
for $i=1,\ldots, \Nbx$, $j=1,\ldots, \Nby$ and $\Box, \vartriangle = \pm 1, 0$.
Note that the symbols $\Box$ and $\vartriangle$ are used to indicate the signs of the $x$ and $y$  components of a direction, respectively. 
With the above one-dimensional cutoff functions, two-dimensional ones are
defined as 
$$\beta_{\Box, \vartriangle;\, i, j}(x_1, x_2) := \beta^{(1)}_{\Box, i}(x_1)  \beta^{(2)}_{\vartriangle, j}(x_2), $$ 
for $\Box, \vartriangle = -1,0,1$ with $(\Box, \vartriangle) \neq (0,0)$, 
and 
\begin{align*}
\beta_{0, 0;\, i, j}(x_1, x_2) :=  \beta^{(1)}_{-1, i}(x_1) \beta^{(1)}_{+1, i+1}(x_1) \beta^{(2)}_{-1, j}(x_2)  \beta^{(2)}_{+1, j+1}(x_2).
\end{align*}

Define the following truncation functions for the half spaces and the quarter spaces in $\R^2$:
$$\chi_{\Box, \vartriangle;\, i, j}(x_1, x_2) := \chi_{I^{(1)}_{\, \Box, i}(x_1) \times I^{(2)}_{\vartriangle, j}(x_2)},$$
where
\begin{align*}
I^{(1)}_{\Box, i}(x_1) &= \left\{ 
\begin{array}{ll}
(-\infty, \xi_i),  & \, \Box = -1,\\
  (\xi_{i+1}, +\infty),  & \, \Box = 1, \\
(-\infty, +\infty),  & \, \Box = 0,                     
\end{array}
\right. &  I^{(2)}_{\vartriangle, j}(x_2) &= \left\{
\begin{array}{ll}
(-\infty, \eta_j),  & \, \vartriangle = -1,\\
(\eta_{j+1}, +\infty),  & \, \vartriangle = 1,\\
(-\infty, +\infty),  & \, \vartriangle = 0.
\end{array}
\right. 
\end{align*}
Using the above cutoff functions, truncation functions and linear operators,
we are able to define the source transfer operators in $\R^2$ as:
\begin{align*}
\Psi_{\Box, \vartriangle;\, i, j} (v) := -\L_{i + \Box, j + \vartriangle} (\beta_{\Box, \vartriangle;\, i, j}  v) \chi_{\Box, \vartriangle;\, i, j},
\end{align*}
for $ \Box, \vartriangle = \pm 1, 0$ with $(\Box, \vartriangle) \neq (0,0).$
Then the additive overlapping DDM with source transfer in $\R^2$  \cite{Leng2019} can be stated  as follows:

\begin{algorithm_}[Additive overlapping DDM with source transfer   in $\R^2$  \cite{Leng2019}]~ 
	\label{alg:add2D}
   \begin{algorithmic}[1]
   \State Set $\{u^{0}_{i,j}\} = 0$ in $\R^2$ for $i=1,2,\ldots, \Nbx,$ $j =1,2,\ldots,\Nby$.
   \State Step 1: solve the PML problems $\P_{\Omega_{i,j}}$ with the source $f_{i,j}$
     \begin{equation} \label{eq:alg_1_2}
       \L_{i,j} u^{1}_{i,j} = f_{i,j}, \quad in \;\;\R^2,
     \end{equation} 
     for $ i=1,2,\ldots,\Nbx,\; j =1,2,\ldots,\Nby$. 
     \For {Step $s = 2, 3, \ldots, \Nbx + \Nby - 1$}
     \State Solve the local subdomain problems: for $ i=1,2,\ldots,\Nbx,\; j =1,2,\ldots,\Nby$,
     \begin{align}  \label{eq:alg_2_2}
       \L_{i,j} u_{i,j}^{s} =
       \sum\limits_{\substack{\Box, \vartriangle = -1,0,1  \\  (\Box, \vartriangle) \neq (0, 0) } } 
       \Psi_{\Box, \vartriangle;\, i-\Box, j-\vartriangle} (u_{i-\Box, j-\vartriangle}^{s - |\Box| - |\vartriangle|}).
     \end{align}
     \EndFor
   \State The DDM solution for  $\P_{\Omega}$ with the source $f$ is then given by
     \begin{equation} \label{eq:sumadd2D}
       u_{\text{DDM}} =  \sum\limits_{s=1, \cdots, \Nbx+\Nby-1} \sum\limits_{i = 1,\ldots,\Nbx \atop j = 1,\ldots,\Nby} 
       \beta_{0, 0;\, {i, j}} u_{i,j}^s.
     \end{equation}
   \end{algorithmic}
 \end{algorithm_}
\vspace{0.2cm}

For the constant medium case, it is proved in \cite{Leng2019} that the DDM solution $u_{\text{DDM}}$ defined by \eqref{eq:sumadd2D} is the exactly the solution of the $\P_{\Omega}$ with the source $f$.  In the above Algorithm \ref{alg:add2D}, all the sources $f_{i,j}$, $i=1,2,\ldots,\Nbx$, $j = 1,2,\ldots,\Nby$ are solved simultaneously.
For a given source $f_{i_0, j_0}$, 
depending on the relative position of the subdomain $\Omega_{i,j}$ to  subdomain $\Omega_{i_0,j_0}$, the subdomains can be divided into two types in the solving process:
the first type consists of  the ones with either $i = i_0$ or $j = j_0$, 
of which the local solutions effected by source $f_{i_0, j_0}$  are obtained by applying horizontal or vertical source transfer, 
 the other type consists of  the ones with  $i \neq i_0$ and $j \neq j_0$, 
of which the local solutions effected by source $f_{i_0, j_0}$ are obtained by applying corner  source transfer. 

To illustrate the method in  $\R^3$, the same notations as above for the $x$ and $y$ components are re-used and we also add the notations for the $z$ component. 
The one-dimensional cutoff functions in the $z$ direction are defined as
\begin{align*}
\beta^{(3)}_{\ocircle, k}(x_3) &= \left\{ 
\begin{array}{ll}
\widehat{\beta}(\frac{\zeta_k - x_3}{d}),  & \, \ocircle = -1, \,\text{and}\, k \neq 1, \\
\widehat{\beta}(\frac{x_3 - \zeta_{k+1}}{d}),  & \, \ocircle = 1, \,\text{and}\, k \neq \Nbz,\\
1,  & \, \text{otherwise,}
\end{array}
\right.
\end{align*}
for $k = 1,\ldots,\Nbz$ and $\ocircle = \pm 1, 0$, 
then the cutoff function for each subdomain are
$$\beta_{\Box, \vartriangle, \ocircle;\, i, j, k}(\xi, \eta, \zeta) :=
\beta^{(1)}_{\Box, i}(\xi) \beta^{(2)}_{\vartriangle, j}(\eta)
\beta^{(3)}_{\ocircle, k}(\zeta),$$ 
where $\Box, \vartriangle, \ocircle = \pm 1, 0$ with $(\Box, \vartriangle, \ocircle) \neq (0, 0, 0)$,
and 
\begin{align*}
\beta_{0, 0, 0;\, i, j}(x_1, x_2) =  \beta^{(1)}_{-1, i}(x_1) \beta^{(1)}_{+1, i+1}(x_1) \beta^{(2)}_{-1, j}(x_2)  \beta^{(2)}_{+1, j+1}(x_2)  \beta^{(3)}_{-1, k}(x_3)  \beta^{(3)}_{+1, k+1}(x_3).
\end{align*}
The truncation functions for the half spaces, the quarter spaces and the eighth spaces in $\R^3$  are defined as
$$\chi_{\Box, \vartriangle, \ocircle;\, i, j, k}(x_1, x_2, x_3) := \chi_{I^{(1)}_{\, \Box, i}(x_1) \times I^{(2)}_{\vartriangle, j}(x_2)
	\times I^{(3)}_{\ocircle, k}(x_3)},$$ 
	where
\begin{align*}
I^{(3)}_{\ocircle, k}(x_3) &= \left\{ 
\begin{array}{ll}
(-\infty, \zeta_k),  & \, \ocircle = -1,\\
(\zeta_k, +\infty),  & \, \ocircle = 1, \\
(-\infty, +\infty),  & \, \ocircle = 0.                     
\end{array}
\right.
\end{align*}
Then the corresponding transfer function in $\R^3$ is defines as 
\begin{align}
\Psi_{\Box, \vartriangle, \ocircle;\, i, j, k} (v) := -\L_{i + \Box, j + \vartriangle, k + \ocircle} (\beta_{\Box, \vartriangle, \ocircle;\, i, j, k} \, v) \chi_{\Box, \vartriangle, \ocircle;\, i, j, k}, 
\end{align}
for  $\Box, \vartriangle, \ocircle = \pm 1, 0$ and $(\Box, \vartriangle, \ocircle) \neq (0, 0, 0)$.
The additive overlapping DDM with source transfer in $\R^3$  \cite{Leng2019} can be stated as follows:

\begin{algorithm_}[Additive overlapping DDM with source transfer  in $\R^3$  \cite{Leng2019}]~ 
	\label{alg:add3D}
   
  \begin{algorithmic}[1]
   \State Set $\{u^{0}_{i,j,k}\} = 0$ in $\R^3$ for $i=1,2,\ldots,\Nbx,$ $j =1,2,\ldots,\Nby$, $k = 1,2,\ldots,\Nbz$.
   \State Step 1: solve the PML problems $\P_{\Omega_{i,j,k}}$ with the source $f_{i,j, k}$
     \begin{equation} \label{eq:alg_1_3}
       \L_{i,j,k} u^{1}_{i,j,k} = f_{i,j,k}, \quad in \;\;\R^3,
     \end{equation}
     for $ i=1,2,\ldots,\Nbx,\; j =1,2,\ldots,\Nby$, $k = 1,2,\ldots,\Nbz$.
   \For {Step $s = 2, 3, \ldots, \Nbx + \Nby + \Nbz - 2$}
   \State Solve the local subdomain problems: for $ i=1,2,\ldots,\Nbx$, $j =1,2,\ldots,\Nby$, $k = 1,2,\ldots,\Nbz$, 
     \begin{align}  \label{eq:alg_2_3}
       \L_{i,j,k} u_{i,j,k}^{s} =
       \sum\limits_{\substack{\Box, \vartriangle, \ocircle = -1,0,1 \\ (\Box, \vartriangle, \ocircle) \neq (0,0,0)}} 
       \Psi_{\Box, \vartriangle, \ocircle;\, i-\Box, j-\vartriangle, k-\ocircle} (u_{i-\Box, j-\vartriangle, k- \ocircle}^{s - |\Box| - |\vartriangle| - |\ocircle|}).
     \end{align}
   \EndFor  
   \State The DDM solution for  $\P_{\Omega}$ with the source $f$ is then given by
     \begin{equation} \label{eq:sumadd3D}
       u_{\text{DDM}} =  \sum\limits_{s=1, \cdots, \Nbx+\Nby+\Nbz-2} \sum\limits_{\substack{i = 1,\ldots,\Nbx \\ j = 1,\ldots,\Nby \\ k = 1,\ldots, \Nbz } } 
       \beta_{0, 0, 0; i, j, k} u_{i,j, k}^s.
     \end{equation}
     \end{algorithmic}
\end{algorithm_}
\vspace{0.2cm}

It is shown in \cite{Leng2019} that in the case of source $f$ lying in only one subdomain $\Omega_{i_0,j_0}$ in $\R^2$ (or $\Omega_{i_0,j_0,k_0}$ in $\R^3$),  the subdomain $\Omega_{i,j}$ (or $\Omega_{i,j,k}$) performs 
nonzero local solving only at 
step $s = |i - i_0| + |j - j_0|+1$  (or $s = |i - i_0| + |j - j_0|+|k-k_0|+1$), and construct the exact solution $u$ in the subdomain at that very step. This results in subdomain solution marching in diagonal directions, and such diagonal marching  suggests a sweeping type solver, which will be derived in the next section.
We note that this property makes it possible to reduce  the sweeping solve of all directions \cite{Zepeda2019}  to  only diagonal directions.

\section{The diagonal sweeping DDM with source transfer in $\R^2$}
 
In this section we will develop the diagonal sweeping DDM with source transfer in $\R^2$ by starting with the source lying only inside  one subdomain.
If the exact solution is constructed for the case of the source lying within only one subdomain
and the solving procedure does not depend on such specific subdomain,
  then the exact solution could be constructed straightforwardly for the case of general source $f$, 
  since the solutions to  the decomposed sources, $f_{i,j}$'s,  are constructed simultaneously and together they form the total exact solution. 
Without loss of generality, we take a $5\times 5$ ($\Nbx=\Nby=5$) domain partition and assume that the source lies only in $\Omega_{3,3}$ ($i_0=j_0=3$)   in our illustration. There are totally $2^2=4$ diagonal directions in $\R^2$: $(+1,+1)$, $(-1,+1)$, $(+1,-1)$, $(-1,-1)$, and the sweep along each of the directions contains a total of $(\Nbx-1)+(\Nby-1)+1 = \Nbx+\Nby-1=9$ steps.
 
We  perform the {\bf first sweep} along the  direction $(+1,+1)$, i.e., from the  lower-left subdomains to the upper-right subdomains, 
where the $s$-th step of this sweep handles the group of subdomains $\{\Omega_{i,j}\}$ with $(i-1)+(j-1)+1=i+j-1=s$. 
  In the first $(i_0-1)+(j_0-1)=4$ steps, the solution is always zero since the local source in $\Omega_{i,j}$  with $(i-1)+(j-1)+1<  5$ is zero. 
  At step $(i_0-1)+(j_0-1)+1=5$, the subdomain problems in $\Omega_{i_0,j_0}=\Omega_{3,3}$ is solved with the source $f_{3,3}$, $3^2-1=8$  transferred sources are generated and passed to its  neighbor subdomains correspondingly for  later use, as  shown in Figure \ref{fig:sweep2d_12}-(a). 
At step 6, the subdomain problems in $\Omega_{3,4}$ and $\Omega_{4,3}$ are solved. Take $\Omega_{4,3}$ for example, the horizontal source transfer is applied, in which  the rightward transferred source from $\Omega_{3,3}$ at  step 5 is used as the local source for $\Omega_{4,3}$, the local subdomain problem is solved, and 5 new transferred sources are generated and  passed to its  corresponding neighbor subdomains, as shown in Figure \ref{fig:sweep2d_12}-(b).
At step 7, the subdomain problems in $\Omega_{3,5}$, $\Omega_{5,3}$ and $\Omega_{4,4}$ are solved. The cases in $\Omega_{3,5}$ and $\Omega_{5,3}$ are similar to  step 6. As for $\Omega_{4,4}$, the corner  source transfer is applied, in which the upward transferred source from $\Omega_{3,4}$ at step 6, the rightward transferred source from $\Omega_{4,3}$ at step 6, and the upper-right transferred source from $\Omega_{3,3}$ at  step 5 are summed as the local source for $\Omega_{4,4}$, the local subdomain problem is solved, and 3 new  transferred sources are generated and passed to its corresponding  neighbor subdomains, as shown in Figure \ref{fig:sweep2d_12}-(c).
The following steps in the sweeping continues and at step 9, the solution is constructed in 
the upper-right quadrant with respect to $(\xi_3, \eta_3)$, $\Omega_{i_0,\Nbx;j_0,\Nby}=\Omega_{3,5;3,5}$.

We note that the subdomains on which the upwards transfers are solved, namely $\Omega_{3,4}$ and $\Omega_{3,5}$, are handled by this sweep of upper-right direction, while  they are handled by the upwards sweep in the L-sweeps method \cite{Zepeda2019}.
Similarly, the subdomains on which the rightwards  transfers are solved, namely $\Omega_{4,3}$ and $\Omega_{5,3}$, are also handled by this sweep, while they are handled by the rightwards sweep in the L-sweeps method. These illustrate the major difference between the L-sweeps method and the proposed diagonal sweeping method, 
that is the horizontal and vertical sweeps in the former method are merged into the diagonal sweeps in the latter method.

It is clear that the directions of sweeps and the  source transfers are important in designing the sweeping algorithm. We define that two vectors $\d_1$ and $\d_2$ in $\R^2$ are in the {\sl similar direction} if and only if  $\d_1 \cdot \d_2 > 0$.
In the steps of the first sweep, it is found that only the transferred sources in the directions   $(+1,0)$, $(0,+1)$ and $(+1,+1)$ are used and they are in the similar directions of the current sweep $(+1,+1)$, while the others are left
for future  sweeps. Thus  the first rule on the  transferred source  in sweeps  in $\R^2$ is defined as:
\begin{rulee} {\rm (Similar directions in $\R^2$)} \label{rule2d_a}
A transferred source  which is not in the similar direction of one sweep should not be used in that sweep.
\end{rulee}

\def\wdff{0.325}
\begin{figure*}[!ht]
	\centering
	\begin{minipage}[t]{\wdff\linewidth}
		\centering
		\includegraphics[width=0.9\textwidth]{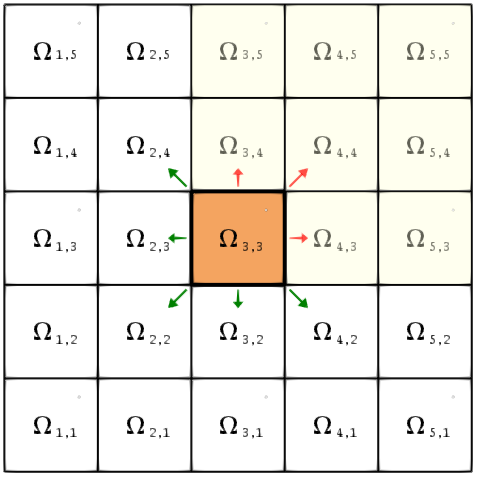}\\
		(a) First sweep: step 5 
	\end{minipage}
	\begin{minipage}[t]{\wdff\linewidth}
		\centering
		\includegraphics[width=0.9\textwidth]{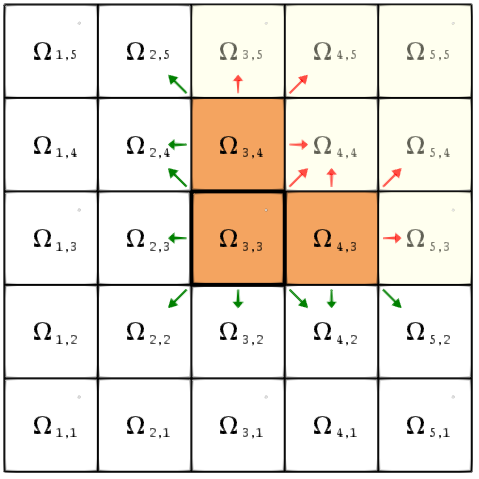}\\
		(b) First sweep: step 6
	\end{minipage}
	\begin{minipage}[t]{\wdff\linewidth}
		\centering
		\includegraphics[width=0.9\textwidth]{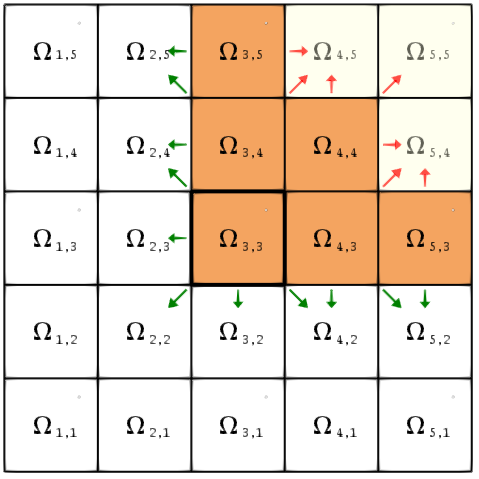}\\
		(c) First sweep: step 7 
	\end{minipage}
	
	\vspace{0.25cm}
	\begin{minipage}[t]{\wdff\linewidth}
		\centering
		\includegraphics[width=0.9\textwidth]{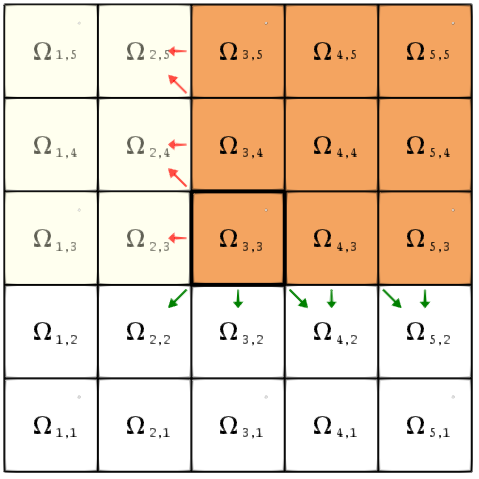}\\
		(d) After first sweep 
	\end{minipage}
	\begin{minipage}[t]{\wdff\linewidth}
		\centering
		\includegraphics[width=0.9\textwidth]{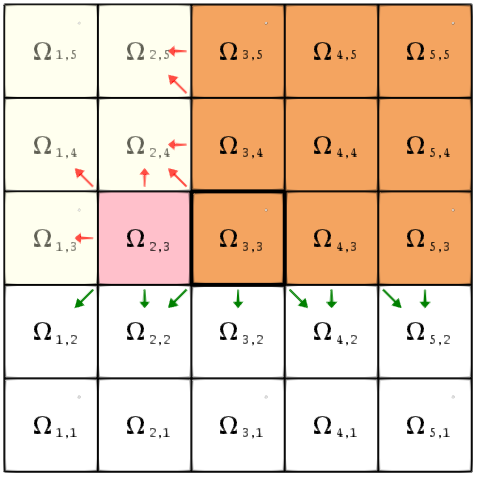}\\
		(e) Second sweep: step 6 
	\end{minipage}
	\begin{minipage}[t]{\wdff\linewidth}	
		\centering
		\includegraphics[width=0.9\textwidth]{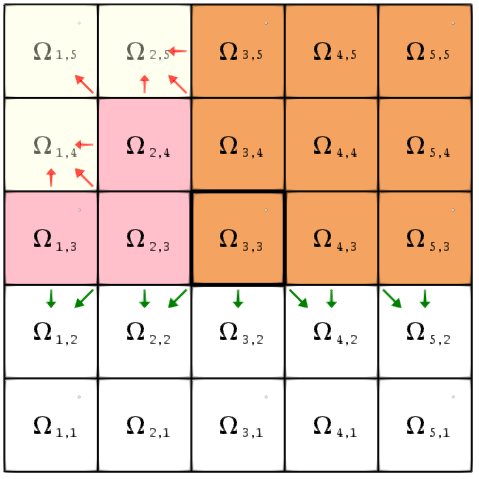}\\
		(f) Second sweep: step 7 
	\end{minipage}
	\caption{The first sweep $(+1,+1)$ and the second sweep $(-1,+1)$ in the diagonal sweeping DDM in $\R^2$. 
		The arrows denote the  transferred sources with their directions, the red ones are in the similar direction to the current sweep, 
		while the green ones are not (the green ones are excluded from being used in the current sweep due to Rule \ref{rule2d_a}).  
		 \label{fig:sweep2d_12}}	
\end{figure*}

In the {\bf second sweep}, the direction of sweep is chosen to be $(-1,+1)$, which aims at constructing the solution in the upper-left quadrant  $\Omega_{1,i_0-1;j_0,\Nby}={\Omega_{1,2;\,  3,5}}$.
Note that the $s$-th step of this sweep handles the group of subdomains $\{\Omega_{i,j}\}$ with $(\Nbx-i)+(j-1)+1=5-i+j=s$. 
 According to Rule $\ref{rule2d_a}$, among all the transferred sources left from the previous sweep (i.e., the first sweep), the ones with directions  $(-1,0)$, $(0,+1)$ and $(-1,+1)$ will  be used in this sweep, since they are in the similar direction to the current sweeping direction $(-1,+1)$. There is nothing to solve at the first 
 $(\Nbx-(i_0-1))+(j_0-1)=5$ steps of this sweep. At step 6, the subdomain problem in $\Omega_{2,3}$ is solved with the leftward transferred source from $\Omega_{3,3}$, and 5 transferred sources are generated and passed to its neighbor subdomains as  shown in Figure \ref{fig:sweep2d_12}-(e). 
At step 7, the subdomain problems on $\Omega_{1,3}$ and $\Omega_{2,4}$ are solved as shown Figure \ref{fig:sweep2d_12}-(f), and so on for the following steps, and after step 9 the solution is constructed in the upper-left quadrant  $\Omega_{1,2; 3,5}$, again leaving some transferred sources for future sweeps.
It is found that the subdomains that need  leftwards transfer solving are handled in the second sweep, while the subdomains that need  upwards transfer solving have already been handled in the first sweep.

In the {\bf third sweep}, the direction of sweep is chosen to be $(+1,-1)$, which aims at constructing the solution in the lower-right quadrant  $\Omega_{i_0,\Nbx;1,j_0-1}={\Omega_{3,5;\,  1,2}}$ as  shown in Figure \ref{fig:sweep2d_34}-(a). Note that the $s$-th step of this sweep handles the group of subdomains $\{\Omega_{i,j}\}$ with $(i-1)+(\Nby-j)+1=5+i-j=s$. 
The transferred sources from the upper-right quadrant ${\Omega_{3,5;\,  3,5}}$ are needed, 
while the transferred sources from the upper-left quadrant ${\Omega_{1,2;\,  3,5}}$ should be excluded,
 thus we need to introduce one more rule for the  source transfer in sweeps.
 Note that the new rule should not make decisions for transferred sources based on the relative position with respect to $\Omega_{3,3}$, otherwise the method is only valid for this special case of the source lying within only  $\Omega_{3,3}$.
 The second rule on the  transferred source  in sweeps  in $\R^2$ is defined as follows:
	\begin{rulee} {\rm (Opposite directions in $\R^2$)} \label{rule2d_b}
		The horizontal or vertical transferred source generated in one sweep  should not be used
		in a later sweep if these two sweeps have opposite directions.
	\end{rulee}
	
Such a rule  affects neither the transferred sources in the previous two sweeps nor the  transferred sources from the upper-right quadrant ${\Omega_{3,5;\,  3,5}}$ in the third sweep,    but effectively prevent the  transferred sources from the upper-left quadrant ${\Omega_{1,2;\,  3,5}}$  to enter the third sweep since they are generated in the second sweep, which has the opposite direction to the third sweep. There is nothing to solve at the first 
 $(i_0-1)+(\Nby-(j_0-1))=5$ steps of the third sweep.
At step 6, the subdomain problem in $\Omega_{3,2}$ is solved with the downward transferred source from $\Omega_{3,3}$, and  5 new transferred sources are generated and  passed to its neighbor subdomains as  shown Figure \ref{fig:sweep2d_34}-(b). 
 At step 7 of the third  sweep, the subdomain problems in $\Omega_{3,1}$ and $\Omega_{4,2}$ are solved as shown Figure \ref{fig:sweep2d_34}-(c), and so on for the following steps, and after step 9, the solution is constructed in the lower-right quadrant $\Omega_{3,5; 1,2}$, leaving a few transferred sources to be used in the fourth 
 sweep.
  It is found that now the subdomains that need  downwards transfer solving are handled in this sweep, and there are no more subdomains that need horizontal or vertical transfer solving.

In the {\bf  fourth sweep} (also the last), the sweep with the direction $(-1,-1)$ is performed. Note that the $s$-th step of this sweep handles the group of subdomains $\{\Omega_{i,j}\}$ with $(\Nbx-i)+(\Nby-j)+1=11-i-j=s$. 
Now all the transferred sources left from previous sweeps are in the similar direction to this sweep (none of horizontal or vertical ones are from the first sweep) as  shown in Figure \ref{fig:sweep2d_34}-(d),  thus according to  Rules \ref{rule2d_a} and \ref{rule2d_b}, all of them will be used in the last sweep. There is nothing to solve at the first  $(\Nbx-(i_0-1))+(\Nby-(j_0-1))=6$ steps of this sweep.
At step 7 of the fourth sweep, the subdomain problem in $\Omega_{2,2}$ is solved as  shown in Figure \ref{fig:sweep2d_34}-(e), and so on for the following steps. After step 9, the solution is constructed in the lower-left quadrant $\Omega_{1,i_0-1;1,j_0-1}=\Omega_{1,2; 1,2}$. Finally after such four diagonal sweeps with directions $(+1,+1)$,  $(-1,+1)$,  $(+1,-1)$ and $(-1,-1)$, the solution in the whole domain is constructed,  as shown in Figure \ref{fig:sweep2d_34}-(f).

\def\wdff{0.325}
\begin{figure*}[ht!]
	\centering
	\begin{minipage}[t]{\wdff\linewidth}
		\centering
		\includegraphics[width=0.9\textwidth]{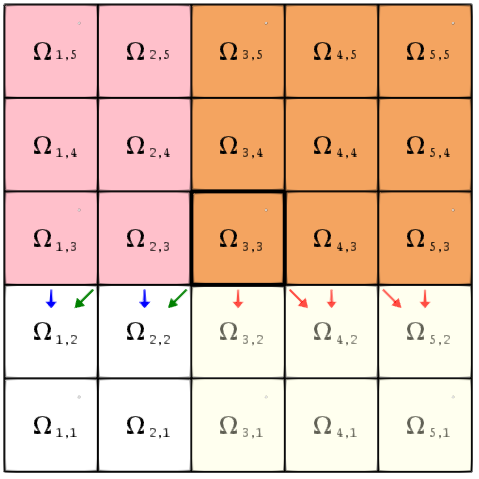}\\
		(a) Before third sweep 
	\end{minipage}
	\begin{minipage}[t]{\wdff\linewidth}
		\centering
		\includegraphics[width=0.9\textwidth]{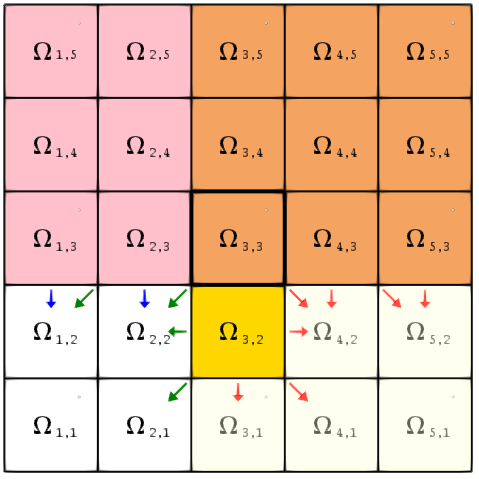}\\
		(b) Third sweep: step 6 
	\end{minipage}
	\begin{minipage}[t]{\wdff\linewidth}
		\centering
		\includegraphics[width=0.9\textwidth]{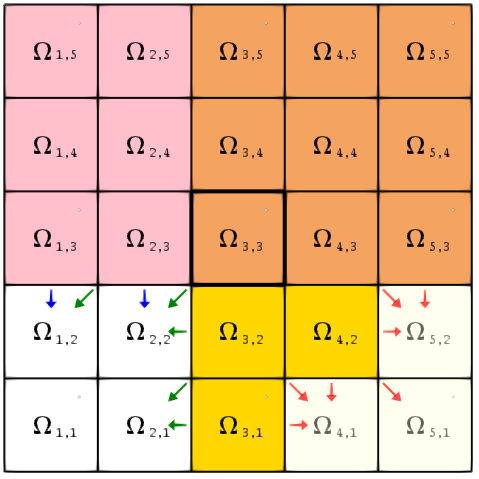}\\
		(c) Third sweep: step 7 
	\end{minipage}
	
	\vspace{0.25cm}
	\begin{minipage}[t]{\wdff\linewidth}
		\centering
		\includegraphics[width=0.9\textwidth]{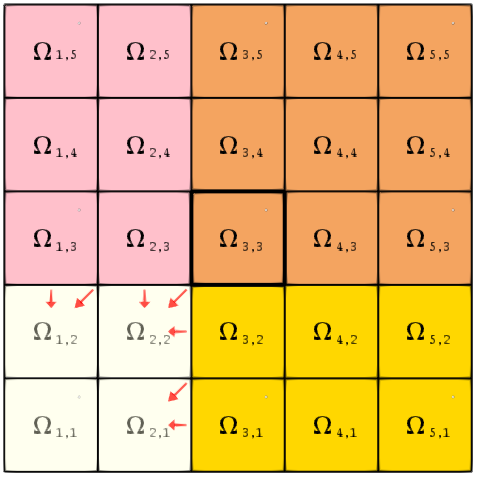}\\
		(d) After third sweep 
	\end{minipage}
	\begin{minipage}[t]{\wdff\linewidth}
		\centering
		\includegraphics[width=0.9\textwidth]{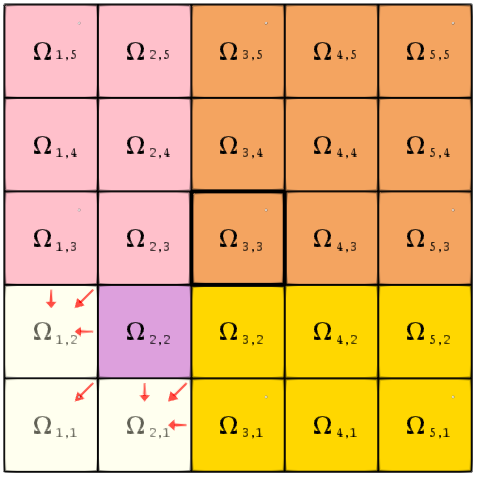}\\
		(e) Fourth sweep:  step 7 
	\end{minipage}
	\begin{minipage}[t]{\wdff\linewidth}	
		\centering
		\includegraphics[width=0.9\textwidth]{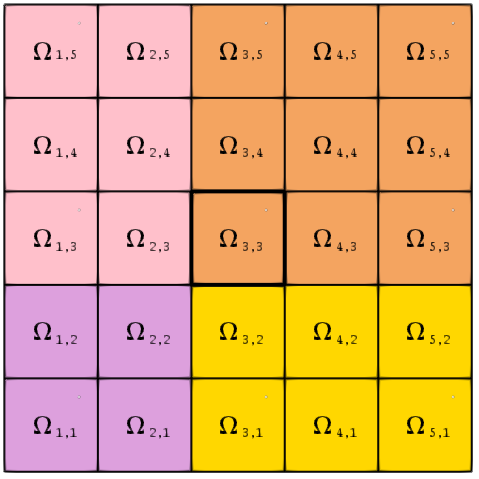}\\
		(f) After fourth sweep 
	\end{minipage}
	\caption{The third sweep $(+1,-1)$ and the fourth sweep $(-1,-1)$ in the diagonal sweep DDM in $\R^2$. 
		The arrows denote the  transferred sources with their directions, both the red and blue ones are in the similar direction to the current sweep
		(however the blues ones are excluded from being used in the current sweep due to Rule \ref{rule2d_b}), and the green ones are not (the green ones are excluded from being used in the current sweep due to Rule \ref{rule2d_a}).	
		 \label{fig:sweep2d_34} }	
\end{figure*}

By generalizing the above algorithm to $\Nbx \times \Nby$ subdomains and general source, we propose the following diagonal sweeping DDM with source transfer in $\R^2$:

\begin{algorithm_}[Diagonal sweeping DDM with source transfer in $\R^2$]~ 	
	\label{alg:diag2D}
	\begin{algorithmic}[1]
	\State Set the sweep order as $(+1,+1)$,  $(-1,+1)$,  $(+1,-1)$, $(-1,-1)$.
	\parState {Set the local subdomain sources 
		at each sweep as
		$r_{i,j}^{1} = f_{i,j}$, and $r_{i,j}^{l} = 0$, for $l=2,3,4$, $i=1,2,\ldots,\Nbx,\; j =1,2,\ldots,\Nby$.}
	\For{Sweep $l = 1,2,\ldots,4$} 
	\For{Step $s = 1, 2,\ldots, \Nbx+\Nby-1$ } 
	\For {each subdomain $\Omega_{i,j}$ in Step $s$ of Sweep $l$} 	
	\parState{
	 Solve the local solution $u_{i, j}^{l}$ with the local source $r^{l}_{i,j}$}
	\begin{equation}
	\L_{i,j} (u_{i, j}^{l}) =  r^{l}_{i,j},
	\end{equation}
	\For {each direction $(\Box,\vartriangle)$ that $\Box, \vartriangle = \pm 1,0$  and $(\Box, \vartriangle) \neq (0,0)$}
	\State Generate the new  source $\Psi_{\Box,\vartriangle; i, j} (u_{i, j}^{l})$ needed to be transferred;
	\parState{Find the smallest sweep number $l'\geq l$, such that the transferred source  $\Psi_{\Box,\vartriangle; i, j} (u_{i, j}^{l})$ could be used in Sweep $l'$, according to Rules \ref{rule2d_a} and \ref{rule2d_b};}
	\parState {Add the transferred source to the $l'$-th local source of the corresponding  neighbor subdomain}
	\begin{equation}
	r^{l'}_{i+\Box, j+\vartriangle} \mathrel{{=}}   r^{l'}_{i+\Box, j+\vartriangle}+\Psi_{\Box,\vartriangle; i, j} (u_{i, j}^{l}).
	\end{equation}
	\EndFor
	\EndFor
	\EndFor
	\EndFor
	\State The DDM solution for  $\P_{\Omega}$ with the source $f$ is then given by
	\begin{equation}\label{eq:sumdiag2D}
	 u_{\text{DDM}} = \sum\limits_{l = 1, \ldots, 4 } 
	\sum\limits_{\substack{\rangeTwoDRow}} \beta_{0,0; i,j} u_{i,j}^{l}.
	\end{equation}
    \end{algorithmic}
\end{algorithm_}
\vspace{0.2cm}

It is then easy to deduce the following result based on  similar process for the $5\times5$ partition. 
\begin{thm}{} \label{thm:sweep2d}
	The DDM solution $u_{\text{DDM}}$ produced by Algorithm \ref{alg:diag2D} is indeed the solution of the problem $P_{\Omega}$ in $\R^2$ in the constant medium case.
\end{thm}

\section{The diagonal sweeping DDM with source transfer in $\R^3$}

The diagonal sweeping DDM  in $\R^2$  (Algorithm \ref{alg:diag2D}) can be further extended to $\R^3$ based on the additive overlapping DDM (Algorithm \ref{alg:add3D}) in $\R^3$. 
 There are totally $2^3=8$ diagonal directions in $\R^3$: $(+1,+1,+1)$, $(-1,+1,+1)$, $(+1,-1,+1)$, $(-1,-1,+1)$, $(+1,+1,-1)$, $(-1,+1,-1)$, $(+1,-1,-1)$, $(-1,-1,-1)$, 
 and the sweep along each of the directions contains a total of $\Nbx+\Nby+\Nbz-2$ steps.

\subsection{Sweeping  orders, source transfer rules and sweeping algorithm}
We choose to use the following sweeping order for our diagonal sweeping DDM in this paper,
which could be viewed as the two-dimensional sweeping order with first the positive $z$ direction and then the negative one:
 \begin{align} \label{Sorder3D}
 \begin{array}{llll}
 (+1,+1,+1),& (-1,+1,+1),& (+1,-1,+1),& (-1,-1,+1), \\
 (+1,+1,-1),& (-1,+1,-1),& (+1,-1,-1),& (-1,-1,-1).
 \end{array}
 \end{align}
Other sweeping order also exists, such as 
\begin{align*}
\begin{array}{llll}
(+1,+1,+1),& (-1,+1,+1),& (+1,-1,+1),& (+1,+1,-1),\\
(-1,-1,+1),& (-1,+1,-1),& (+1,-1,-1),& (-1,-1,-1),
\end{array}
\end{align*}
where the $L_1$ distance between the  successive sweeping directions and  the first one is monotonically increasing. 

Let us first define  the similar direction  in $\R^3$. Two vector $\d_1$ and $\d_2$ in $\R^3$ are called in the {\sl similar
direction} if $\d_1 \cdot \d_2 > 0$ and $\d_1(k) \, \d_2(k) \geq 0$ for
$k = 1,2,3$, where $\d_1(k)$ and $\d_2(k)$ are the $k$-th components of
$\d_1$ and $\d_2$, respectively. 
Then the first rule on the transferred source in sweeps in $\R^3$ (in correspondence to Rule \ref{rule2d_a} in $\R^2$)  is defined below:
\begin{rulee}{\rm (Similar directions in $\R^3$)} \label{rule3d_a}
A transferred source  which is not in the similar direction of one sweep in $\R^3$ should not be used in that sweep.
\end{rulee}

 Note that by projection onto two-dimensional planes,  the three-dimensional construction of the solution becomes the two-dimensional quadrant-wise construction of the solution, thus we follow  Rule \ref{rule2d_b} for $\R^2$, and define the second rule on the transferred source in sweeps in $\R^3$ as follows:
\begin{rulee}{\rm (Opposite directions in $\R^3$)} 	\label{rule3d_b}
 Suppose a transferred source with direction $\d_{\text{src}}$ is 
 generated in one sweep with direction $\d_1$, then it should not be used in the later sweep with direction $\d_2$, if under any of $x-y$, $x-z$,
  $y-z$ plane projections, the projection of $\d_{\text{src}}$ has exactly one zero component and the projections of $\d_1$ and $\d_2$ are opposite.
 \end{rulee}

Now we propose the diagonal sweeping DDM with source transfer in $\R^3$ in the following:

\begin{algorithm_}[Diagonal sweeping DDM with source transfer in $\R^3$]
    $\quad$  
	\label{alg:diag3D}
	\begin{algorithmic}[1]
		\parState {Set the sweep order as list (\ref{Sorder3D}) } 
		\parState {Set the local subdomain sources for each sweep as
				$r_{i,j,k}^{1} = f_{i,j,k}$, and $r_{i,j,k}^{l} = 0$, for $l=2,3,\ldots,8$, $i=1,2,\ldots,\Nbx,\; j =1,2,\ldots,\Nby, \; k =1,2,\ldots,\Nbz$.}
		\For{Sweep $l = 1,\ldots,8$} 
		\For{Step $s = 1, \ldots, \Nbx+\Nby+\Nbz-2$ } 
		\For {subdomain $\Omega_{i,j,k}$ in Step $s$ of the current sweep} 	
		\parState{
			Solve the local solution $u_{i, j, k}^{l}$ with the local source of current sweep}
		\begin{eqnarray}
		\L_{i,j,k} (u_{i, j, k}^{l}) =  r^{l}_{i,j, k},
		\end{eqnarray} 
		\For {each direction $(\Box,\vartriangle,\ocircle)$ that $\Box, \vartriangle, \ocircle = \pm 1,0$  and $(\Box, \vartriangle, \ocircle) \neq (0,0,0)$ }
		\State Compute new transferred source $\Psi_{\Box,\vartriangle,\ocircle; i, j,k} (u_{i, j, k}^{l})$;
		\parState{Find the smallest sweep number $l'\geq l$, such that the transferred source  $\Psi_{\Box,\vartriangle, \ocircle; i, j, k} (u_{i, j, k}^{l})$ could be used in Sweep $l'$, according to Rules \ref{rule3d_a} and \ref{rule3d_b};} 
		\parState {Add the transferred source to the $l'$-th local source of the corresponding  neighbor subdomain}
		\begin{eqnarray}
		r^{l'}_{i+\Box, j+\vartriangle, k+\ocircle} \mathrel{{=}}  r^{l'}_{i+\Box, j+\vartriangle, k+\ocircle} +\Psi_{\Box,\vartriangle,\ocircle; i, j, k} (u_{i, j, k}^{l}). 
		\end{eqnarray}
		\EndFor
		\EndFor
		\EndFor
		\EndFor
		\State The DDM solution for  $\P_{\Omega}$ with the source $f$ is then given by
		\begin{equation}\label{eq:sumdiag3D}
		 u_{\text{DDM}} = \sum\limits_{l = 1, \ldots, 8} 
		\sum\limits_{\substack{\rangeThreeDRow}} \beta_{0,0,0; i,j,k}  u_{i,j,k}^{l} .
		\end{equation}
	\end{algorithmic}
\end{algorithm_}
\vspace{0.2cm}

 \subsection{Verification of the DDM solution}
 
Next we verify that the DDM solution $u_{\text{DDM}}$ produced by Algorithm \ref{alg:diag3D} is indeed the solution to the problem $P_{\Omega}$ in $\R^3$ in the constant medium case. Again the case of the source lying within only one subdomain is verified, for instance, $\supp f \subset \Omega_{i_0,j_0,k_0}$, then the case of general source follows if the solving process does not depend on such specific  subdomain. Let us call $\Omega_{i_0,j_0,k_0}$ the origin subdomain.

\def\wdff{0.325}
\begin{figure*}[!ht]	
	\centering
	\begin{minipage}[t]{\wdff\linewidth}
		\centering
		\includegraphics[width=0.9\textwidth]{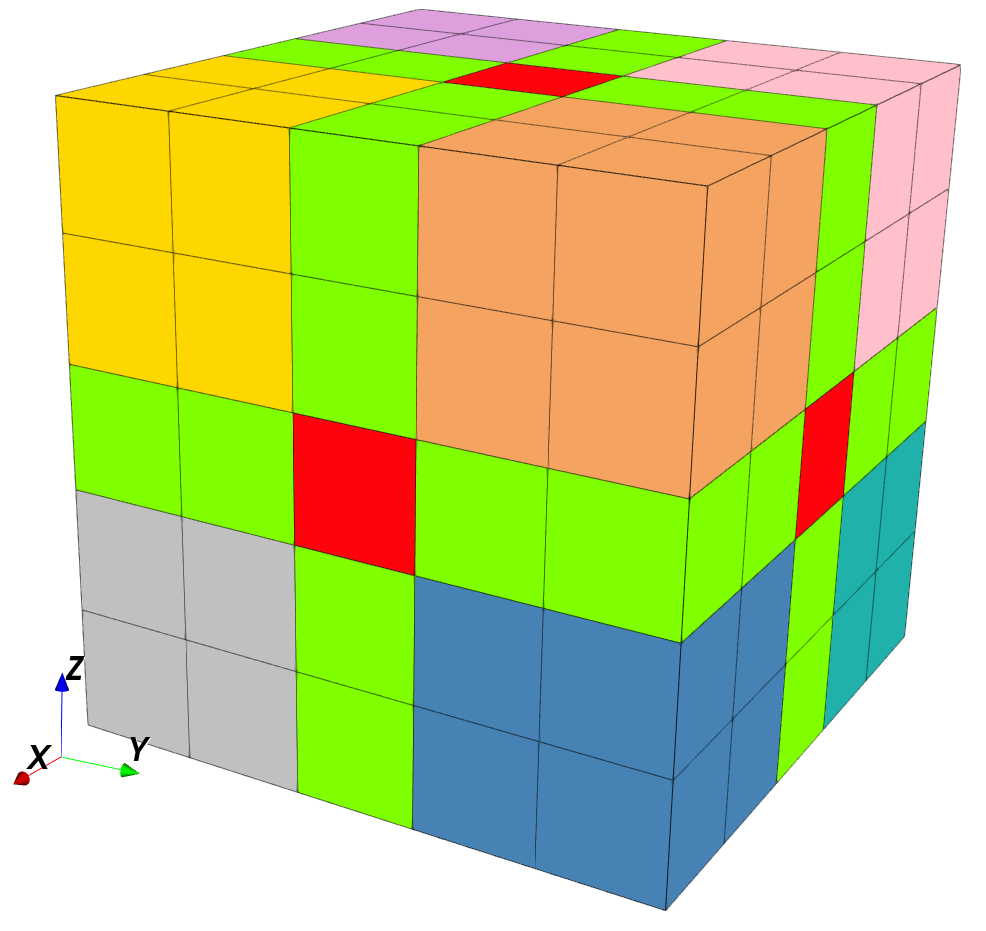}\\
		(a) All 27 regions
	\end{minipage}
	\begin{minipage}[t]{\wdff\linewidth}
		\centering
		\includegraphics[width=0.9\textwidth]{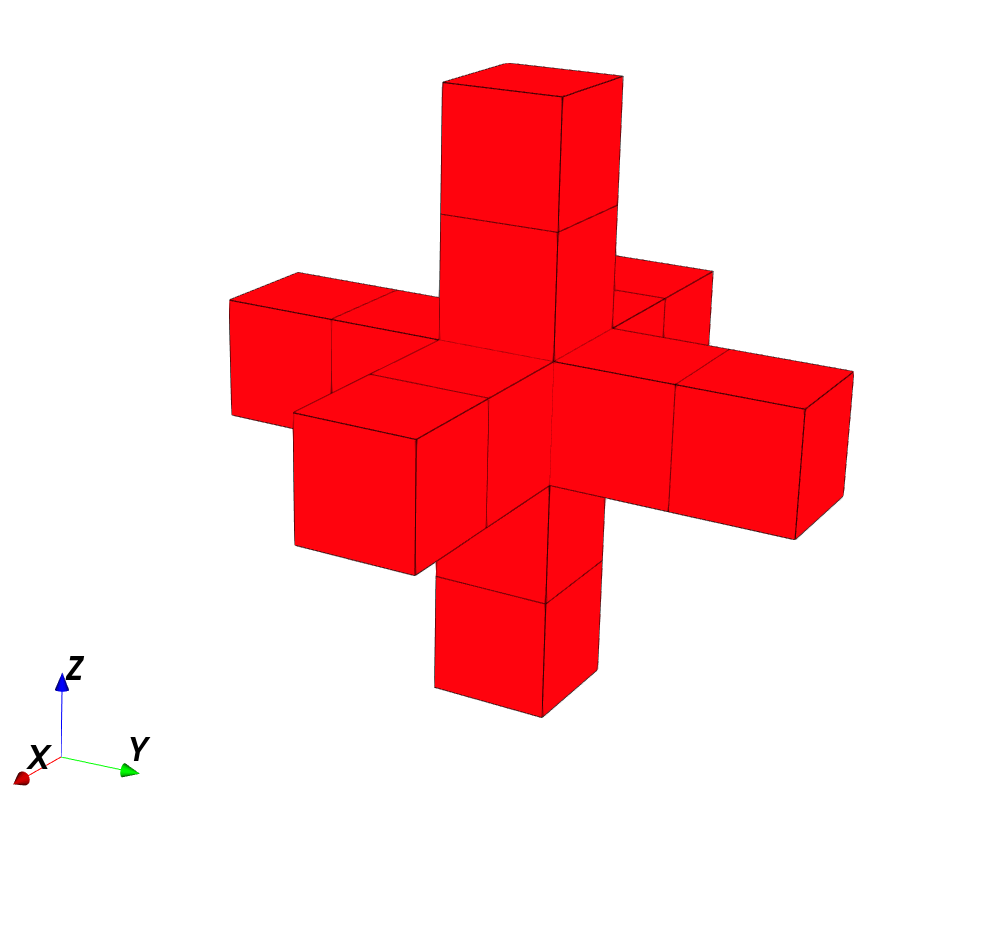}\\
		(b) 6 axial regions
	\end{minipage}
	\begin{minipage}[t]{\wdff\linewidth}
		\centering
		\includegraphics[width=0.9\textwidth]{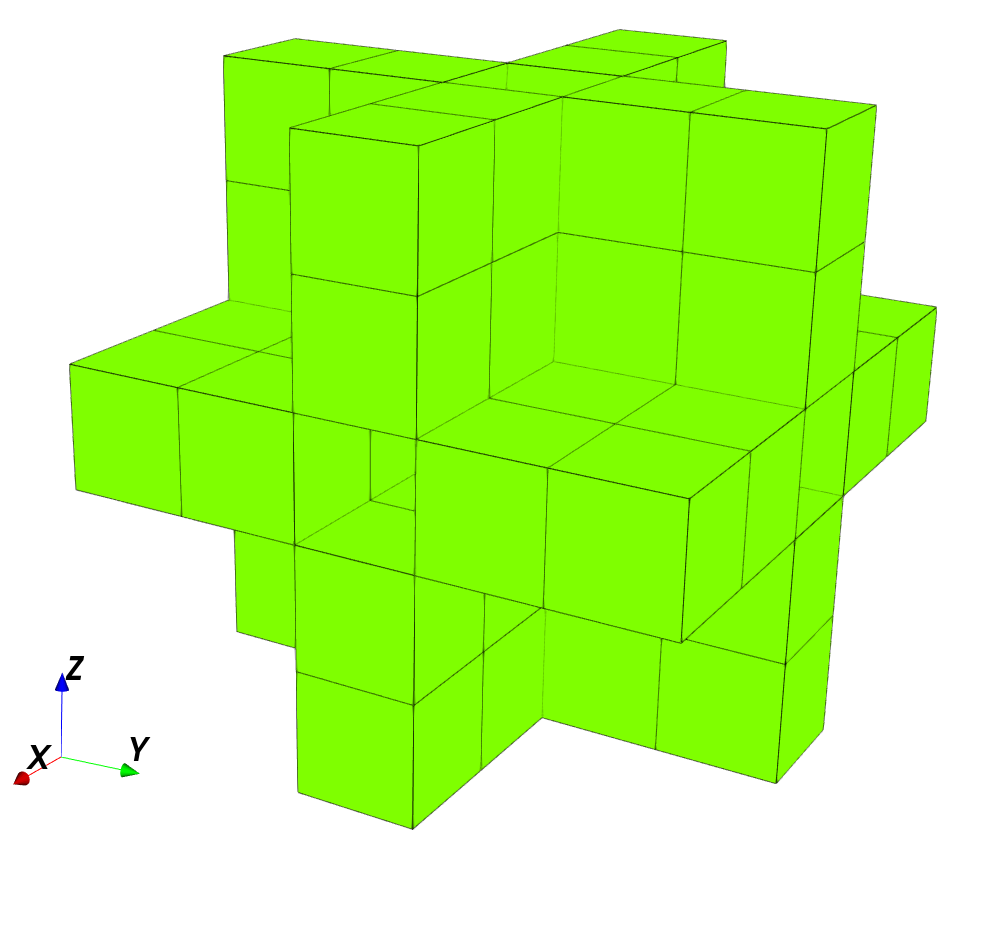}\\
		(c) 12 planar regions
	\end{minipage}
	
	\vspace{0.3cm}
	\begin{minipage}[t]{\wdff\linewidth}
		\centering
		\includegraphics[width=0.9\textwidth]{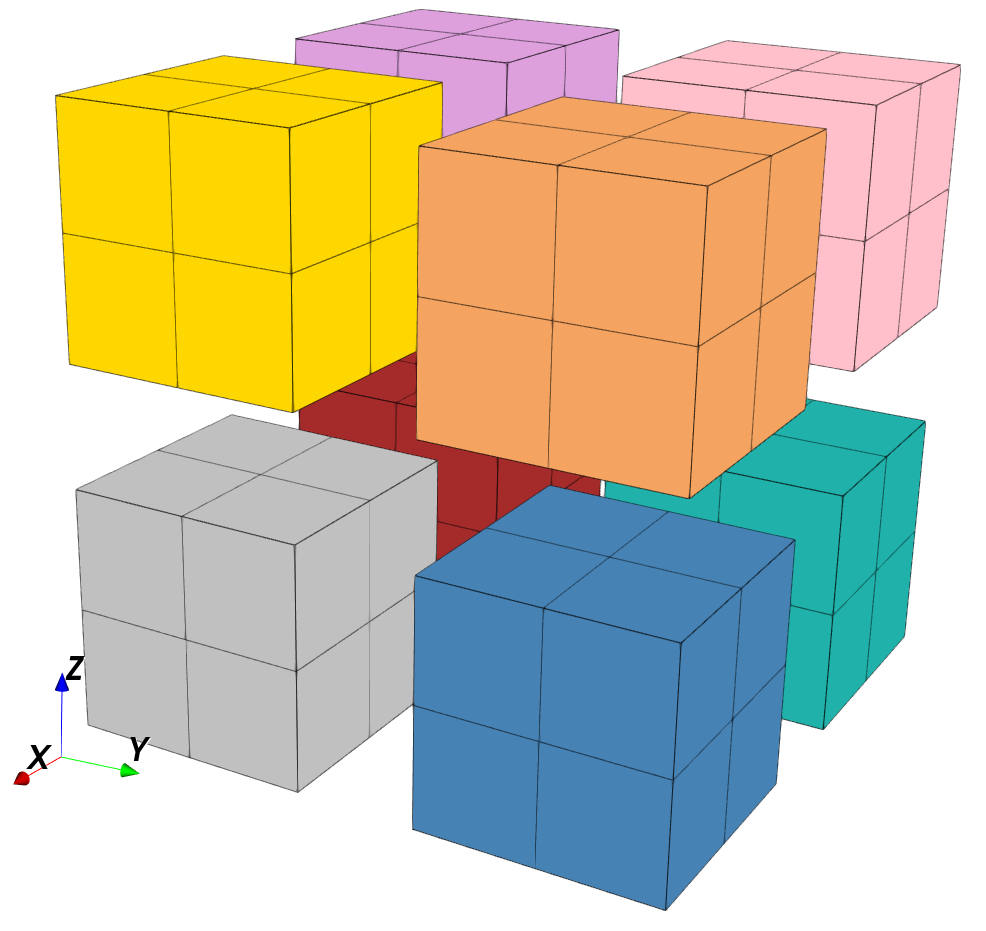}\\
		(d) 8 octantal regions
	\end{minipage}
	\begin{minipage}[t]{\wdff\linewidth}
		\centering
		\includegraphics[width=0.9\textwidth]{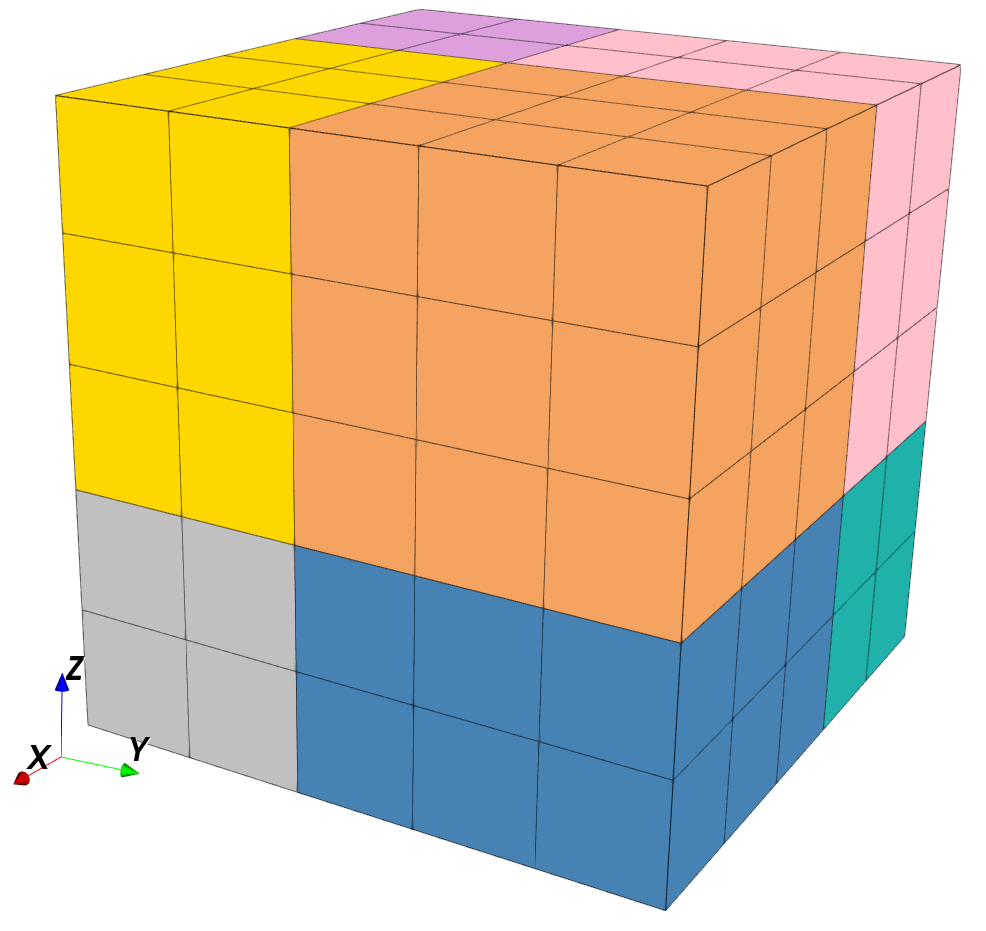}\\
		(e) 8 extended octantal regions
	\end{minipage}
	\caption{Divided regions based on the relative position to the particular subdomain $\Omega_{i_0,j_0,k_0}$.  \label{fig:oct_div}}
\end{figure*}

In each sweep of the sweeping diagonal DDM, the local solution of some subdomains in certain region is to be constructed and we describe these regions  in the following. 
The whole domain is split into $3^3=27$ regions
based on the relative position  to the specific subdomain $\Omega_{i_0, j_0, k_0}$, 
as  shown in Figure \ref{fig:oct_div}-(a),
we denote them by $\Omega^{(\dirThreeD)}$, $\dirThreeD = \pm 1, 0$,  
\begin{align}
\Omega^{(\dirThreeD)} = \bigcup\limits_{\substack{i \in I_{\Box}(i_0),\, 
		j \in I_{\vartriangle}(j_0),\, k \in I_{\ocircle}(k_0) \\ \rangeThreeDCol  } } \Omega_{i,j,k} 
\end{align}
where $I_{s}(a)$ is a set with $I_{1}(a)=\{a+1,a+2,\ldots,+\infty\}$, 
$I_{0}(a)=\{a\}$, 
and $I_{-1}(a)=\{-\infty, \dots, a-2,a-1\}$.
These regions could be divided into four types: 
\begin{itemize}
	\item the origin one $(\dirThreeD)=(0,0,0)$,
	which contains the source;
	\item 6 axial ones with exactly two zeros in $(\dirThreeD)$ as shown in Figure \ref{fig:oct_div}-(b),
	in which the subdomains are solved with $x$, $y$ or $z$ directional source transfers;
	\item 12  planar ones with exactly one zeros in $(\dirThreeD)$ as shown in Figure \ref{fig:oct_div}-(c),
	in which the subdomains are solved with $x$-$y$, $y$-$z$ or $x$-$z$ directional source transfers;
	\item 8 octantal ones with no zeros in $(\dirThreeD)$ as shown in Figure \ref{fig:oct_div}-(d), 
	in which the subdomains are solved with $x$-$y$-$z$ directional source transfers.
\end{itemize}

The L-sweeps method \cite{Zepeda2019} constructs the solution in each of the 27 regions separately with 26 sweeps, however, our diagonal sweep DDM merges the origin, axial and  planar regions  into the octant regions, and constructs the solution in 8 extended octantal regions (shown in Figure \ref{fig:oct_div}-(e)) with 8 sweeps. 
Specifically, we denote the 8 extended octantal regions  by $\widetilde{\Omega}^{(\dirThreeD)} $, $\dirThreeD = \pm 1$ with $(\dirThreeD)$ being referred as the direction of the octants, and  we have
\begin{align} \label{octa3D}
\begin{array}{ll}
\widetilde{\Omega}^{(+1,+1,+1)} =  \Omega_{i_0,\Nbx;j_0,\Nby;k_0,\Nbz},\quad &\widetilde{\Omega}^{(-1,+1,+1)} =  \Omega_{1,i_0-1;j_0,\Nby;k_0,\Nbz}, \\
\widetilde{\Omega}^{(+1,-1,+1)} =  \Omega_{i_0,\Nbx;1,j_0-1;k_0,\Nbz},\quad &\widetilde{\Omega}^{(-1,-1,+1)} =  \Omega_{1,i_0-1;1,j_0-1;k_0,\Nbz}, \\
\widetilde{\Omega}^{(+1,+1,-1)} =  \Omega_{i_0,\Nbx;j_0,\Nby;1,k_0-1},\quad &\widetilde{\Omega}^{(-1,+1,-1)} =  \Omega_{1,i_0-1;j_0,\Nby;1,k_0-1}, \\
\widetilde{\Omega}^{(+1,-1,-1)} =  \Omega_{i_0,\Nbx;1,j_0-1;1,k_0-1},\quad &\widetilde{\Omega}^{(-1,-1,-1)} =  \Omega_{1,i_0-1;1,j_0-1;1,k_0-1}.
\end{array}
\end{align}
The extended octantal regions will be referred as the octants for short in the remaining part of the paper.
Each octant $\widetilde{\Omega}^{\dirThreeD}$ is to be solved in the sweep along the direction $(\dirThreeD)$.

The definition of  neighbor octants is introduced as follows.
The distance of two octants is measured by the half of $L_1$ distance of their directions, thus any octant has three distance-1  neighbor octants (or  face  neighbor octants),  three distance-2  neighbor octants (or edge  neighbor octants), and one distance-3  neighbor octant (or the opposite octant).
Before solving an octant in the current sweep, 
some other octants may have already been solved in the previous sweeps, then the octant to be solved in the current sweep may have zero, one, two or three solved face  neighbor octants, 
these are the four cases that we will encounter a few times in the solving process,
as shown in Figure \ref{fig:oct_solve_type}-(a) to (c).

\def\wdff{0.325}
\begin{figure*}[!ht]	
	\centering
	\begin{minipage}[t]{\wdff\linewidth}
		\centering
		\includegraphics[width=0.9\textwidth]{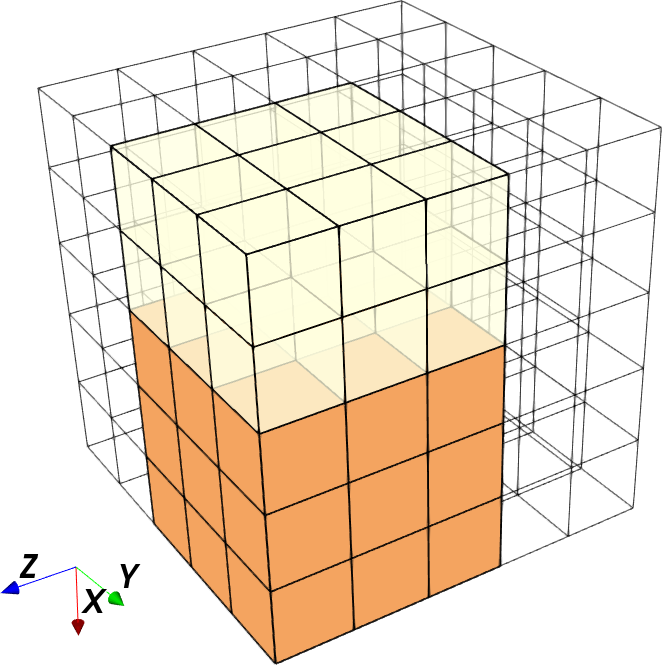}\\
		(a) One solved face neighbor octant
	\end{minipage}
	\begin{minipage}[t]{\wdff\linewidth}
		\centering
		\includegraphics[width=0.9\textwidth]{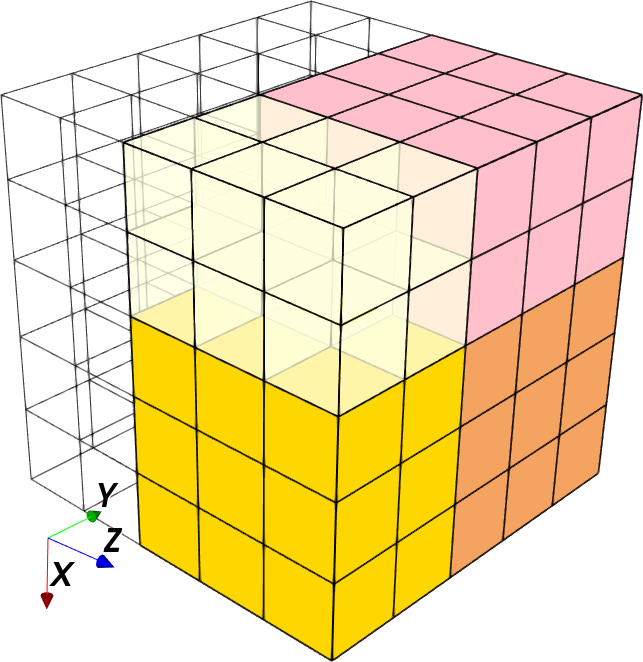}\\
		(b) Two solved face neighbor octants
	\end{minipage}
	\begin{minipage}[t]{\wdff\linewidth}
		\centering
		\includegraphics[width=0.9\textwidth]{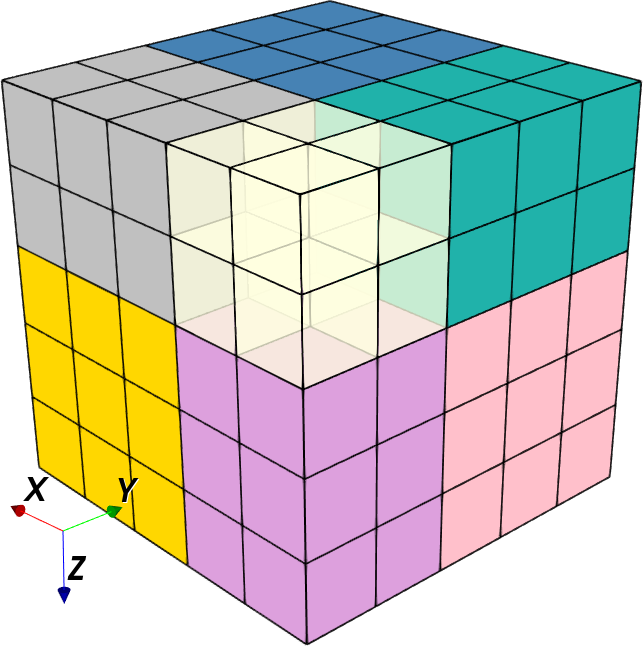}\\
		(c) Three solved face neighbor octants
	\end{minipage}
	\caption{
		The light yellow transparent region is the octant to be solved, and the other colored ones are the already solved octants. \label{fig:oct_solve_type}}\vspace{0.25cm}
\end{figure*}

 For all the eight diagonal sweeps in $\R^3$, the choice of transferred sources to be used in each sweep become the key problem and is quite complicated, thus we first discuss some basic properties of the octant-wise solving, and then develop  some useful tools for the verification using  Rules \ref{rule3d_a} and \ref{rule3d_b} on source transfer.
 To better describe the unused transferred sources generated from an octant solving in the corresponding sweep, we categorize them by the faces, edges, and vertices of the octant as follows.
The  unused transferred sources associated with a face of the octant are defined as the unused transferred sources that are generated by the boundary subdomains of the octant and have the similar direction to the octant face, see Figure \ref{fig:tsrc_type}-(a) for an illustration.
The  unused transferred sources associated with an edge of the octant  are defined as the intersection of the unused transferred sources associated with the two faces sharing the  edge, see Figure \ref{fig:tsrc_type}-(b).
The  unused transferred sources associated with a vertex of an octant are defined as the intersection of the unused transferred sources associated with all three faces,  see Figure \ref{fig:tsrc_type}-(c).

\def\wdff{0.325}
\begin{figure*}[!ht]	
	\centering
	\begin{minipage}[t]{\wdff\linewidth}
		\centering
		\includegraphics[width=0.9\textwidth]{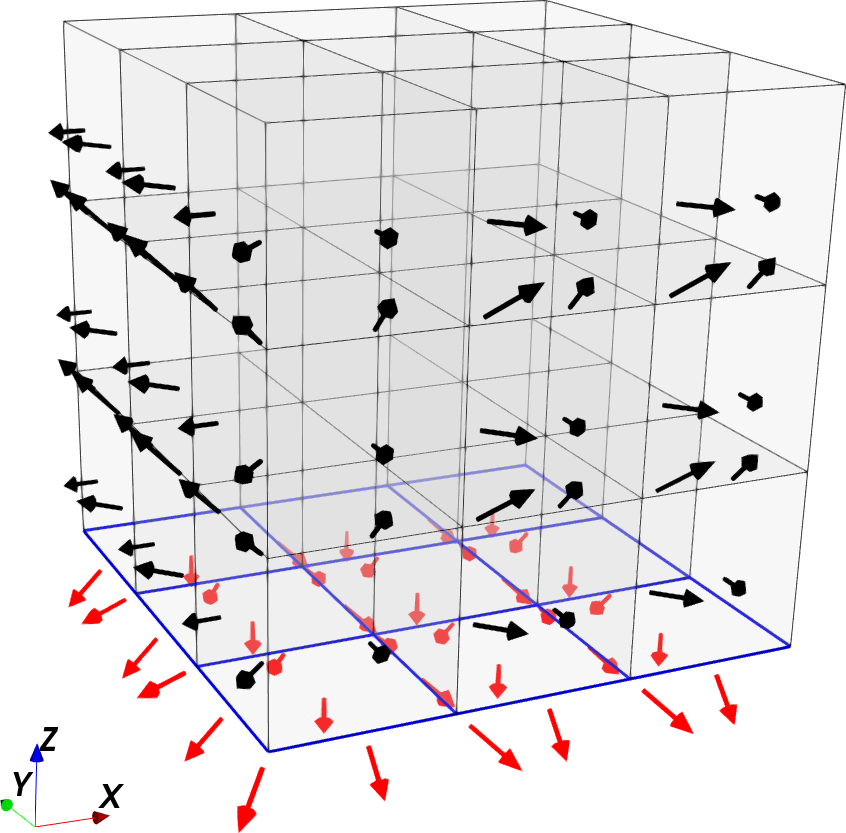}\\
		(a) 
	\end{minipage}
	\begin{minipage}[t]{\wdff\linewidth}
		\centering
		\includegraphics[width=0.9\textwidth]{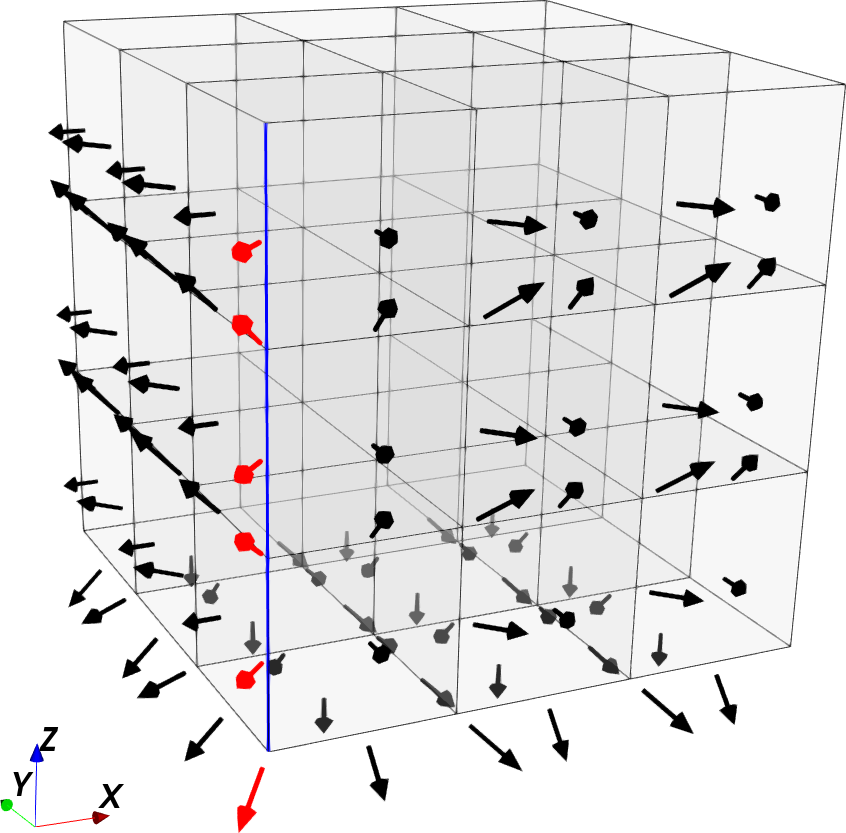}\\
		(b) 
	\end{minipage}
	\begin{minipage}[t]{\wdff\linewidth}
		\centering
		\includegraphics[width=0.9\textwidth]{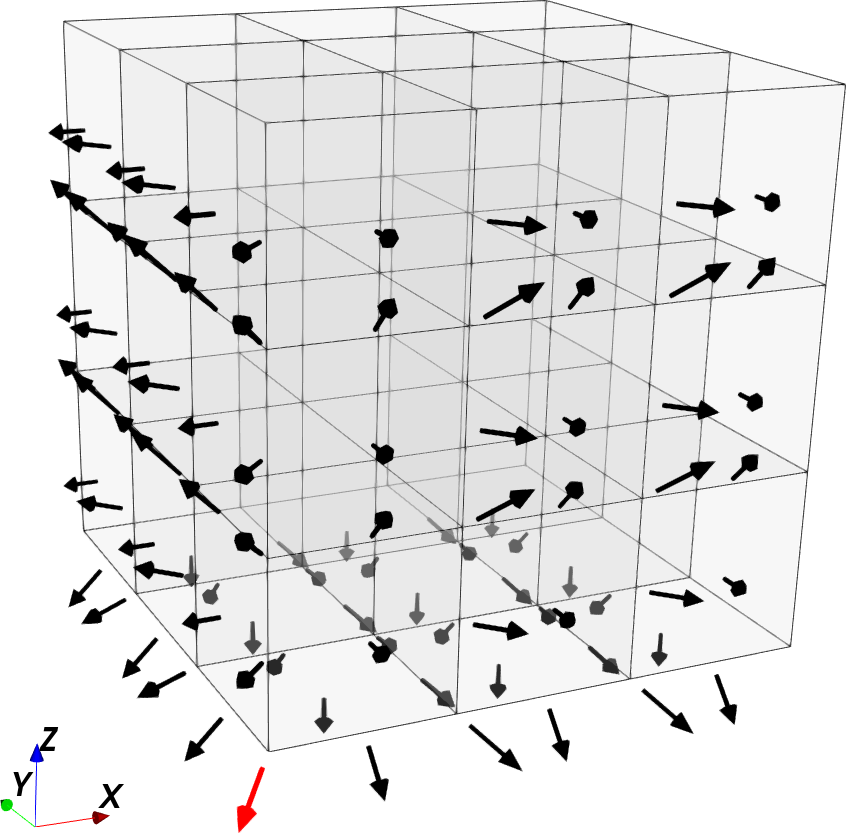}\\
		(c) 
	\end{minipage}
	\caption{Illustration of the unused transferred sources associated with a face (a), or an edge (b), or a vertex (c) of one octant, which is denoted by red arrows.
	The face and edge are marked with blue lines.   \label{fig:tsrc_type}}
\end{figure*}

It is obvious that  the unused transferred sources needed to solve an octant in the corresponding sweep must be  in the similar direction to the octant, and we will refer  the unused transferred sources in the similar direction to the octant as the \textit{candidate} transferred sources for the octant to be solved. Note that  not all the candidate  transferred sources are needed to  solve an octant,  the following result holds.

\begin{lemmaa} \label{lemma:candidate}
	Suppose that an octant is to be solved in the corresponding sweep, then all the candidate transferred sources it needs are those 
	associated with the sharing faces of the solved face  neighbor octants,
	or associated with the sharing edges of the solved edge  neighbor octants,
	or associated with the sharing vertices of the solved opposite octants.
\end{lemmaa}	 

Lemma \ref{lemma:candidate}  presents the requirement that we need to verify during the  sweeps of the diagonal sweeping DDM in $\R^3$. To simplify the verification of Algorithm \ref{alg:diag3D}, 
a few tools are introduced below. 
\begin{lemmaa}{\rm (Shared face in $\R^3$)} \label{lemma:I}
	When solving an octant, the candidate transferred sources associated with the shared face of its solved face  neighbor octant
	are successfully selected in Algorithm \ref{alg:diag3D}  according to Rule \ref{rule3d_a} and \ref{rule3d_b}.
\end{lemmaa}

\begin{proof}
Since the octant and its face  neighbor octant could have only one opposite component,  the directions of the two octants aren't opposite under any of $x$-$y$,  $y$-$z$ and $x$-$z$ plane projection (in which two  opposite components are required). Consequently, Rule \ref{rule3d_b} for the  opposite direction  doesn't apply and these candidate transferred sources will not be excluded.
\end{proof}

\def\wdff{0.4}
\begin{figure*}[!ht]	
	\centering
	\begin{minipage}[t]{\wdff\linewidth}
		\centering
		\includegraphics[width=0.9\textwidth]{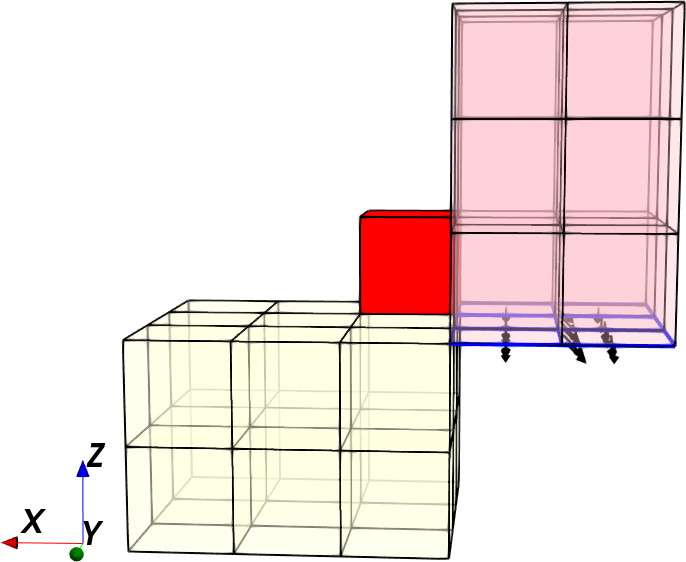}\\
		(a) 
	\end{minipage}
	\hspace{0.5cm}
	\begin{minipage}[t]{\wdff\linewidth}
		\centering
		\includegraphics[width=0.9\textwidth]{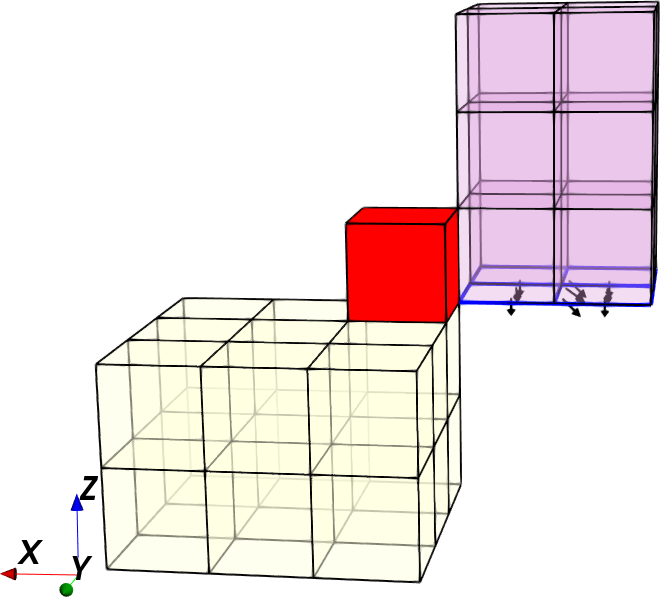}\\
		(b) 
	\end{minipage}
	\caption{Illustration of the case in Lemma \ref{lemma:II}.
		The light yellow transparent region is the octant to be solved, the red subdomain is the origin subdomain.
		The pink or purple region is the solved octant, the face to be checked is marked with blue frames, and the arrows 
		denote the unused transferred sources to be excluded from the current sweep. 
		 \label{fig:lemma_op}}
\end{figure*}

The following Lemma is used to check in  Algorithm \ref{alg:diag3D}  whether the candidate transferred sources associated with a face of a distance-2 or distance-3 solved octant are   excluded from the current sweep,
which is a very common situation.
\begin{lemmaa}{\rm (Nonadjacent face in $\R^3$)} \label{lemma:II}
	When solving an octant, 
	suppose under one of the $x$-$y$,  $y$-$z$ and $x$-$z$ plane projection, 
	both the octant to be solved and the origin subdomain are in the same half of a plane.  Under this plane projection, a  distance-2 or distance-3 solved octant is in the opposite position, one of its face is to be checked and both the solved octant and the face are in the other half of the plane. 
	Then the unused transferred sources associated with the to-be-checked face of the solved  octant,   
	will be excluded from this octant solving in Algorithm \ref{alg:diag3D} according to Rule \ref{rule3d_b}.	
\end{lemmaa}
\begin{proof}
The situation of the above Lemma is illustrated  in Figure \ref{fig:lemma_op}. Assume that
under the $x$-$z$ projection, the octant $\widetilde{\Omega}^{(+1,+1,-1)}$ (yellow) is to be solved, 
the solved octant is either
the distance-2 octant $\widetilde{\Omega}^{(-1,+1,+1)}$ (pink) in Figure \ref{fig:lemma_op}-(a)
or the  distance-3 octant $\widetilde{\Omega}^{(-1,-1,+1)}$ (purple) in Figure \ref{fig:lemma_op}-(b),
and the face to be checked has the outer normal $(0,0,-1)$. 
The negative $x$-half plane has the octant and the origin subdomain, while
 the positive $x$-half plane has the solved octant and the face.
Now suppose one of the candidate transferred sources associated with the to-be-checked face of the solved octant $\widetilde{\Omega}^{(-1,+1,+1)}$ (pink) or $\widetilde{\Omega}^{(-1,-1,+1)}$ (purple) 
has direction $(d_x, d_y, d_z)$.
The candidate transferred sources are associated with the face of the outer normal  $(0,0,-1)$, thus $d_z = -1$.
Since the  origin subdomain is in the  positive $x$ direction, we have $d_x \leq 0$.
If $d_x < 0$, the transferred sources will not be in the similar direction of the octant $\widetilde{\Omega}^{(+1,+1,-1)}$. 
If $d_x = 0$, then the direction of the transferred sources  will become 
$(0, \pm 1, -1)$, which is $(0,-1)$ under the $x$-$z$ projection. Since the octant $\widetilde{\Omega}^{(+1,+1,-1)}$ (light yellow) and the solved octant are in the opposite position under the $x$-$z$ projection, 
  Rule \ref{rule3d_b} applies and the candidate transferred sources associated with the to-be-checked face of the solved  octant  are excluded in Algorithm \ref{alg:diag3D}.
\end{proof}

\begin{lemmaa}{\rm (Shared edge in $\R^3$)}\label{lemma:III}
	When solving an octant that has two or three solved face  neighbor octants, the candidate transferred sources associated with the shared edge of its solved distance-2 neighbor octants
are successfully  selected in Algorithm \ref{alg:diag3D} according to Rules \ref{rule3d_a} and \ref{rule3d_b}.
\end{lemmaa}
\begin{proof}
	The octant and its solved distance-2 neighbor octant is only opposite under one plane projection, and under that plane projection, these candidate transferred sources will have two non-zero components, hence Rule \ref{rule3d_b} doesn't apply and these candidate transferred sources will not be excluded.
\end{proof}

\begin{lemmaa}{\rm (Shared vertex in $\R^3$)} \label{lemma:IV}
	When solving an octant that has three solved face  neighbor octants, the candidate transferred sources associated with the shared vertex of its  solved distance-3 neighbor octant
are successfully  selected in Algorithm \ref{alg:diag3D} according to Rules \ref{rule3d_a} and \ref{rule3d_b}.
\end{lemmaa}
\begin{proof}
This is  obvious since these candidate transferred sources have three non-zero components and Rule \ref{rule3d_b} doesn't apply at all.
\end{proof}

With the above results, the verification of Algorithm \ref{alg:diag3D} becomes much easier.  The candidate transferred sources associated with the sharing faces, edges and vertices have already been selected by Lemmas \ref{lemma:I}, \ref{lemma:III} and  \ref{lemma:IV}, and now the main concern is whether the  candidate transferred sources that aren't  listed in Lemma \ref{lemma:candidate} are excluded by Lemma \ref{lemma:II} with the nonadjacent face.
Without loss of generality, we take a $5\times 5\times 5$ ($\Nbx=\Nby=\Nbz=5$) domain partition
 to illustrate the solving process,  and assume the source lies in the subdomain $\Omega_{3,3,3}$ ($i_0=j_0=k_0=3$).  The sweep along each of the directions contains a total of $\Nbx+\Nby+\Nbz-2 = 5+5+5-2=13$ steps.
  We will illustrate the first sweep in details  and  the following sweeps  will be performed similarly;
in particular, we will discuss  and verify the choice of transferred sources to be used by each of the octant solves.

\def\wdff{0.325}
\begin{figure*}[!ht]	
	\centering
	\begin{minipage}[t]{\wdff\linewidth}
		\centering
		\includegraphics[width=0.9\textwidth]{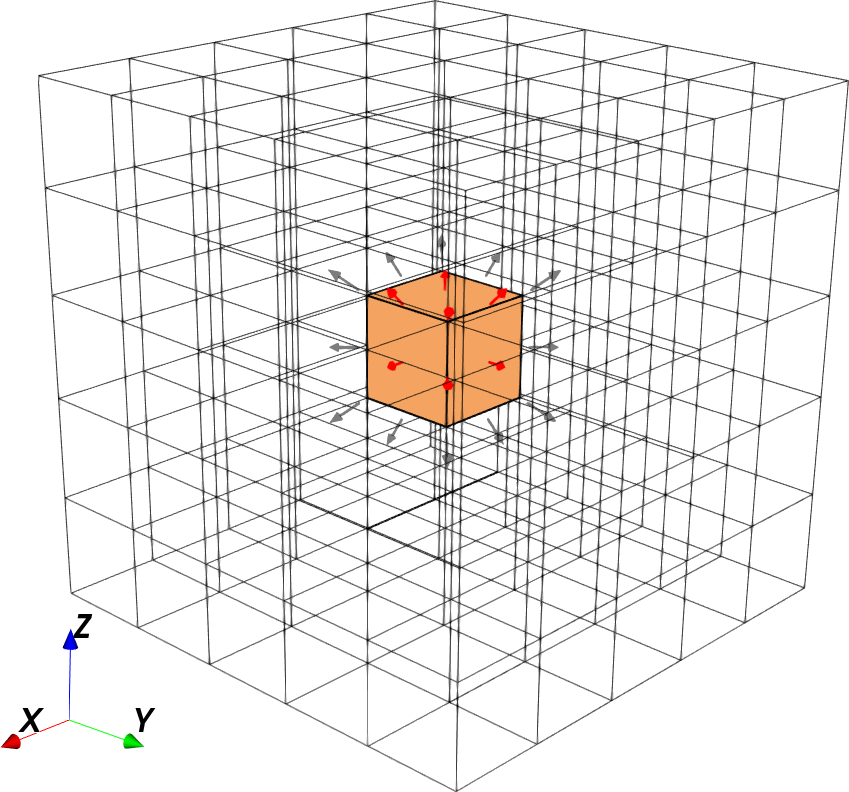}\\
		(a) First sweep: step 7 
	\end{minipage}
	\begin{minipage}[t]{\wdff\linewidth}
		\centering
		\includegraphics[width=0.9\textwidth]{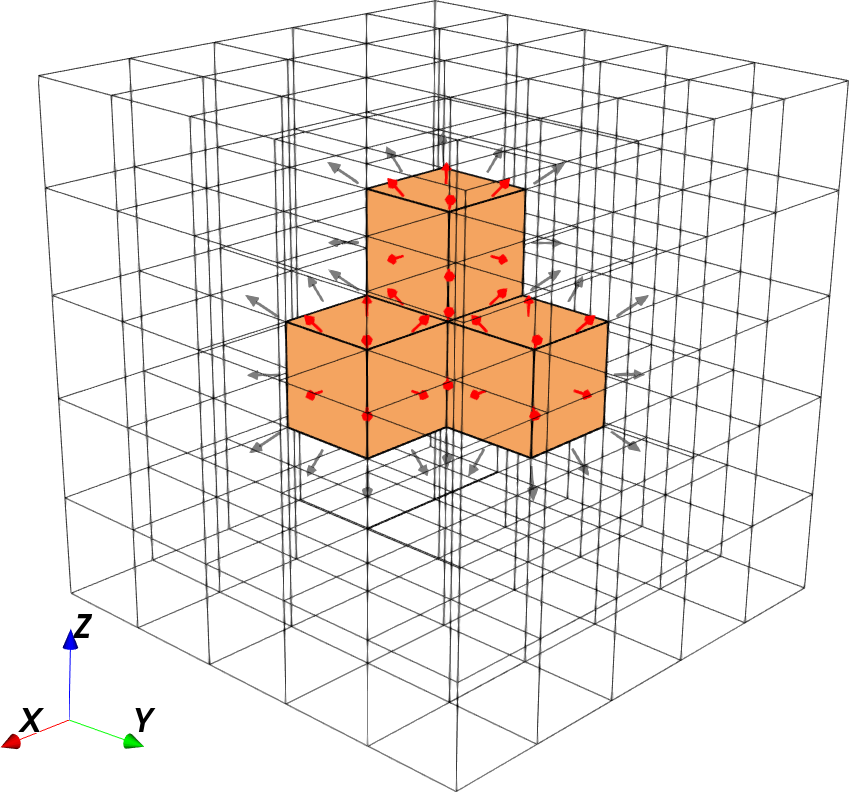}\\
		(b) First sweep: step 8 
	\end{minipage}
	\begin{minipage}[t]{\wdff\linewidth}
		\centering
		\includegraphics[width=0.9\textwidth]{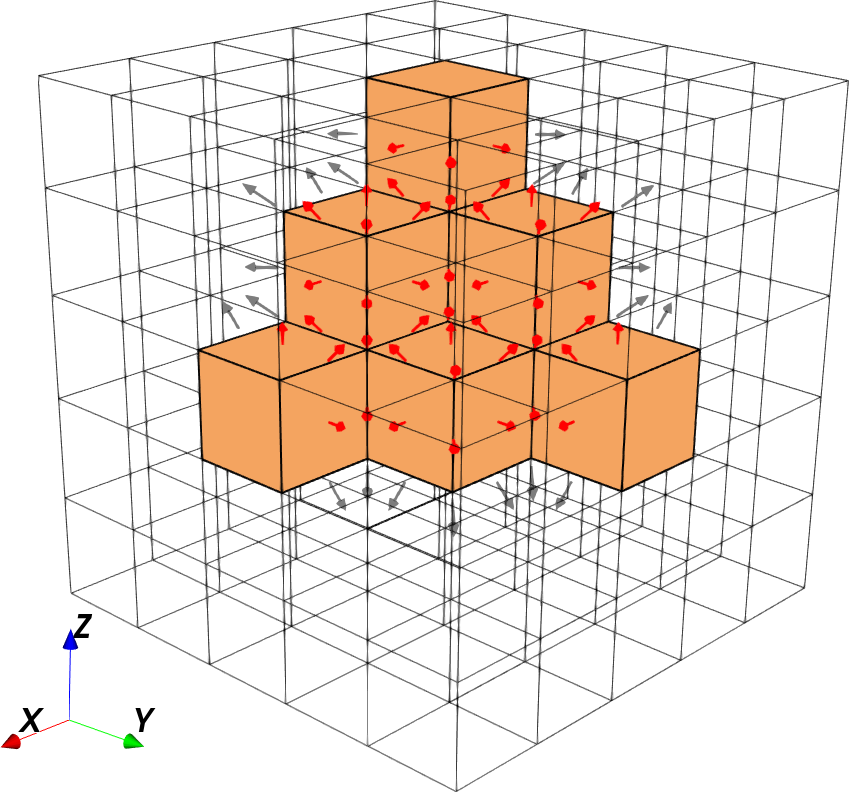}\\
		(c) First sweep: step 9 
	\end{minipage}
	
	\vspace{0.3cm}
	\begin{minipage}[t]{\wdff\linewidth}
		\centering
		\includegraphics[width=0.9\textwidth]{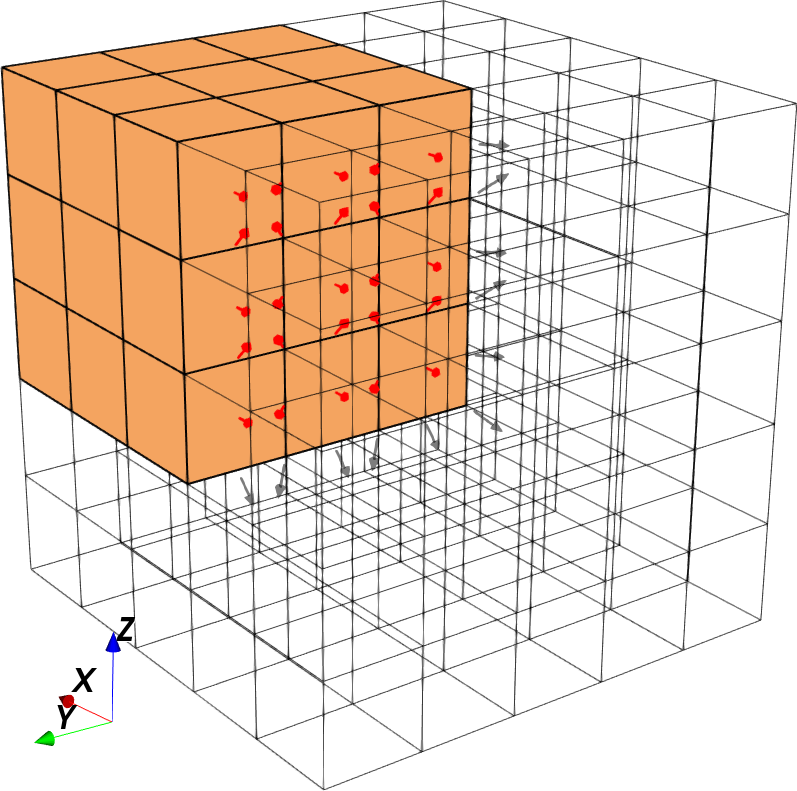}\\
		(d) Before second sweep 
	\end{minipage}
	\begin{minipage}[t]{\wdff\linewidth}
		\centering
		\includegraphics[width=0.9\textwidth]{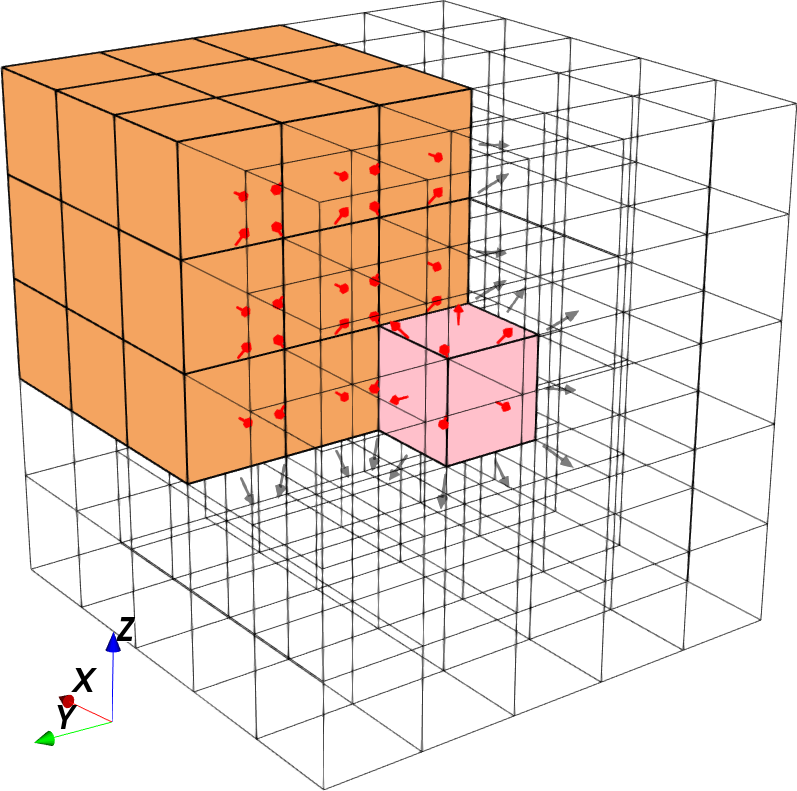}\\
		(e) Second sweep: step 8 
	\end{minipage}
	\begin{minipage}[t]{\wdff\linewidth}
		\centering
		\includegraphics[width=0.9\textwidth]{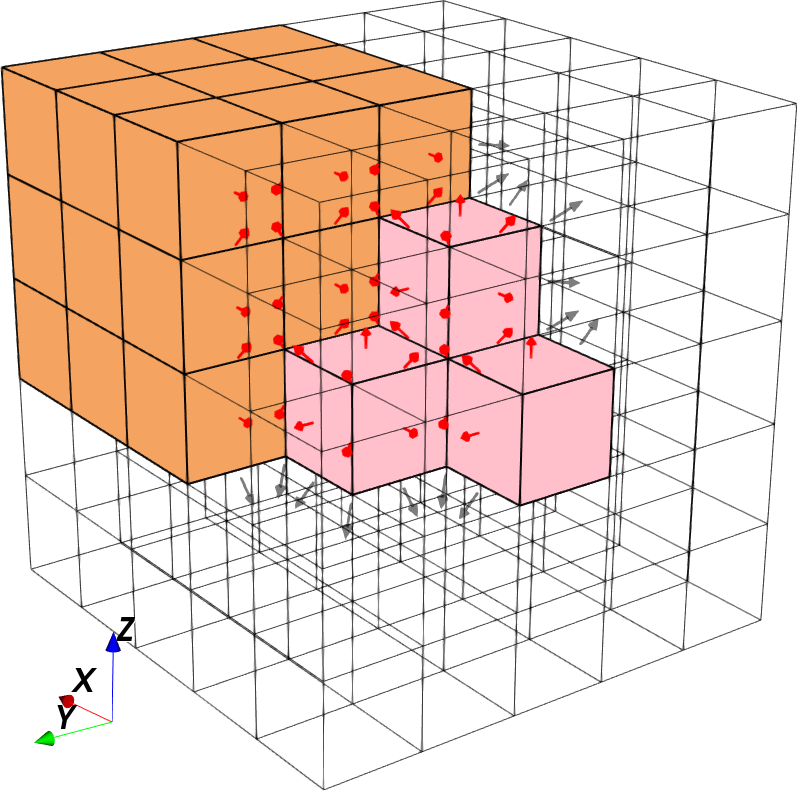}\\
		(f) Second sweep: step 9  
	\end{minipage}
	\caption{The first sweep $(+1,+1,+1)$ and the second sweep $(-1,-1,-1)$ in the diagonal sweeping DDM in $\R^3$. 
          The arrows denote the  transferred sources with their directions, where the red ones are used in the current sweep
          while the black ones are not (due to Rule \ref{rule3d_a}).
          \label{fig:sweep3d_12}}
\end{figure*}

The {\bf first sweep} of direction $(+1,+1,+1)$ is performed to construct the solution in the octant  $ \widetilde{\Omega}^{(+1,+1,+1)} =\Omega_{3,5;3,5;3,5}$,
where the $s$-th step of this sweep handles the group of subdomains $\{\Omega_{i,j,k}\}$ with $(i-1)+(j-1)+(k-1)+1=i+j+k-2=s$. 
In the first $(i_0-1)+(j_0-1)+(k_0-1)=6$ steps, the local source and solutions in the subdomains are all zero.
Then at step $(i_0-1)+(j_0-1)+(k_0-1)+1=7$, the subdomain  problem in $\Omega_{i_0,j_0,k_0}=\Omega_{3,3,3}$ is solved with the source $f_{3,3,3}$, and $3^3-1=26$ transferred sources are generated and passed to its  neighbor subdomains correspondingly as  shown in Figure \ref{fig:sweep3d_12}-(a).
At step $8$, 
the $x$, $y$ and $z$ directional source transfers are applied on 
 $\Omega_{4,3,3}$, $\Omega_{3,4,3}$ and $\Omega_{3,3,4}$ respectively,
 and the problem in each of these subdomains is solved with just one transferred source at step 7 from $\Omega_{3,3,3}$ as the local source,
 i.e., the subdomain problem in $\Omega_{4,3,3}$ is solved with the $(1,0,0)$ directional transferred source,
 the subdomain problem in $\Omega_{3,4,3}$  with the $(0,1,0)$ directional transferred source and  the subdomain problem in $\Omega_{3,3,4}$  with the $(0,0,1)$ directional transferred source. For each of them, 
17 new transferred sources are then generated and passed to its corresponding  neighbor subdomains as  shown in Figure \ref{fig:sweep3d_12}-(b).
At step 9, 
the $x$, $y$ and $z$ directional source transfers are applied on $\Omega_{5,3,3}$, $\Omega_{3,5,3}$ and $\Omega_{3,3,5}$ respectively and the problems in these subdomains are solved just as step 8. Additionally, the $x$-$y$, $y$-$z$ and $x$-$z$ directional source transfers are applied on $\Omega_{4,4,3}$, $\Omega_{3,4,4}$ and $\Omega_{4,3,4}$, the problems in these subdomains are solved with the sum of three transferred sources from their neighbor subdomains respectively, 
e.g., the subdomain problem in $\Omega_{4,4,3}$ is solved with the sum of 
the $(1,0,0)$ directional transferred source from $\Omega_{3,4,3}$ at step 8,  
the $(0,1,0)$ directional transferred source from $\Omega_{4,3,3}$ at step 8, and 
the $(1,1,0)$ directional transferred source from $\Omega_{3,3,3}$ at step 7. For each of the subdomains $\Omega_{4,4,3}$, $\Omega_{3,4,4}$ and $\Omega_{4,3,4}$,
12 new  transferred sources are  generated and  passed to their corresponding  neighbor subdomains as shown in Figure \ref{fig:sweep3d_12}-(c).
At  step 10,  the $x$, $y$, $z$, $x$-$y$, $y$-$z$ and $x$-$z$ directional source transfers are applied just as step 9.  Additionally the $(+1,+1,+1)$ directional source transfer is applied on $\Omega_{4,4,4}$, and this subdomain  problem  is solved with the sum of seven transferred sources from its  neighbor subdomains that are solved in previous steps.
The following steps continue similarly  and after $\Nbx+\Nby+\Nbz-2$ steps,  the solution in octant $\widetilde{\Omega}^{(+1,+1,+1)}$ is successfully constructed. 

Then the {\bf second sweep} with direction $(-1,+1,+1)$ is performed, which aims at constructing the solution in the octant  $\widetilde{\Omega}^{(-1,+1,+1)} =\Omega_{1,2;3,5;3,5}$. The $s$-th step of this sweep handles the group of subdomains $\{\Omega_{i,j,k}\}$ with $(\Nbx-i)+(j-1)+(k-1)+1=4-i+j+k=s$. 
The sweeping solve procedure  is similar to the first sweep as  shown in Figure \ref{fig:sweep3d_12}-(d) to (f), except that 
some transferred sources  from the first sweep  are used due to Rule \ref{rule3d_a}.

%
\def\wdff{0.325}
\begin{figure*}[!ht]	
	\centering
	\begin{minipage}[t]{\wdff\linewidth}
		\centering
		\includegraphics[width=0.9\textwidth]{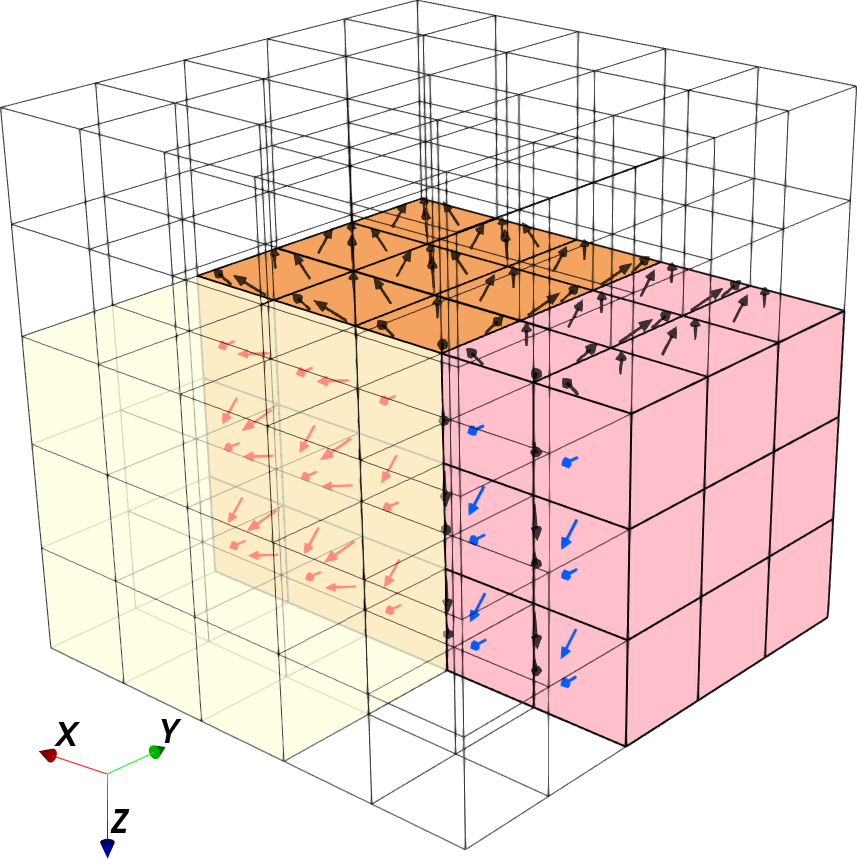}\\
		(a) Third sweep
	\end{minipage}
	\begin{minipage}[t]{\wdff\linewidth}
		\centering
		\includegraphics[width=0.9\textwidth]{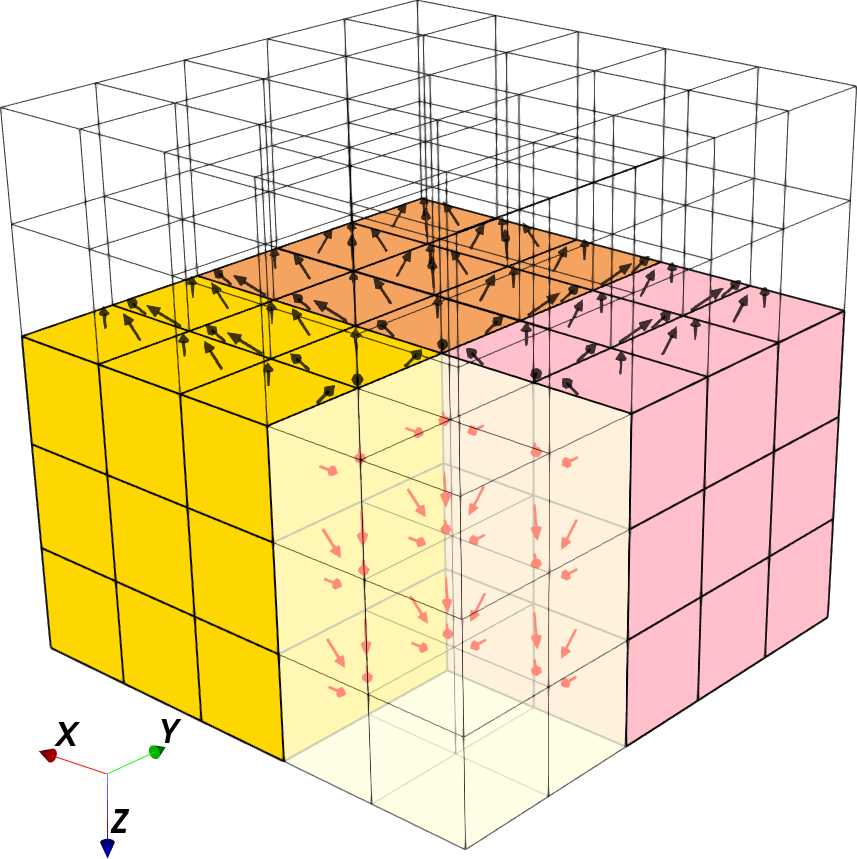}\\
		(b) Fourth sweep
	\end{minipage}
	\begin{minipage}[t]{\wdff\linewidth}
		\centering
		\includegraphics[width=0.9\textwidth]{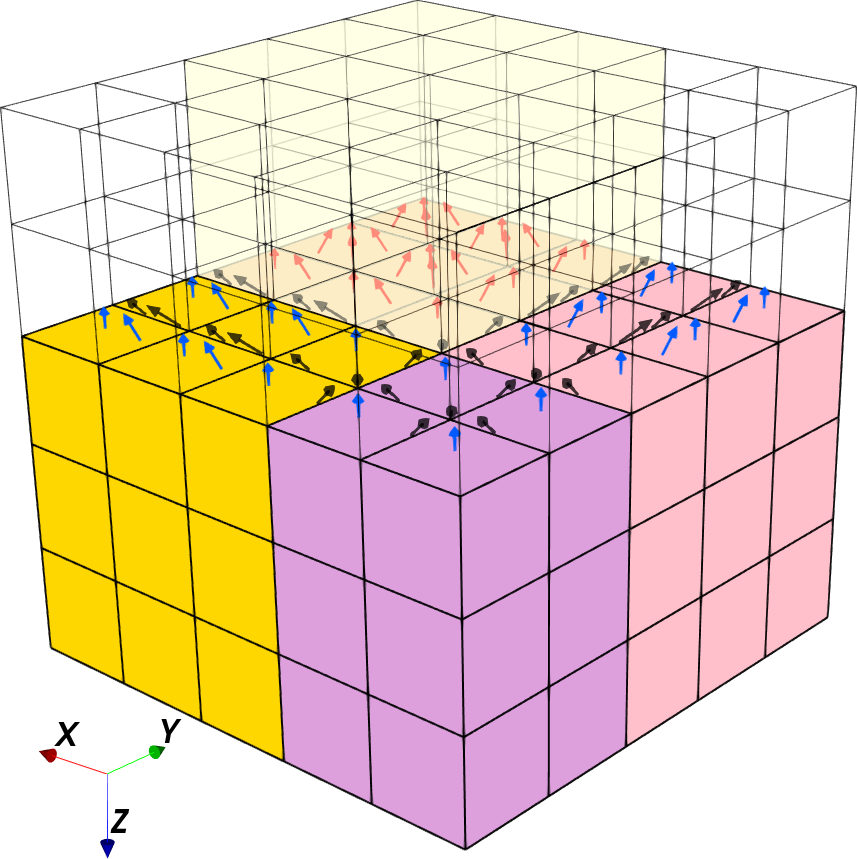}\\
		(c) Fifth sweep
	\end{minipage}
	
	\vspace{0.3cm}
	\begin{minipage}[t]{\wdff\linewidth}
		\centering
		\includegraphics[width=0.9\textwidth]{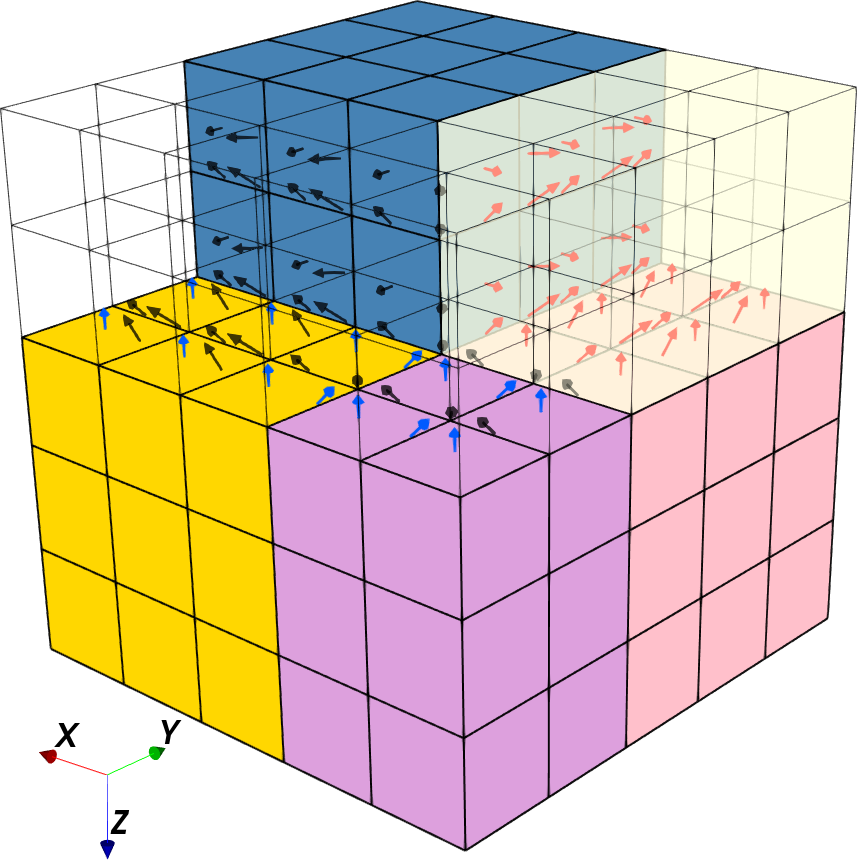}\\
		(d) Sixth sweep
	\end{minipage}
	\begin{minipage}[t]{\wdff\linewidth}
		\centering
		\includegraphics[width=0.9\textwidth]{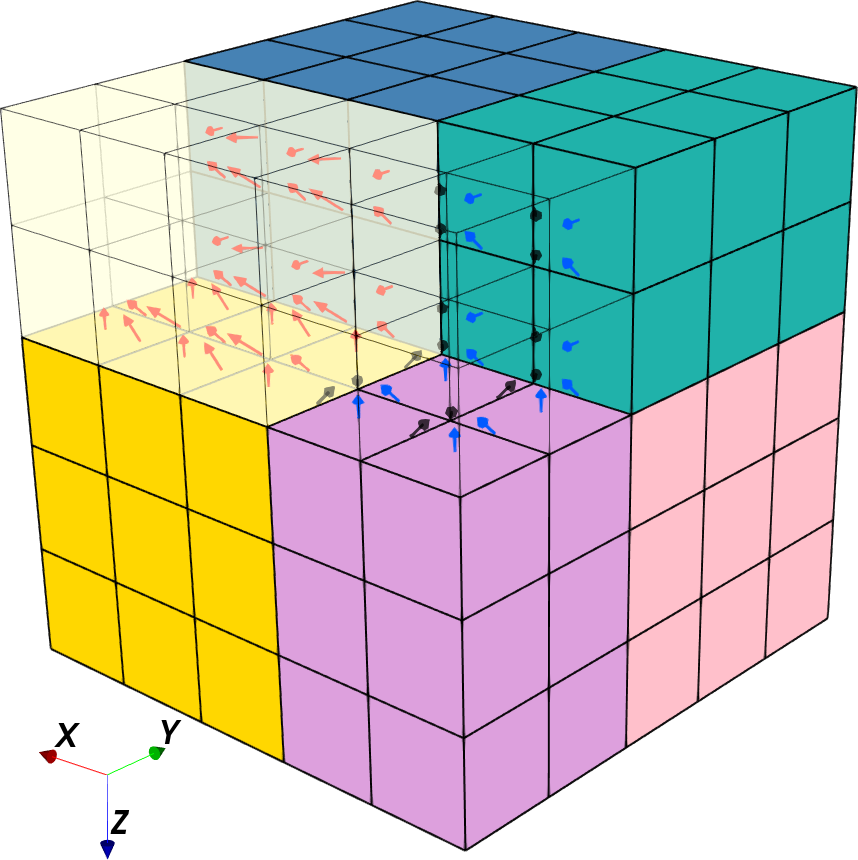}\\
		(e) Seventh sweep
	\end{minipage}
	\begin{minipage}[t]{\wdff\linewidth}
		\centering
		\includegraphics[width=0.9\textwidth]{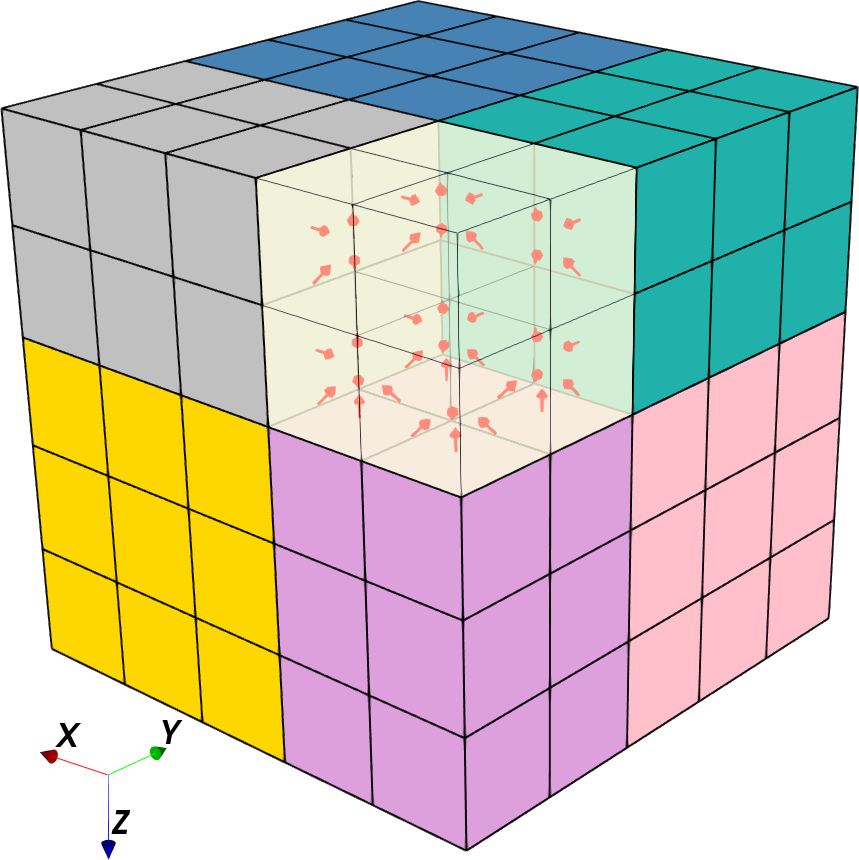}\\
		(f) Eighth sweep
	\end{minipage}
	\caption{At the beginning of the third to eighth sweeps in the diagonal sweeping DDM in $\R^3$,
          where the light yellow transparent region denotes the octant to be solved in the current sweep.
		The arrows denote the  transferred sources with their directions, both the red and blue ones are in the similar direction to the current sweep ( 
		however the blues ones are excluded from being used in the current sweep due to Rule \ref{rule3d_b}), and the black ones are not (the black ones are excluded from being used in the current 
		sweep due to Rule \ref{rule3d_a}).	
          \label{fig:sweep3d_36}}
\end{figure*}

\def\wdff{0.325}
\begin{figure*}[h!]	
	\centering
	\begin{minipage}[t]{\wdff\linewidth}
		\centering
		\includegraphics[width=0.9\textwidth]{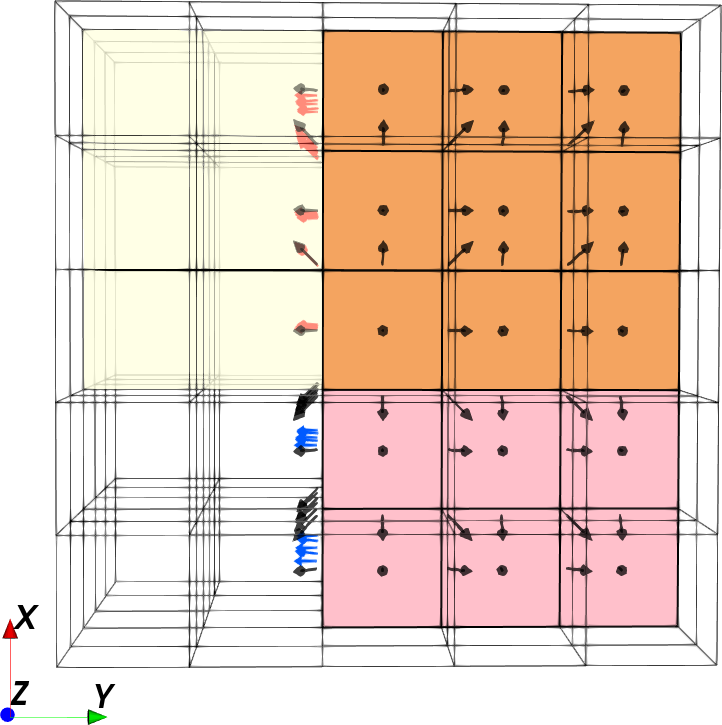}\\
		(a) Third sweep
	\end{minipage}
	\begin{minipage}[t]{\wdff\linewidth}
		\centering
		\includegraphics[width=0.9\textwidth]{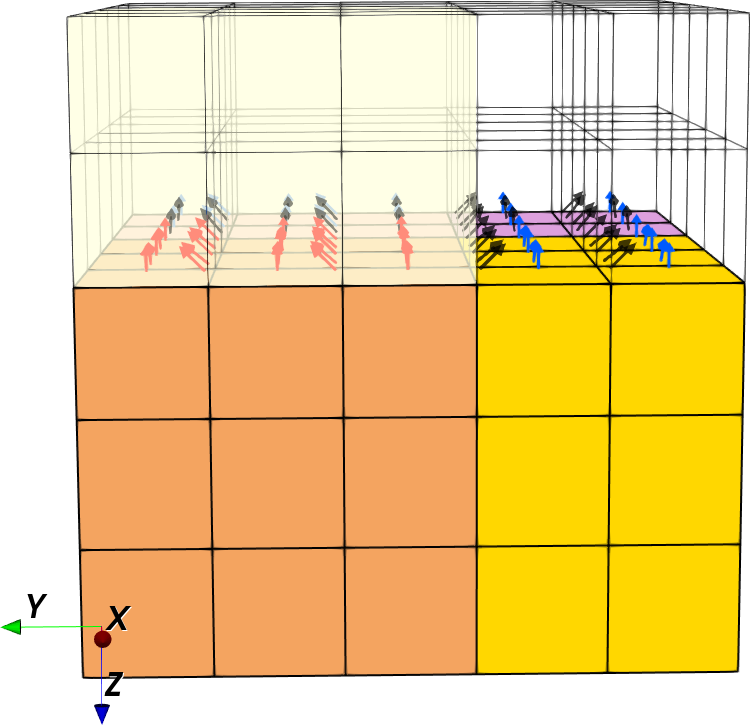}\\
		(b) Fifth sweep
	\end{minipage}
	\begin{minipage}[t]{\wdff\linewidth}
		\centering
		\includegraphics[width=0.9\textwidth]{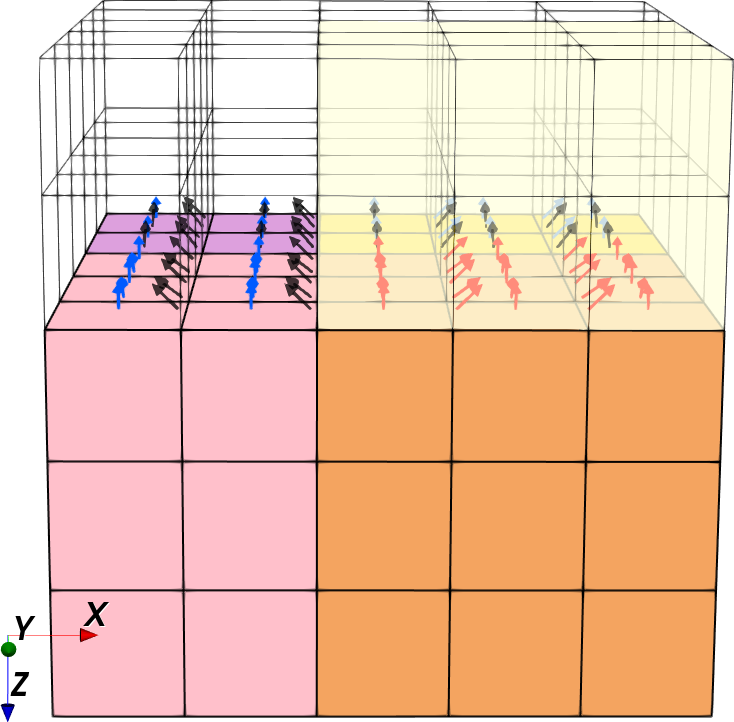}\\
		(c) Fifth sweep
	\end{minipage}
	
	\vspace{0.3cm}
	\begin{minipage}[t]{\wdff\linewidth}
		\centering
		\includegraphics[width=0.9\textwidth]{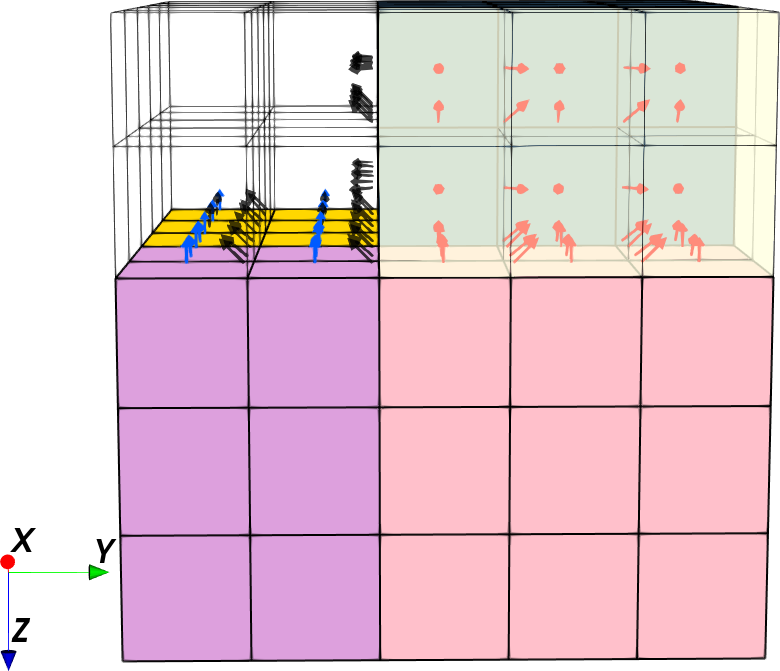}\\
		(d) Sixth sweep
	\end{minipage}
	\begin{minipage}[t]{\wdff\linewidth}
		\centering
		\includegraphics[width=0.9\textwidth]{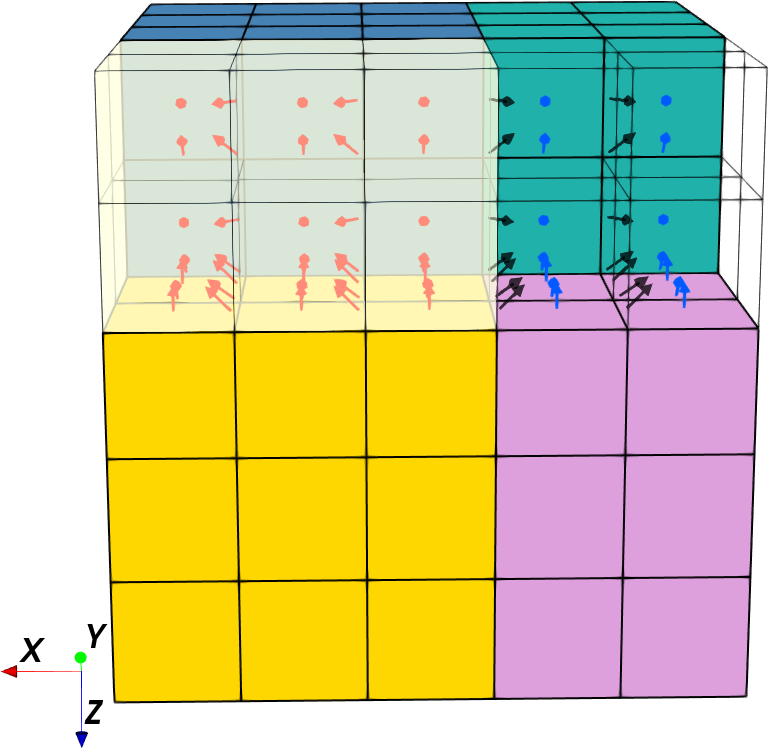}\\
		(e) Seventh sweep
	\end{minipage}
	\begin{minipage}[t]{\wdff\linewidth}
		\centering
		\includegraphics[width=0.9\textwidth]{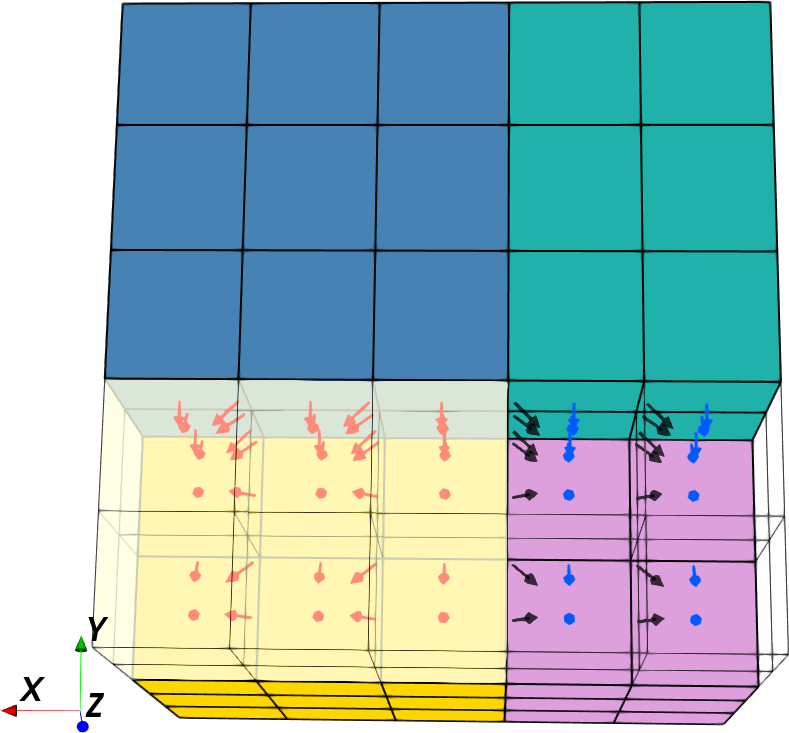}\\
		(f) Seventh sweep
	\end{minipage}
	\caption{
		Check the unused transferred sources  
	at the beginning of certain sweeps in the diagonal sweeping DDM in $\R^3$ using Lemma \ref{lemma:II}, by taking a different view of the third, fifth, sixth and seventh sweep of Figure 
	 \ref{fig:sweep3d_36}.	Note that the origin subdomain is $\Omega_{3,3,3}$.
		 \label{fig:sweep3d_op}}
\end{figure*}

\def\BrownOct{$\widetilde{\Omega}^{(+1,+1,+1)}$ (brown)}
\def\PinkOct{$\widetilde{\Omega}^{(-1,+1,+1)}$ (pink)}
\def\OringeOct{$\widetilde{\Omega}^{(+1,-1,+1)}$ (orange)}
\def\PurpleOct{$\widetilde{\Omega}^{(-1,-1,+1)}$ (purple)}
\def\BlueOct{$\widetilde{\Omega}^{(+1,+1,-1)}$ (blue)}
\def\GreenOct{$\widetilde{\Omega}^{(-1,+1,-1)}$ (green)}
\def\GrayOct{$\widetilde{\Omega}^{(+1,-1,-1)}$ (gray)}

In the {\bf third sweep} with direction $(+1,-1,+1)$, the octant to be solved is $\widetilde{\Omega}^{(+1,-1,+1)} =\Omega_{3,5;1,2;3,5}$, which has one solved face  neighbor  \BrownOct, as  shown in Figure \ref{fig:sweep3d_36}-(a). 
Out of the four faces of the two solved octants, one is shared, two are not in the similar direction, and the remaining one face  with the outer normal $(0,-1,0)$ of \PinkOct 
 { is} to be checked using Lemma \ref{lemma:II} as shown in Figure \ref{fig:sweep3d_op}-(a).
Under the $x$-$y$ plane projection, in the negative $x$-half plane we have the octant to be solved and the origin subdomain, while in the positive $x$-half plane we have the distance-2 solved octant \PinkOct { and} the face with the outer normal $(0,-1,0)$, thus the candidate transferred sources associated with the face are excluded using Lemma \ref{lemma:II}.
Therefore,  the solution in  the octant $\widetilde{\Omega}^{(+1,-1,+1)}$  could be constructed by this sweep.

In the {\bf fourth sweep} with direction $(-1,-1,+1)$, the octant  to be solved is $\widetilde{\Omega}^{(-1,-1,+1)} =\Omega_{1,2;1,2;3,5}$, which has two solved face  neighbor octants, \PinkOct \, and  \OringeOct  \, as  shown in Figure \ref{fig:sweep3d_36}-(b).
Out of the five faces of the three solved octants, two are shared and the rest three ones are not in the similar direction.

In the {\bf fifth sweep} with direction $(+1,+1,-1)$, the octant  to be solved is $\widetilde{\Omega}^{(+1,+1,-1)} =\Omega_{3,5;3,5;1,2}$, which  has one solved face  
neighbor  \BrownOct\, as  shown in Figure \ref{fig:sweep3d_36}-(c). 
Out of the four faces of the solved octants, one is shared, the rest three
are the faces with the outer  normal $(0,0,-1)$ of \PinkOct, \OringeOct \, and \PurpleOct, which are to be checked using Lemma \ref{lemma:II}.
Under the $y$-$z$ plane projection (checked in the positive and negative $z$-half planes), the candidate transferred source associated with the faces of \OringeOct \, and \PurpleOct \, are excluded as  shown in Figure \ref{fig:sweep3d_op}-(b).
Under the $x$-$z$ plane projection (checked in the positive and negative  $x$-half planes), the candidate transferred source associated with the faces of \PinkOct \, and \PurpleOct \, are excluded  as  shown in Figure \ref{fig:sweep3d_op}-(c).

In the {\bf sixth sweep} with direction$(-1,+1,-1)$, the octant to be solved is $\widetilde{\Omega}^{(-1,+1,-1)} =\Omega_{1,2;3,5;1,2}$, which has two solved face  neighbor octants, \BlueOct \, and \PinkOct \, as shown in Figure \ref{fig:sweep3d_36}-(d).
Out of the five faces of the solved octants, two are shared, two are not in the similar direction with the current sweep, and  the rest one face
is with the outer normal $(0,0,-1)$ of \PurpleOct, which is to be checked using Lemma \ref{lemma:II}.
Under the $y$-$z$ plane projection (checked in the positive and negative $y$-half planes), the candidate  transferred sources associated with the face of \PurpleOct \, are excluded as shown in Figure \ref{fig:sweep3d_op}-(d).

In the {\bf seventh sweep} with direction $(+1,-1,-1)$, the octant to be solved is $\widetilde{\Omega}^{(+1,-1,-1)} =\Omega_{3,5;1,2;1,2}$, which   has two solved face  neighbor octants, \BlueOct \, and \OringeOct \, as shown in Figure \ref{fig:sweep3d_36}-(e).
Out of the four faces of  the solved  octant, two are shared, and the rest two 
are the face with the outer normal $(0,0,-1)$ of \PurpleOct \, and the face with the outer normal $(0,-1,0)$ of  \GreenOct, which  are to be checked using Lemma \ref{lemma:II}.
Under the $x$-$z$ plane projection (checked in the positive and negative  $x$-half planes), the candidate transferred sources associated with the face of \PurpleOct \, are excluded as  shown in Figure \ref{fig:sweep3d_op}-(e). 
Under the $x$-$y$ plane projection  (checked in the positive and negative $x$-half planes), the candidate transferred sources associated with the face  \GreenOct \,  are excluded as shown in Figure \ref{fig:sweep3d_op}-(f).

In the {\bf eighth sweep} (also the last sweep) with direction $(-1,-1,-1)$, the octant to be solved is $\widetilde{\Omega}^{(-1,-1,-1)} =\Omega_{1,2;1,2;1,2}$, which has three face  neighbor octants, \PurpleOct, \GreenOct, and \GrayOct, as  shown in Figure \ref{fig:sweep3d_36}-(f). The octant to be solved has only shared faces, edges and vertices with solved octants. After the eight diagonal sweeps  the total solution is finally constructed.

 By extending  the  above process  to the case of $\Nbx \times \Nby\times \Nbz$ domain partition and  general source, we obtain the following result.

\begin{thm}{} \label{thm:sweep3d}
  The DDM solution $u_{\text{DDM}}$   produced by Algorithm \ref{alg:diag3D} is indeed the solution of the problem $P_{\Omega}$ in $\R^3$ in the constant medium case.
\end{thm}

\begin{rem}
The diagonal sweeping DDM could be used as a preconditioner for Krylov subspace methods such as GMRES when solving the discrete system of the Helmholtz equation. 
Let us denote by $n_{\text{iter}}$ the needed number of iterations  for the relative residual to reach certain tolerance. Assume that the size of the subdomain problem is fixed, then 
the complexity of the factorization and solving one subdomain problem becomes $O(1)$, then the total complexity of solving one RHS is $O(N  n_{\text{iter}})$, where $N$ is the size of the discrete system.  Through the numerical experiments presented in Section \ref{numexp}, we demonstrate $n_{\text{iter}} \sim O(\log N)$, thus 
the total complexity of solving one RHS is $O(N \log N)$ by using the proposed diagonal sweeping DDM as the preconditioner. 
\end{rem}

\begin{rem}
   The proposed diagonal sweeping DDM is very suitable for parallel solution of the Helmholtz problem with multiple RHSs 
  in many practical applications, such as seismic imaging and electromagnetic scattering. Taking the full wave inversion (FWI) in seismic imaging as an example, it is a large scale nonlinear optimization problem aimed at solving the subsurface geophysical parameters. In one optimization step, a forward and adjoint wavefield modeling needs to be solved for each shot, which is then used to calculate the gradient of the misfit between the observed and modeled seismograms. There are usually hundreds of shots and all the shots are independent of each other, thus the wavefield modeling problem is indeed a problem with multiple RHSs.
  
We can use the pipeline technique to parallelize the proposed DDM for solving such problem and obtain good scalability.
Suppose that the number of cores to be used is equal to the number of subdomains, and since the subdomains are solved in different orders for different sweeps, in order to keep the solving order of cores the same in the pipeline, each core is assigned to solve one of $2^n$ pre-assigned subdomains in each of the total $2^n$ sweep.  Let us take the 3D case for illustration. 
There are $\Nbx \times \Nby \times \Nbz$ subdomains and cores, and
the subdomains $\Omega_{i',j',k'}$,
where $i' = i, \Nbx+1-i$, $j' = j, \Nby+1-j$, $k'=k, \Nbz+1-k$,
are assigned to the core of rank $((i'-1) \Nby + j'-1)\Nbz + k'$, and the solving order of cores is kept the same as the first sweep.
The pipeline overhead time, which is time that all cores begin to work, is $(\Nbx+\Nby+\Nbz-2) T_{0} $, where $T_{0}$ is the time for solving one subdomain problem. 
Denote the number of RHSs by  $N_{\text{RHS}}$ and assume it is a multiple of
$\Nbx+\Nby+\Nbz-2$,  then the total time cost of solving all RHSs using the pipeline is
\begin{align*}
  \left(\Nbx+\Nby+\Nbz-2 \right)T_0 + 8 n_{\text{iter}} N_{\text{RHS}}  T_0,  
\end{align*}  
thus the average solving time for one RHS is
\begin{equation}\label{eq:ave_diag}
  8 n_{\text{iter}} T_0 + \frac{\Nbx+\Nby+\Nbz-2}{N_{\text{RHS}}}  T_0.
\end{equation}
The idle of the cores at the beginning the pipeline only cause
the average solving time to increase by a neglectable factor, e.g.,
when $N_{\text{RHS}} = 2 (\Nbx+\Nby+\Nbz-2) $ and
$n_{\text{iter}} = 10$ (which is very common in real applications), the
idle of the cores only increases the average solving time by
$0.625 \%$.

With the similar pipeline setup, the recursive sweeping DDM \cite{Liu2015b,Wu2015}
has the average solving time of one RHS as 
\begin{equation} \label{eq:ave_recur}
8 n_{\text{iter}} T_{0} +  \frac{ \Nbx \Nby \Nbz }{N_{\text{RHS}} } T_0,
\end{equation}
under the condition that $N_{\text{RHS}}$ is a multiple of
$\Nbx \Nby \Nbz$, which is very hard to satisfy for real
applications since $\Nbx \Nby \Nbz$ is often larger than $N_{\text{RHS}}$. What is more, comparing the average solving
time \eqref{eq:ave_diag} and \eqref{eq:ave_recur}, our diagonal
sweeping DDM is clearly much more efficient and scalable when
$\Nbx\approx \Nby\approx \Nbz$.

The parallelization of the L-sweeps method \cite{Zepeda2019} 
adopts another way that suits better solving one RHS. Each row of the
subdomains of the checkboard domain decomposition is assigned with
one core, thus only a total of $\Nbx$ cores are used in the
computation and each core handles the corresponding subdomains
during one sweep of the L-sweeps method.
Using $\Nbx$ cores to solve problem of $\Nbx \times \Nby$ subdomains
implies that the problem size per core grows as $\Nby$ increases, thus
the parallelization is not weak scalable, thus not suitable for large
problems with many subdomains.

\end{rem}

\section{Numerical experiments}\label{numexp}

The proposed diagonal sweeping DDM  with source  transfer (Algorithms \ref{alg:diag2D}  in $\R^2$ and \ref{alg:diag3D}  in $\R^3$) will be tested with various experiments
 to demonstrate its  performance for numerically  solving the Helmholtz problem  \eqref{eq:helm}, especially with high frequency.   
First,  the convergence of the proposed method will be tested.  In the  constant medium case, the  discrete  DDM solution is an excellent 
approximation to the continuous  Helmholtz problem, and the total error of the  approximation comes from the spatial numerical discretization and the  truncation of PML.
By choosing appropriately the PML medium parameters, including the PML width and the absorbing parameter $\widehat{\sigma}$,  the total error is expected to be dominated by the  spatial discretization error.
Note that although only the solving orders are different in the diagonal sweeping DDM and the additive overlapping DDM \cite{Leng2019}, the errors coming from the truncation of PML in two methods are not the same, hence the convergence of the diagonal sweeping method still needs to be tested.
Second, the proposed method will be tested as the preconditioner for the GMRES method to solve the global discrete system since the discrete DDM solution  is  an approximation to the discrete Helmholtz problem in general. Many factors affect this approximation, including the truncation of PML, the reflections in the medium and the discretization, etc. 
The performance of the algorithm is tested with constant medium problem, layered media problems, and a more realistic problem (the 2004 BP model), to demonstrate the great potentials of the proposed diagonal sweeping DDM.

In all the numerical experiments,  the Helmholtz equation is discretized on structured meshes with the second-order central finite difference scheme, which is  a five-points  stencil in two dimensions and a seven-points stencil  in three dimensions, respectively.
The diagonal sweeping DDM algorithms are implemented in parallel using Message Passing Interface (MPI) and the local subdomain problems   are solved with the direct solver ``MUMPS'' \cite{MUMPS}. 
The supercomputer ``LSSC-IV'', located in State Key Laboratory of Scientific and Engineering Computing, Chinese Academy of Sciences,  is used for all numerical tests, which has a total of 408 nodes and each node  has two 2.3GHz Xeon Gold 6140 processors (18 cores and 
192G memory).  The number of cores is always chosen to be equal to the number of the subdomains.

\subsection{Convergence tests of the discrete DDM solutions}
In this subsection, the convergence of the diagonal sweeping DDM   is tested for the constant medium Helmholtz problems,
where the wave number and the number of subdomains are both  fixed while the mesh resolution is uniformly increased.

\subsubsection{2D constant medium problem}

In this example, a Helmholtz problem in $\R^2$  with a constant wave number
$\kappa/2\pi = 25$ is solved using  Algorithm \ref{alg:diag2D}.
The computational domain is $B_L=[-L,L]^2$ with $L = 1/2$, and the interior domain without PML is
$B_l=[-l, l]^2$ with $l = 25/56$.
Denote by $q$ the mesh density, which is defined to be  the number of nodes per wave length,
a series of uniformly refined meshes are used, where the mesh density increases
approximately from $22$ to $270$.
A $5 \times 5$ domain partition is used for all  meshes.
The source is chosen as  
\begin{equation*}
f(x_1, x_2) = \frac{16 \kappa^2}{\pi^3}  e^{- (\frac{4 \kappa}{\pi})^2 ((x_1-r_1)^2+(x_2-r_2)^2) },
\end{equation*}
where $(r_1, r_2) = (0.09, 0.268)$, whose support mostly lies in four subdomains,
$\Omega_{3,4}$, $\Omega_{3,5}$, $\Omega_{4,4}$ and $\Omega_{4,5}$.
The results on the errors and convergence rates of the discrete DDM solutions for this 2D problem  are shown in Table \ref{tab:num2D}.
The optimal  convergences (order 2) of the errors measured by both $L^2$ and $H^1$ norms along the refinement of the meshes are  obtained, which also demonstrate that the total errors are indeed dominated by the finite difference discretization errors in this test as expected. 

\begin{table}[!ht]
	\centering
	\begin{tabular}{|c|c|r|c|r|c|}
		\hline
		Mesh     & Local Size  & \mR{$L^2$ Error}     & Conv. & \mR{$H^1$  Error}    & Conv. \\
		Size     & without PML &                      & Rate  &                      & Rate  \\
		\hline\hline
		560$^2$  & 100$^2$     & 3.13$\times 10^{-3}$ &       & 4.89$\times 10^{-1}$ &       \\
		1120$^2$ & 200$^2$     & 7.78$\times 10^{-4}$ & 2.0   & 1.22$\times 10^{-1}$ & 2.0   \\
		2240$^2$ & 400$^2$     & 1.94$\times 10^{-4}$ & 2.0   & 3.04$\times 10^{-2}$ & 2.0   \\
		4480$^2$ & 800$^2$     & 4.86$\times 10^{-5}$ & 2.0   & 6.64$\times 10^{-3}$ & 2.0   \\
		6720$^2$ & 1200$^2$    & 2.17$\times 10^{-5}$ & 2.0   & 3.44$\times 10^{-3}$ & 2.0   \\
		\hline
	\end{tabular}
	\caption{The errors and convergence rates of the numerical solutions obtained by Algorithm \ref{alg:diag2D}  for the 2D constant medium problem.} \label{tab:num2D}
\end{table}

\subsubsection{3D constant medium problem}

Next a Helmholtz problem in $\R^3$ with a constant wave number
$\kappa/2\pi = 10$ is solved using  Algorithm \ref{alg:diag3D}.
The computational domain is $B_L=[-L,L]^3$ with $L = 1/2$, and the interior domain without PML is 
$B_l=[-l,l]^3$ with $l = 3 / 8$.
A series of refined meshes are used, where the mesh density increases
approximately from $8$ to $16$.
A $3 \times 3 \times 3$ domain partition is used for all  meshes.
The source is chosen as  
\begin{equation*}
f(x_1, x_2, x_3) = \frac{64 \kappa^3}{\pi^{9/2}}  e^{- (\frac{4 \kappa}{\pi})^2 ((x_1 -r_1)^2+(x_2 -r_2)^2+(x_3 -r_3)^2 ) },
\end{equation*}
where $(r_1, r_2, r_3) = (0.12, 0.133, 0.125)$, whose support mostly lies in  eight subdomains,
$\Omega_{i,j,k}$'s with $i=2,3$, $j=2,3$, $k=2,3$.
The results on the errors and convergence rates of the discrete DDM solutions for this 3D problem  are reported in Table \ref{tab:num3D}. 
Similar to in the $\R^2$ case, the optimal second order convergences for both $L^2$ and $H^1$ errors are obtained as expected, which demonstrate again that the total errors are indeed dominated by the finite difference discretization errors.

\begin{table}[!ht]
	\centering
	\begin{tabular}{|c|c|r|c|l|c|}
		\hline
		Mesh  & Local Size & \mR{$L^2$ Error}               & Conv. & \mR{$H^1$ Error}            & Conv. \\
		Size& without PML &                     &   Rate    &                      &  Rate     \\
		\hline     \hline                                                                                          
		
		80$^3$  & 20$^3$ & 2.67$\times 10^{-2}$ &     & 1.62$\times 10^{0}$  &     \\
		96$^3$  & 24$^3$ & 1.82$\times 10^{-2}$ & 2.1 & 1.12$\times 10^{0}$  & 2.0 \\
		128$^3$ & 32$^3$ & 1.00$\times 10^{-3}$ & 2.1 & 6.20$\times 10^{-1}$ & 2.0 \\
		160$^3$ & 40$^3$ & 6.49$\times 10^{-3}$ & 2.0 & 4.04$\times 10^{-1}$ & 2.0 \\
		\hline
	\end{tabular}
	\caption{The errors and convergence rates of the numerical solutions obtained by Algorithm \ref{alg:diag3D}  for the 3D constant medium problem.} \label{tab:num3D}
\end{table}

\subsection{Performance tests with the DDM solutions as the preconditioner} 

The DDM solutions for the  constant medium Helmholtz problem
could be used as the preconditioner for solution of the global discrete systems that arise from 
discretization of  the Helmholtz equation. 
In particular, we will demonstrate
the effectiveness and efficiency of such preconditioner to the GMRES solver for both constant and layered media problems.
In each GMRES iteration, one preconditioner solving is performed in which 4 diagonal sweeps 
are carried out with  $\Nbx+\Nby-1$ steps in each step in $\R^2$ or  8 diagonal sweeps with  $\Nbx+\Nby+\Nbz-2$ steps  in each sweep in $\R^3$.
In the following tests, the stopping criterion is set to be that the relative residual reaches a tolerance of  10$^{-6}$.

\subsubsection{2D constant medium problem}

Algorithm \ref{alg:diag2D} as the preconditioner is tested for a constant medium problem on the 
square domain $[0,1]^2$ with different frequencies. 
Four shots located at $(x^s_1,x^s_2) = (1/4,1/4)$, $(1/4,3/4)$, $(3/4,1/4)$ and $(3/4,3/4)$, for $s=1,\ldots,4$, are taken as the source,
and the shape of each shot is an approximated $\delta$ function, for instance,
\begin{equation} \label{eq:four_src_2}
\DD f(x_1, x_2) = \sum\limits_{\substack{i = 1 \ldots N_x\\ j=1 \ldots N_y}}  \sum\limits_{s=1,\ldots,4} \frac{1}{h_1 h_2}\delta(x_1^s - i h_2) \delta (x_2^s - j h_2 ),
\end{equation}
where $h_{1}$ and $h_{2}$ are the grid spacing in $x$ and $y$ directions, respectively.
The size of the subdomain problems without PML layer is fixed to be $300 \times 300$, while the number of subdomains ($\Nbx\times\Nby$) and the frequency simultaneously  increases. The PML layer is of 30 grid points, which is approximately 2.5 wave length. 
The results on the numbers of GMRES iterations (denoted by $n_{\text{iter}}$)  and the running times are shown in Table \ref{tab:iter_const}, {where $T_{\text{it}}$ denotes the total time measured in seconds.}
As we can see,  $n_{\text{iter}}$ grows as the number of the subdomains grows, and roughly, $n_{\text{iter}}$ is proportional to  $\log(\omega)$ or $\log(N)$.
  We mainly focus on the iteration number in this test and leave the pipeline tests for multiple RHSs  to some of later experiments.

\begin{table}[ht!]
	\centering
	\begin{tabular}{|r|c|r|r|r|r|r|}
		\hline
		\mC{Mesh} & \mR{$\Nbx \times \Nby$} & Freq.                & GMRES                  & \mR{$T_{\text{it}}$} \\
                \mC{Size} &                         & \mC{$\omega / 2\pi$} & \mC{$n_{\text{iter}}$} &                      \\
                \hline
                \hline  
		600$^2$   & 2 $\times$ 2   & 55   & 2 & 22   \\
		1200$^2$  & 4 $\times$ 4   & 105  & 2 & 52   \\
		2400$^2$  & 8 $\times$ 8   & 205  & 3 & 130  \\
		4800$^2$  & 16 $\times$ 16 & 405  & 3 & 268  \\
		9600$^2$  & 32 $\times$ 32 & 805  & 4 & 759  \\
		14400$^2$ & 48 $\times$ 48 & 1205 & 5 & 1429 \\
		\hline
	\end{tabular}
	\caption{The performance of Algorithm \ref{alg:diag2D} as the preconditioner  for the 2D constant medium problem
		with the subdomain problem size being fixed.} \label{tab:iter_const}
\end{table}

\subsubsection{2D Layered media problems and discussions}

The layered media problem is of particular interest, since it is the  common case in the reflection seismology. We first demonstrate how the algorithm handles the reflections in the medium.  Let us consider the simplest case of two-layered media problem in $\R^2$, for instance, a region with two media is partitioned into $2 \times 2$ subdomains and the interface of the two media is in the upper-half of the region, say, $\Omega_{1,2}$ and $\Omega_{2,2}$, as  shown in Figure \ref{fig:two_three_layers}-(left).
The wave solution to this problem contains the wave generated by the source in one medium, the reflection in the same medium and the refraction in the other medium.

\begin{figure*}[!ht]	
\centerline{
		\includegraphics[width=0.4\textwidth]{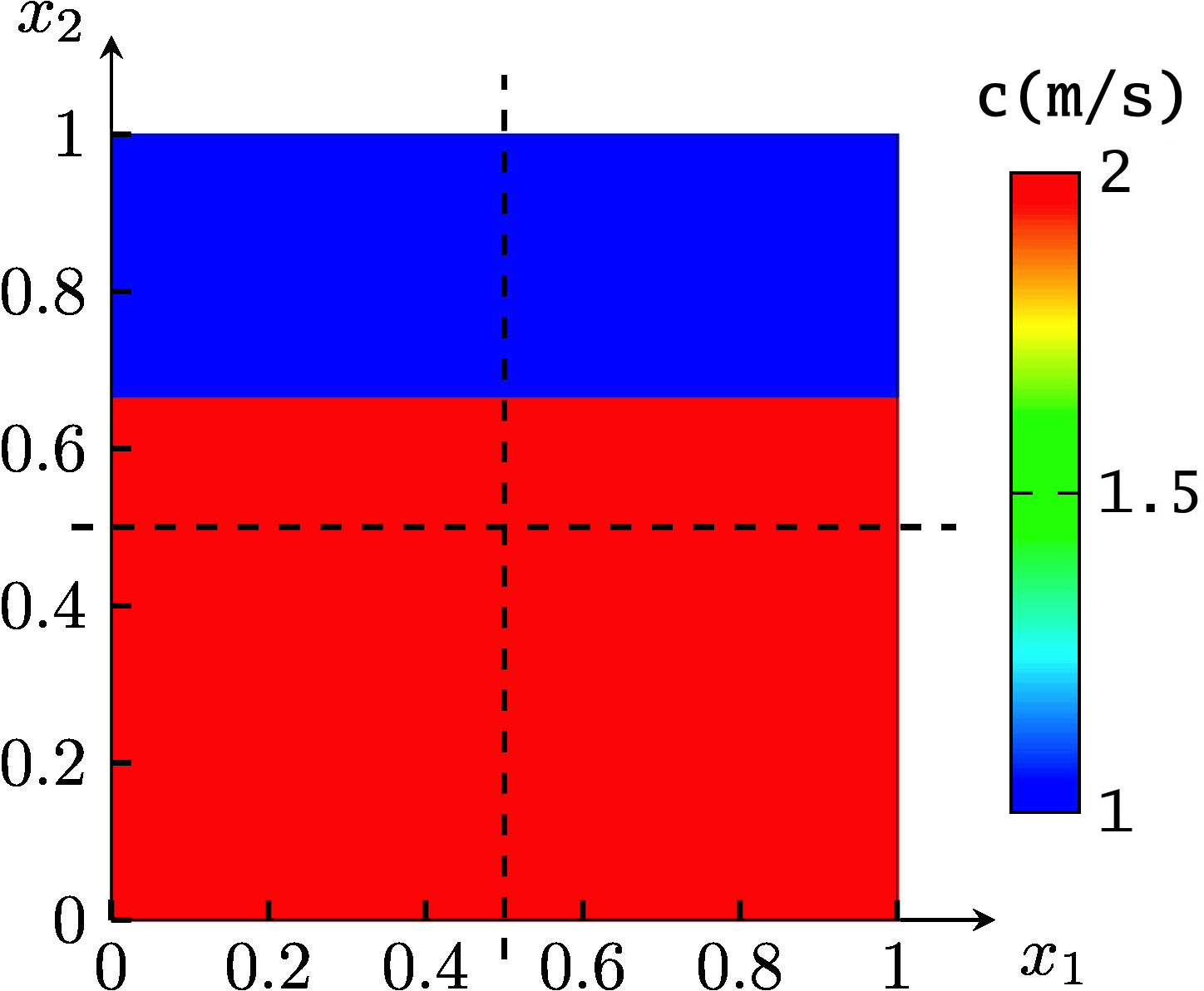}\hspace{0.5cm}
		\includegraphics[width=0.4\textwidth]{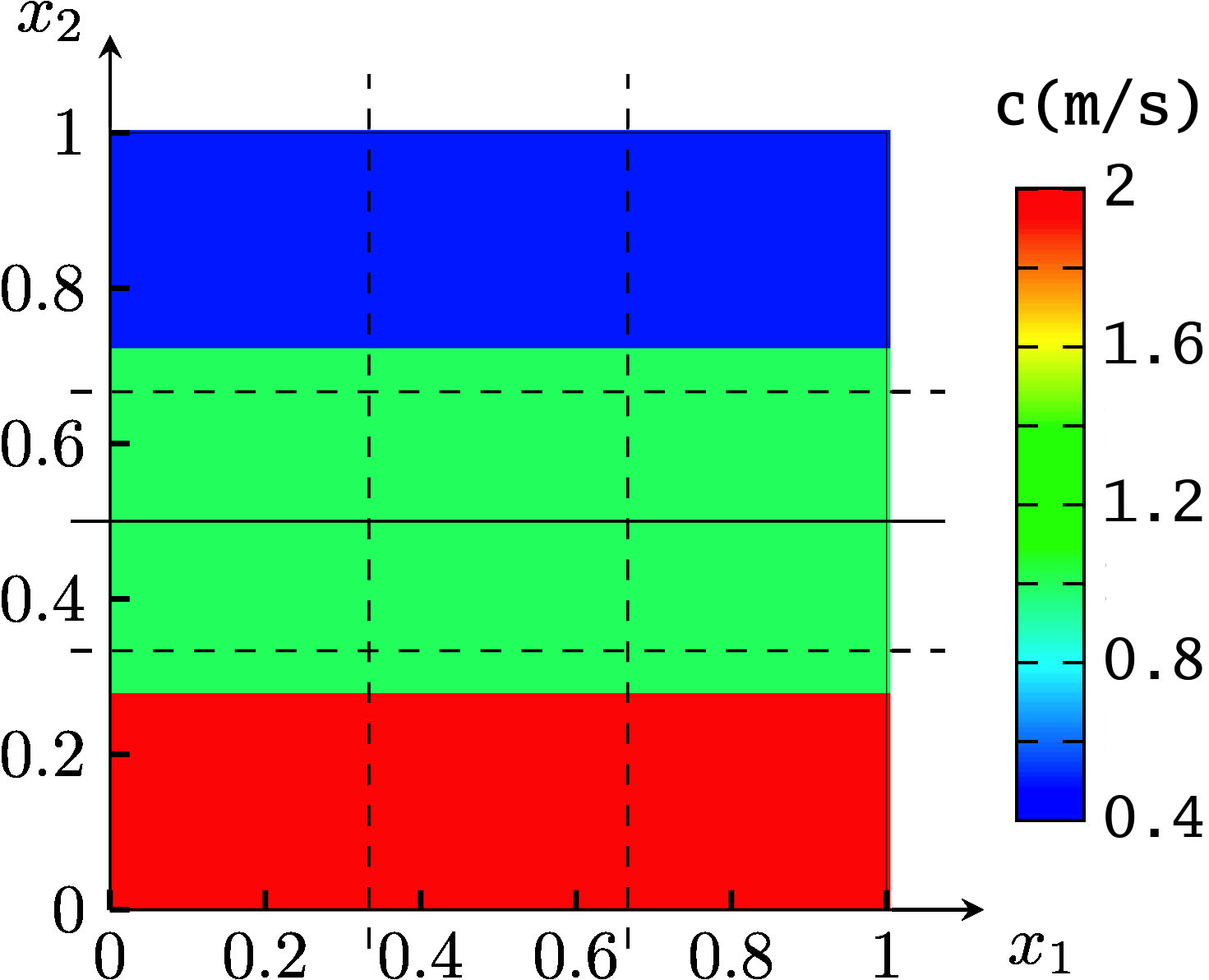}
		}
	\caption{Velocity profiles of the 2D two-layered media  (left) and  three-layered media (right) problems. The $2\times2$ partition is illustrated with the dotted line in the left figure.  The $1\times2$ and $3\times3$ partition are illustrated with the solid and dotted lines respectively in the right figure.  
		\label{fig:two_three_layers}}
\end{figure*}

  In the case of $2\times 2$ partition, the solution always could  be obtained in one iteration (four sweeps) by Algorithm \ref{alg:diag2D} when the source lies in any of the four subdomains.
  Suppose that the source lies in the subdomain $\Omega_{2,1}$, in the first sweep the solution in $\Omega_{2,2}$ is obtained and the reflection in $\Omega_{2,1}$ is missing as  shown  Figure \ref{fig:reflect_2x2}-(a). In the second sweep, the solution in $\Omega_{1,2}$ is obtained, and the reflection in $\Omega_{1,1}$ is missing as  shown in Figure \ref{fig:reflect_2x2}-(b). The third sweep brings the missing reflections to $\Omega_{1,2}$ and the fourth sweep brings the missing reflections to $\Omega_{1,1}$ as  shown in Figures \ref{fig:reflect_2x2}-(c) to (d). Thus the solution with the reflections in the whole domain is obtained in one iteration.

\def\wdff{0.24}
\def\hdff{0.1cm}
\begin{figure*}[!ht]	
	\centering
	\begin{minipage}[t]{\wdff\linewidth}
		\centering
		\includegraphics[width=1\textwidth]{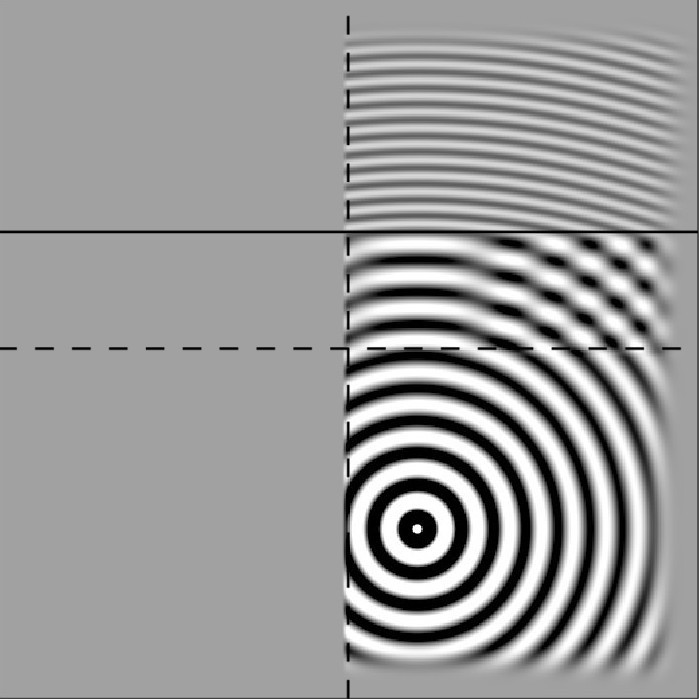}\\
		(a) First sweep $\nearrow$
	\end{minipage}
	\begin{minipage}[t]{\wdff\linewidth}
		\centering
		\includegraphics[width=1\textwidth]{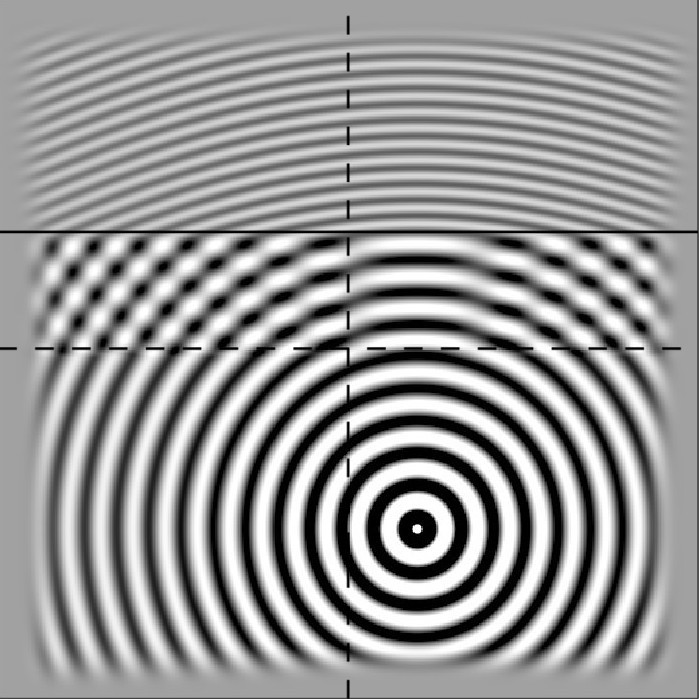}\\
		(b) Second sweep  $\nwarrow$
	\end{minipage}
	\begin{minipage}[t]{\wdff\linewidth}
		\centering
		\includegraphics[width=1\textwidth]{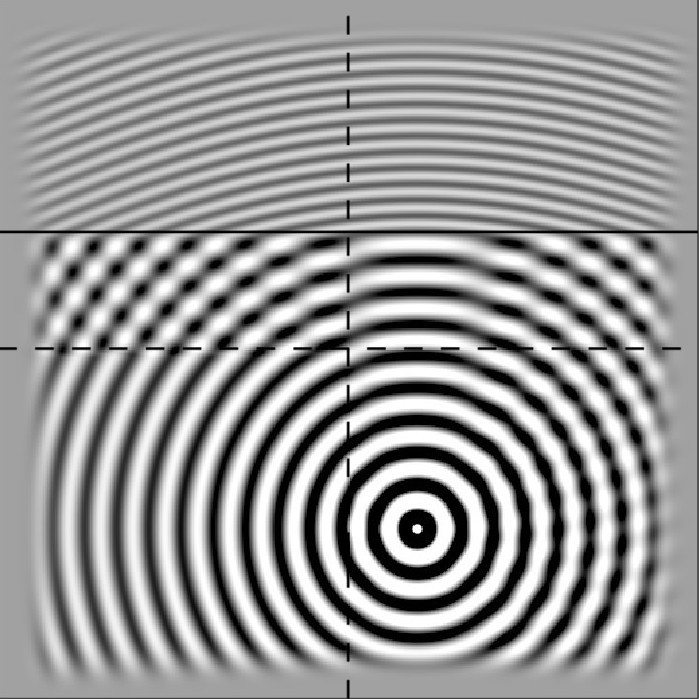}\\
		(c) Third sweep $\searrow$
	\end{minipage}
	\begin{minipage}[t]{\wdff\linewidth}
		\centering
		\includegraphics[width=1\textwidth]{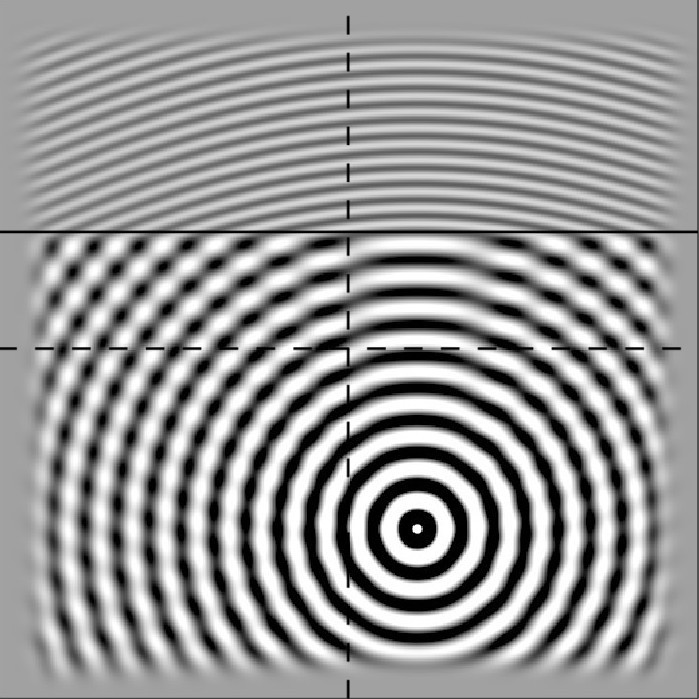}\\
		(d) Fourth sweep $\swarrow$
	\end{minipage}
	\caption{Solution after sweeps for the $2\times 2$ partition with the source lying in $\Omega_{2,1}$. The arrows in the captions denote the sweeping directions. \label{fig:reflect_2x2}}
\end{figure*}

  However, for the general partition $\Nbx \times \Nby$, in the case that the source lies above the interface of the two media, two iterations (eight sweeps) are needed to produce the solution with the reflections in the whole domain by Algorithm  \ref{alg:diag2D}.
Nevertheless, for the general case of multi-layered media, since the reflections are traveling back and forth in the layers, the effect of the source location to the algorithm is expected to be negligible, which will be shown by the next test.

The three-layered media case has been widely used for many DDMs to illustrate that the residual decay rate is controlled by the medium properties  \cite{Leng2015}.
For instance, a square region $[0,1]^2$ with three-layered media is partitioned into two subdomains 
and the subdomain interface lies in the middle layer as  shown in Figure  \ref{fig:two_three_layers}-(right).
Then a series of reflections occur in the middle layer at the upper and lower medium interface during the DDM iterations, just as the time domain wave traveling. The maximum residual decay rate is related to the reflection rates at the medium interfaces. Thus we define the residual decay rate per iteration  of two subdomain partition as the optimal residual decay rate, where two reflections take place   in one iteration  consists an upward and a downward sweeping. A $3\times 3$ subdomain partition is used to test Algorithm \ref{alg:diag2D} for the three-layered media problem, and a shot located at $(1/4\Delta \xi, 1/3 \Delta \eta)$ is used as the source.   As discussed in the two-layered media case, in one iteration of the algorithm, an effective upward sweeping and a downward sweeping are performed, thus the optimal residual decay rate per iteration is expected to be achieved and the results shown in  Figure \ref{fig:3layers} verify that  it is indeed obtained
(note that in order to remove the influence of Krylov space correction, we use the DDM algorithm as an iterative solver rather than a preconditioner  in this test).

\def\wdff{1}
\begin{figure*}[!ht]	
	\centering
		\begin{minipage}[t]{\wdff\linewidth}
			\centering
			\includegraphics[width=0.6\textwidth]{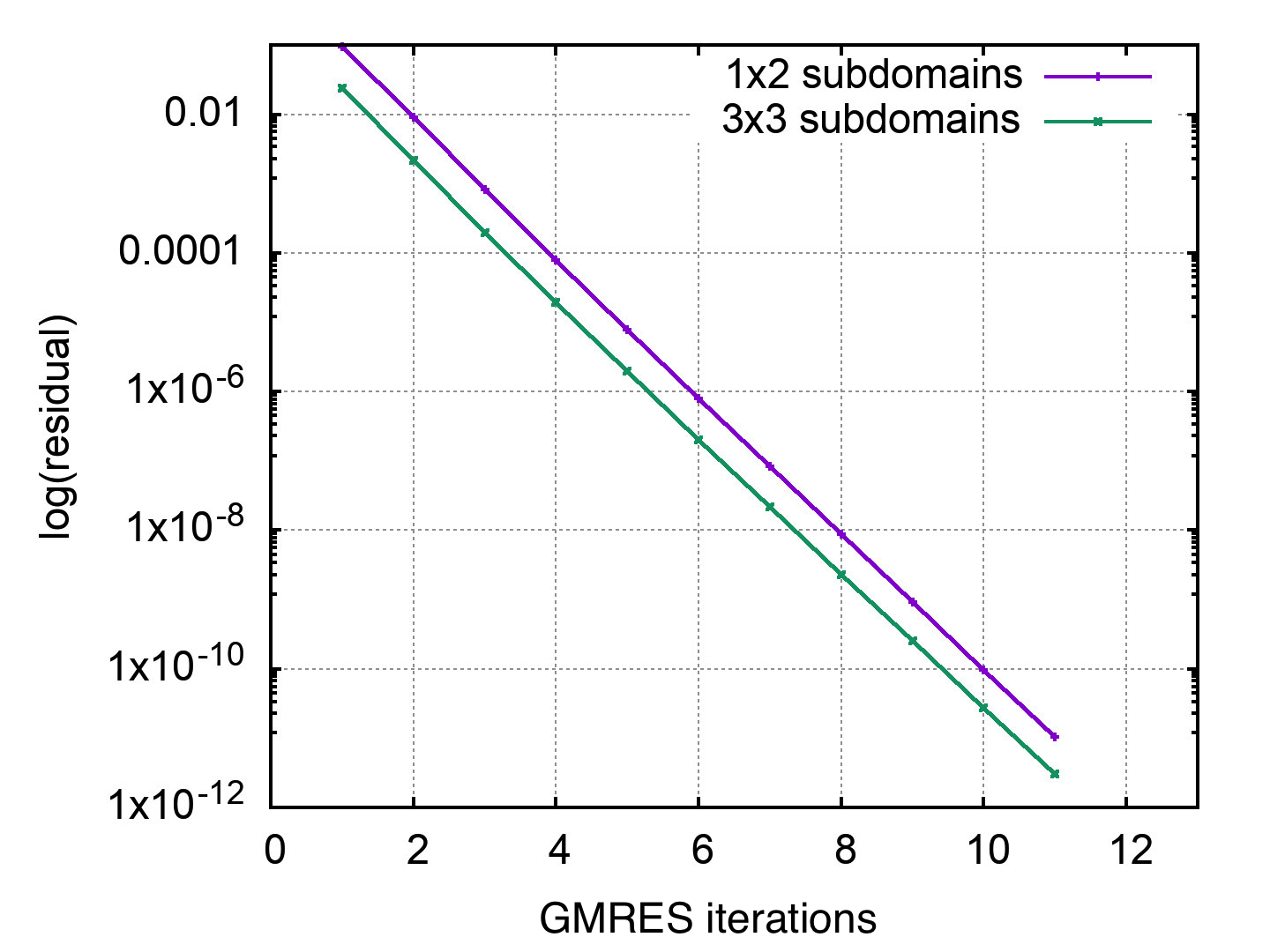}\\
		\end{minipage}
		\vspace{-0.6cm}
	\caption{The residual at each iteration for the 2D three-layered media problem. \label{fig:3layers}}
\end{figure*}

\subsubsection{The BP-2004 model in $\R^2$}

The performance of  Algorithm \ref{alg:diag2D} as the preconditioner is further tested  with  the 2D BP-2004 benchmark model \cite{Billette2004}, which contains a salt body 
and sharp velocity contrasts and has been popularly used for benchmarking reverse time migration.
The left side of the model is used in the test, which 
is $[0,24]\times[-12,0]$ measured in kilometers based on a geological cross section through the Western Gulf of Mexico, 
and the velocity varies from 1000 m/s to 5000 m/s, as shown in Figure \ref{fig:seg04_vel}.
  For a $\Nbx \times \Nby$ domain partition,
 a total number of $N_{\text{RHS}} = 2 (\Nbx + \Nby -1)$ shots are tested as sources, 
 and each of the shots is located at $(\hat{x}, -\frac{1}{4} \Delta \eta)$, 
 where $\hat{x}$ is a random position in the x-direction range of the domain without PML.
The size of the subdomain problems is fixed to be $200 \times 200$, while the number of subdomains ($\Nbx \times \Nby$) and frequency simultaneously increase. The PML layer is of 30 grid points.  The pipeline technique is implemented and used for handling this multiple RHSs problem as discussed in Remark 2.
An approximate solution to the problem of one of the random shots with the angular frequency $\omega/2\pi = 8.56$ is presented in Figure \ref{fig:seg04_sol}.
The results on the numbers of GMRES iterations  and the running times are reported in Table \ref{tab:iter_seg04}.
It is easy to see that
the number of GMRES iterations again grows roughly  proportional to  $\log(\omega)$ and  so does the average solving time $\Tave := \frac{T_{\text{it}}}{N_{\text{RHS}}}$, which demonstrates excellent efficiency and parallel scalability of  the proposed diagonal sweeping DDM with the pipeline processing.

\begin{figure}[!ht]
	\centerline{
		\includegraphics[width=.82\textwidth]{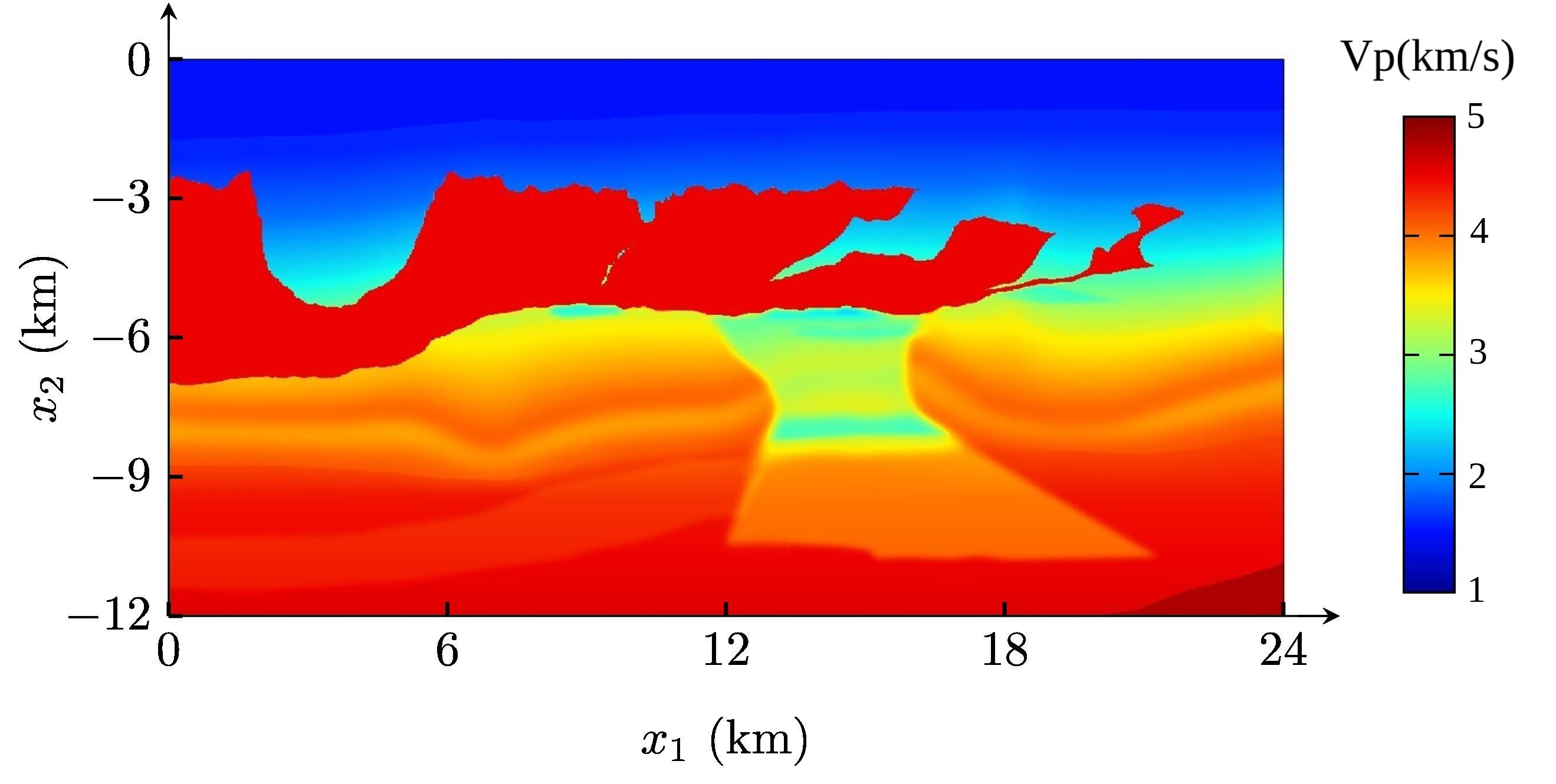} 
	}
		\vspace{-0.4cm}
	\caption{The velocity profile of the BP-2004 model.
	}
	\label{fig:seg04_vel}
\end{figure}

\begin{figure}[!ht]
	\centerline{
		\includegraphics[width=.75\textwidth]{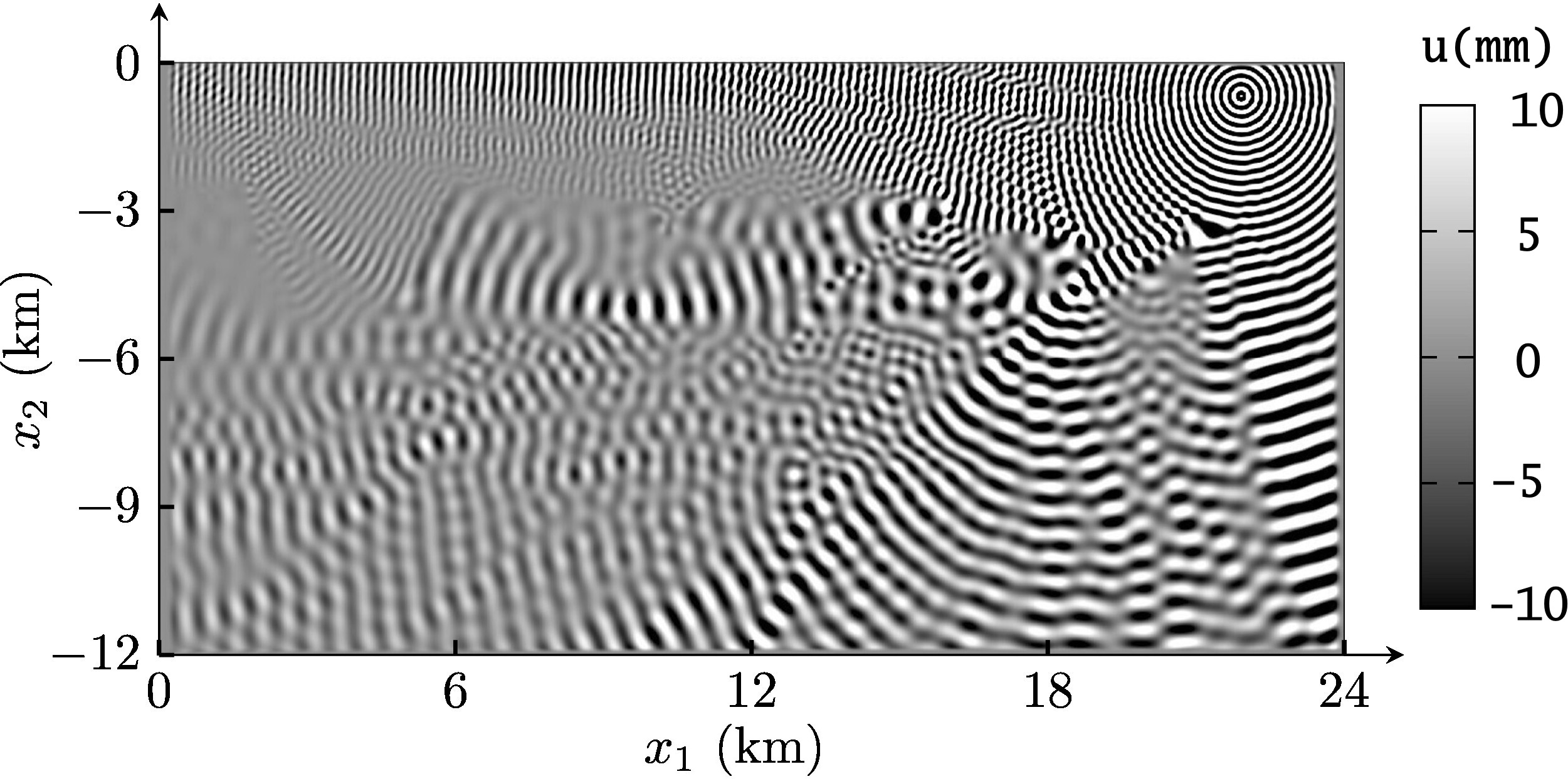} 
	}
	\vspace{-0.1cm}
	\caption{The real part of the approximate solution to the problem of the shot located at (21.931 km, -0.6944 km) with angular frequency $\omega/2\pi = 8.56$ on the mesh of size $1600^2$  in the BP-2004 model.}
	\label{fig:seg04_sol}
\end{figure}

\begin{table}[!ht]
	\centering
	\begin{tabular}{|r|c|r|r|r|r|r|r|r|}
		\hline
		Mesh      & \mR{$\Nbx \times \Nby$} & Freq.                & \mR{$N_{\text{RHS}}$} & GMRES                  & \mR{ $T_{\text{it}}$} & \mR{$\Tave $} \\
		\mC{Size} &                         & \mC{$\omega / 2\pi$} &                       & \mC{$n_{\text{iter}}$} &                       &               \\
                \hline
                \hline  
                400$^2$  & 2 $\times$ 2   & 2.37  & 6   & 6  & 71.4 & 11.9\\
		800$^2$  & 4 $\times$ 4   & 4.43  & 14  & 7  & 221  & 15.8 \\
		1600$^2$ & 8 $\times$ 8   & 8.56  & 30  & 8  & 478  & 15.9 \\
                3200$^2$ & 16 $\times$ 16 & 16.82 & 62  & 10 & 1276 & 20.6 \\
                6400$^2$ & 32 $\times$ 32 & 33.33 & 126 & 11 & 2832 & 22.5 \\
		\hline
	\end{tabular}
	\caption{The performance of Algorithm \ref{alg:diag2D} as the preconditioner  for the BP-2004 model 
		with the subdomain problem size being fixed.} \label{tab:iter_seg04}
\end{table}

\subsubsection{3D layered media problem} 

Algorithm  \ref{alg:diag3D} as the  preconditioner is tested for a 3D  five-layered  media problem on the cuboidal domain $[0,1]^3$ with different frequencies, see Figure \ref{fig:layer-3d_vel}-(left).
  For a $\Nbx \times \Nby \times \Nbz$ domain partition,
  a total number of $N_{\text{RHS}} = 2 (\Nbx + \Nby + \Nbz -2)$ shots are tested as sources,
  and each of the shot is located at $(\hat{x}, \hat{y}, 0.85)$, 
  where $(\hat{x}, \hat{y})$ is a random position within the range of  $(0.15, 0.85)\times(0.15, 0.85)$. The pipeline technique is again  used for handling this multiple RHSs problem.
The size of the subdomain problems  without PML layer is fixed to be $30^3$,
and the mesh density is kept to be $8$,
while the number of subdomains ($\Nbx\times\Nby\times\Nbz$) and the frequency simultaneously  increase. The PML layer is of 12 grid points, which is approximately 1.5 wave length.
 An approximate solution to the problem of one of the random shots with the angular frequency $\omega/2\pi = 27.40$ is presented in Figure  \ref{fig:layer-3d_vel}-(right).
The results on the numbers of GMRES iterations  and the running times are reported in Table \ref{tab:iter_R3}.
As we can see,  $n_{\text{iter}}$ grows as the number of the subdomains grows, and roughly, $n_{\text{iter}}$ is again proportional to  $\log(\omega)$ or $\log(N)$, and so does the average solving time $\Tave$, which again show that the proposed diagonal sweeping DDM 
is very efficient and scalable when combined with the pipeline processing.

\begin{figure}[!ht]
	\centerline{
          \includegraphics[width=.45\textwidth]{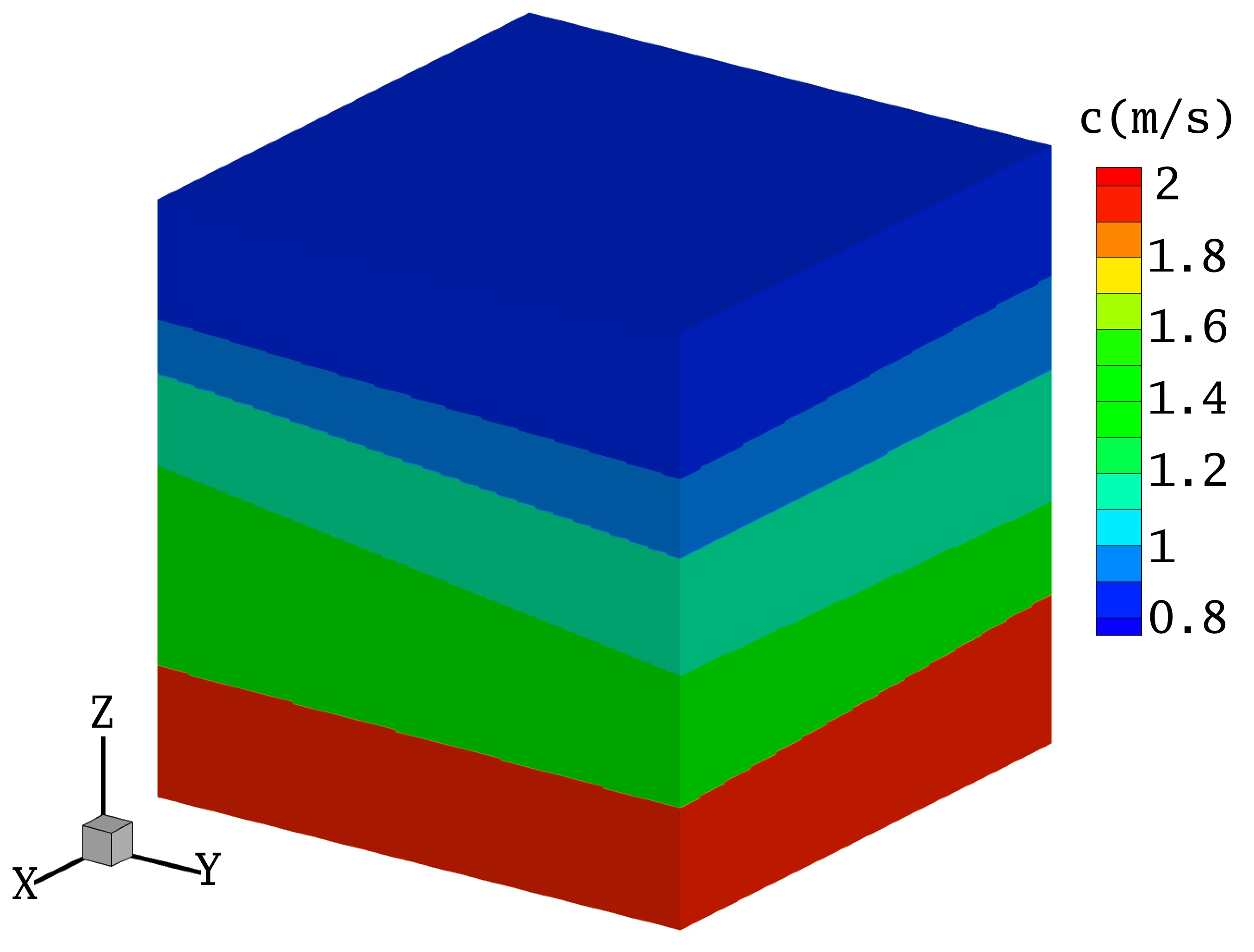}
          \includegraphics[width=.45\textwidth]{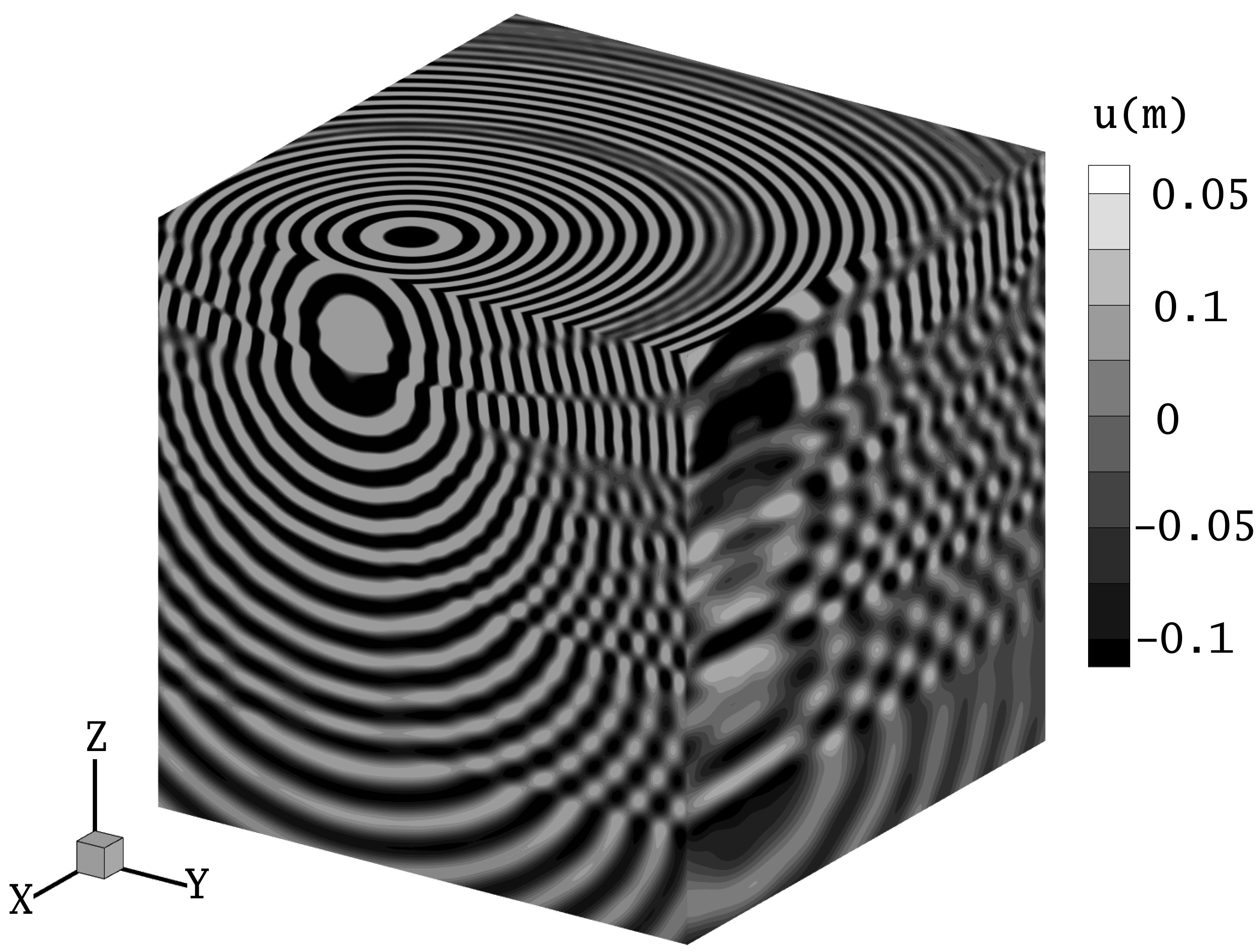} 
	}
	\caption{The velocity profile of the 3D layered media problem (left) and the real part of the approximate solution (right) to the problem of the shot located at (0.811, 0.383, 0.85) with angular frequency $\omega/2\pi = 27.40$ on the mesh of size $240^3$.
	}
	\label{fig:layer-3d_vel}
\end{figure}

\begin{table}[!ht]
	\centering
	\begin{tabular}{|r|c|r|r|r|r|r|r|}
          \hline
          Mesh      & \mR{$\Nbx \times \Nby$} & Freq.                & \mR{$N_{\text{RHS}}$} & GMRES                  & \mR{ $T_{\text{it}}$} & \mR{$\Tave $} \\
          \mC{Size} &                         & \mC{$\omega / 2\pi$} &                       & \mC{$n_{\text{iter}}$} &                       &               \\
          \hline  
          \hline  
          60$^3$  & $2\times 2\times 2$    & 8.65  & 8  & 4 & 290  & 36.3 \\ 
          120$^3$ & $4\times 4\times 4$    & 14.90 & 20 & 5 & 862  & 43.1 \\   
          180$^3$ & $6\times 6\times 6$    & 21.15 & 32 & 5 & 1530 & 47.8 \\  
          240$^3$ & $8\times 8\times 8$    & 27.40 & 44 & 6 & 2469 & 56.1 \\  
          300$^3$ & $10\times 10\times 10$ & 33.65 & 56 & 6 & 3232 & 57.7 \\      
          \hline
	\end{tabular}
	\caption{The performance of Algorithm \ref{alg:diag3D} as the preconditioner   for  the 3D layered media problem 
		with the subdomain problem size being fixed.} \label{tab:iter_R3}
\end{table}

\section{Conclusions}

In this paper, we have developed a diagonal sweeping domain decomposition method with source transfer for solving the high-frequency Helmholtz equation in $\R^n$. Through careful
analysis and extensive numerical experiments, we demonstrate the effectiveness and efficiency of the proposed method as a direct solver or a preconditioner for Krylov subspace methods. 
Comparing to the L-sweeps method \cite{Zepeda2019} with trace transfer, the proposed method with source transfer reduces from all directional sweeps of total $3^n-1$ to only diagonal sweeps of total $2^n$. Furthermore, the proposed method can handle the reflections in the medium in a more proper way. Due to the close relation between source transfer and trace transfer, the proposed method could be naturally extended to the polarized  trace approach with some modifications, and the differences of the resulted  DDMs caused by different  transfer methods and their performance comparisons are currently under our study. At the same time,
the application of the proposed diagonal sweeping DDM to 3D seismic imaging is  another main focus of our future research, the parallel frequency domain solver based on the proposed method will be optimized in several ways  including pipeline setup,  domain decomposition strategy and sparse direct solver, to challenge  the popularly used  time domain solvers in term of  computational cost. In addition, the extension of the proposed method  to the frequency domain wave equations, e.g. electromagnetic and elastic equations, is also worthy of further investigation. 

\section*{Acknowledgements}
W. Leng's research  is partially supported by
National Natural Science Foundation of China under grant number 11501553 and National Center for
Mathematics and Interdisciplinary Sciences of Chinese Academy of Sciences. L. Ju's research is partially supported by US National Science Foundation under grant number DMS-1818438.


\end{document}
